\documentclass[10pt,a4paper]{amsart}
\usepackage{verbatim}

\usepackage{geometry} 
\newgeometry{tmargin=2.4cm, bmargin=3cm, lmargin=2.2cm, rmargin=2.2cm}

\usepackage[T1]{fontenc}
\usepackage{graphicx}
\usepackage{enumerate}
\usepackage{amsmath,amsfonts,amssymb}
\usepackage{mathtools}
\usepackage{ulem}
\usepackage{color}
\usepackage{cite}
\definecolor{citation}{rgb}{0.2,0.58,0.2} 
\definecolor{formula}{rgb}{0.1,0.2,0.6}
\definecolor{url}{rgb}{0.3,0,0.5}
\usepackage{pgf,tikz}
\usepackage{mathrsfs}
\usepackage{fancyhdr}
\usepackage{dsfont}
\usepackage{esint}
\usepackage{hyperref}
\usepackage{aligned-overset}

\DeclareMathOperator*{\osc}{osc}
\DeclareMathOperator*{\tail}{Tail}

\definecolor{darkgreen}{rgb}{0.00, 0.50, 0.00}

\newcommand{\nocontentsline}[3]{}
\newcommand{\tocless}[2]{\bgroup\let\addcontentsline=\nocontentsline#1{#2}\egroup}

\title[Riesz potential estimates for mixed problems]{Riesz potential estimates\\ for mixed local-nonlocal problems with measure data}

\author{Iwona Chlebicka}\address{Iwona Chlebicka \\
Institute of Applied Mathematics and Mechanics, University of Warsaw \\ ul. Banacha 2, 02-097 Warsaw, Poland\\  \texttt{e-mail: i.chlebicka@mimuw.edu.pl}} 

\author{Kyeong Song}\address{Kyeong Song\\ School of Mathematics, Korea Institute for Advanced Study \\ Seoul 02455, Republic of Korea \\ \texttt{e-mail: kyeongsong@kias.re.kr}}

\author{Yeonghun Youn}\address{Yeonghun Youn\\ Department of Mathematics,
Yeungnam University, Gyeongsan 38541, Republic of Korea\\ \texttt{e-mail: yeonghunyoun@yu.ac.kr}}

\author{Anna Zatorska-Goldstein}\address{Anna Zatorska-Goldstein\\
Institute of Applied Mathematics and Mechanics, University of Warsaw \\ ul. Banacha 2, 02-097 Warsaw, Poland\\ \texttt{e-mail: azator@mimuw.edu.pl}}

\date{}
\begin{document}

\thanks{
\vspace{1mm}}
\maketitle \sloppy

\thispagestyle{empty}

\belowdisplayskip=18pt plus 6pt minus 12pt \abovedisplayskip=18pt
plus 6pt minus 12pt\
\parskip 4pt plus 1pt
\parindent 0pt

\newcommand{\ic}[1]{\textcolor{teal}{#1}}

\def\tens#1{\pmb{\mathsf{#1}}}
\def\mean#1{\mathchoice%
          {\mathop{\kern 0.2em\vrule width 0.6em height 0.69678ex depth -0.58065ex
                  \kern -0.8em \intop}\nolimits_{\kern -0.4em#1}}%
          {\mathop{\kern 0.1em\vrule width 0.5em height 0.69678ex depth -0.60387ex
                  \kern -0.6em \intop}\nolimits_{#1}}%
          {\mathop{\kern 0.1em\vrule width 0.5em height 0.69678ex
              depth -0.60387ex
                  \kern -0.6em \intop}\nolimits_{#1}}%
          {\mathop{\kern 0.1em\vrule width 0.5em height 0.69678ex depth -0.60387ex
                  \kern -0.6em \intop}\nolimits_{#1}}}
          
\newcommand{\dv}{{\rm div}}
\def\aI{\texttt{(a1)}}
\def\aII{\texttt{(a2)}}
\newcommand{\opA}{{\mathcal{ A}}}
\newcommand{\opAo}{{\opA_o}}
\newcommand{\bopA}{{\bar{\opA}}}
\newcommand{\wt}{\widetilde}
\newcommand{\ve}{\varepsilon}
\newcommand{\vp}{\varphi}
\newcommand{\vt}{\vartheta}
\newcommand{\vr}{\varrho}
\newcommand{\vk}{\varkappa}
\newcommand{\pa}{\partial}
\newcommand{\cW}{{\mathcal{W}}}
\newcommand{\supp}{{\rm supp}}
\newcommand{\bk}[1]{{B^{#1}}}
\newcommand{\loc}{{\mathrm{loc}}}
\newcommand{\mfs}{{\mathfrak{s}}}

\def\R{{\mathbb{R}}}
\def\N{{\mathbb{N}}}
\def\rp{{[0,\infty)}}
\def\r{{\mathbb{R}}}
\def\n{{\mathbb{N}}}
\def\l{{\mathbf{l}}}
\def\bu{{\bar{u}}}
\def\bg{{\bar{g}}}
\def\bG{{\bar{G}}}
\def\ba{{\bar{a}}}
\def\bv{{\bar{v}}}
\def\wtgamma{{\wt\gamma}} 
\def\calV{{\mathcal{V}}} 
\def\sa{{s_{\rm app}}}
\def\cL{{\mathcal{L}}}
\def\dx{{\,\mathrm{d}x}}
\def\dy{{\,\mathrm{d}y}}
\def\ds{{\,\mathrm{d}s}}
\def\dtau{{\,\mathrm{d}\tau}}
\def\dt{{\,\mathrm{d}t}}

\def\cI{{\mathcal{I}}} 
\def\cB{{\mathcal{B}}}

\def\tedv{{\tens{\dv}}}
\def\temu{{\tens{\mu}}}
\def\btemu{{\overline{\temu}}}
\def\tea{\tens{a}}

\def\tI{\text{I}}
\def\tII{\text{II}}
\def\tIII{\text{III}}
\def\rn{{\mathbb{R}^{n}}}
\def\Rm{{\mathbb{R}^{m}}}
\def\Rn{{\mathbb{R}^{n}}}
\def\id{{\mathsf{Id}}} 
\def\Mb{{\mathcal{M}(\Omega,\Rm)}} 

\newtheorem{coro}{\bf Corollary}[section]
\newtheorem{theo}[coro]{\bf Theorem} 
\newtheorem{lem}[coro]{\bf Lemma}
\newtheorem{rem}[coro]{\bf Remark} 
\newtheorem{defi}[coro]{\bf Definition} 
\newtheorem{ex}[coro]{\bf Example} 
\newtheorem{fact}[coro]{\bf Fact} 
\newtheorem{prop}[coro]{\bf Proposition}

\newcommand{\data}{\textit{\texttt{data}}}


\parindent 1em

\begin{abstract}
We study gradient regularity for 
mixed local-nonlocal  problems modelled upon
\[ -\Delta_p u +(-\Delta_p)^su=\mu\qquad\text{for} \quad 2-\tfrac{1}{n}<p<\infty\quad \text{and}\quad s\in(0,1)\,,\]
where $\mu$ is a bounded Borel measure. We prove pointwise bounds for the gradient $Du$ in terms of the truncated 1-Riesz potential of $\mu$. 
\end{abstract}


\section{Introduction}

In this paper, we investigate mixed local-nonlocal nonlinear problems whose prototype is
\begin{equation}\label{eq:model}
    -\Delta_p u +(-\Delta_p)^su=\mu\qquad\text{for} \quad 2-\frac{1}{n}<p<\infty\quad \text{and}\quad s\in(0,1)\,.
\end{equation}
Here, $\Delta_p$ is the classical $p$-Laplacian, $(-\Delta_p)^s$ is the $s$-fractional $p$-Laplacian defined by
\[ (-\Delta_{p})^{s}u(x) \coloneqq P.V.\int_{\mathbb{R}^{n}}\frac{|u(x)-u(y)|^{p-2}(u(x)-u(y))}{|x-y|^{n+sp}}\dy, \qquad x \in \mathbb{R}^{n}, \]
and $\mu$ is a signed Borel measure with finite total mass. 
We investigate solutions to this kind of problem, which are understood in a properly general sense (see Definition~\ref{def.sola} below). Specifically, we provide pointwise potential estimates and infer fine properties of the gradient of the solutions. 
Such results in the relam of nonlinear potential theory have been extensively studied for various kinds of elliptic problems, see, e.g., \cite{Ba15,BaJDE22,ByunSong,ByYoun,CGZG-Wolff,CKW,CYZG-Wolff,KiMa94,KuMiCV14,KuMiSi,Trudinger-Wang}.  We refer to Sections~\ref{sec:main}--\ref{sec:coro} below for more details on our results.

\subsection{Mixed local-nonlocal operators} The last two decades brought new insight into the theory of mixed problems encompassing both local and nonlocal features. Operators used in such analysis also play a crucial role in probability theory as infinitesimal generators of stochastic processes that display both diffusion and jump. 
The early results on mixed operators, primarily focused on the linear case whose simplest model is $-\Delta + (-\Delta)^{s}$,
were obtained by using probabilistic \cite{CKSV,ChKu2010} and viscosity methods \cite{BarIm}. For various results concerning mixed local-nonlocal linear problems, including regularity, maximum principle, and other qualitative properties, see \cite{BDVV,BDVV23,DPLV,diva} and references therein.

Studies on regularity results for mixed local-nonlocal $p$-Laplacian type operators were initiated in the paper \cite{GarKin22}, where the De Giorgi--Nash--Moser theory was investigated by combining the techniques for the fractional $p$-Laplacian \cite{DKP14,DKP16} with those for the classical $p$-Laplacian. We refer to \cite{Balci,BDVV21,BLS,GarLin23,gauk,Nak} and references therein for further results concerning mixed local-nonlocal nonlinear problems. 
We would like to single out the paper \cite{DFM-mixed} where maximal regularity, including (local) gradient H\"older regularity, was proved for a larger class of mixed problems related to
\begin{equation*} 
-\Delta_{p} + (-\Delta_{\gamma})^{s}
\end{equation*}
with $p,\gamma>1$ and $s\in(0,1)$ satisfying $p > s\gamma$. 
A key idea in \cite{DFM-mixed} is to consider the nonlocal term as a lower order term (due to the conditions on the constants $p,s,\gamma$), which allows one to compare the mixed problem with a local problem to obtain gradient regularity. We also refer to \cite{BKL} for global gradient estimates for the same kind of operators with measurable nonlinearities. 

\subsection{Nonlinear potential theory} 
Nonlinear potential theory is a subject that aims to provide pointwise, oscillation estimates and fine properties of solutions for nonlinear equations, which extend in a most natural way the classical ones valid for linear equations via the representation formula. The fundamental work \cite{KiMa94} initiated the potential theory of degenerate/singular elliptic problems of $p$-Laplacian type; see also \cite{Trudinger-Wang} for a different proof. Specifically, the following pointwise estimate for solutions holds:
\begin{equation}\label{wolff.u}
|u(x_{0})| \lesssim \mathcal{W}^{|\mu|}_{1,p}(x_{0},R) + \mean{B_{R}(x_{0})}|u|\dx \,,
\end{equation}
where the nonlinear Wolff potential of $|\mu|$ is defined  as
\[ \mathcal{W}^{|\mu|}_{\beta,p}(x_{0},R) \coloneqq \int_{0}^{R}\left[\frac{|\mu|(B_{\varrho}(x_{0}))}{\varrho^{n-\beta p}}\right]^{1/(p-1)}\frac{\mathrm{d}\varrho}{\varrho}\,,  \qquad \beta>0 \,. \]
Also, a pointwise lower bound for $u$ via $\mathcal{W}^{|\mu|}_{1,p}$ is available when both $\mu$ and $u$ are nonnegative, which implies the sharpness of \eqref{wolff.u} in the sense that $\mathcal{W}^{|\mu|}_{1,p}$ cannot be replaced by any other smaller potentials.

Pointwise gradient estimates for nonlinear elliptic measure data problems were first obtained in \cite{Min11} for the case $p=2$, whose form is essentially the same as the classical one valid for the Poisson equation. In \cite{DuMiAJM2011}, the gradient bound via $\mathcal{W}^{|\mu|}_{1/p,p}$-Wolff potential was obtained for $p$-Laplacian type equations for $p \ge 2$. See also \cite{KuMiJFA2012} for some oscillation estimates that interpolate between Wolff potential estimates for solutions and their gradient. The paper \cite{DuMiAJM2011} introduced a method employing comparison and excess decay estimates, which is by now standard in gradient potential estimates. 

Surprisingly, in contrast with \eqref{wolff.u}, it was shown in \cite{KuMiARMA2013} that pointwise estimates for the gradient can be actually established via Riesz potentials for $p\geq 2$, which originally appeared in estimates for the linear Poisson equation:
\begin{equation}\label{Riesz.Du} 
|Du(x_{0})| \lesssim \left[\cI^{|\mu|}_{1}(x_{0},R)\right]^{1/(p-1)} + \mean{B_{R}(x_{0})}|Du|\dx \,, 
\end{equation}
where
\begin{equation*}
\cI^{|\mu|}_{1}(x_0,R) \coloneqq \int_0^R  \frac{|\mu|(B_{\varrho}(x_0))}{\varrho^{n-1}}\,\frac{\mathrm{d}\varrho}{\varrho}\,
\end{equation*} 
is the truncated 1-Riesz potential of $|\mu|$. In particular,  
\eqref{Riesz.Du} improves the Wolff potential estimate in \cite{DuMiAJM2011} in the sense that
\[ p \ge 2 \;\; \Longrightarrow \;\; \left[\cI^{|\mu|}_{1}(x_{0},R)\right]^{1/(p-1)} \lesssim \mathcal{W}^{|\mu|}_{1/p,p}(x_{0},2R)\,. \]
Moreover, it turns out that the Riesz potential estimates similar to \eqref{Riesz.Du} hold for any $p>1$, see \cite{DuMiJFA2010,NP23ARMA}. 
As a consequence, various gradient regularity results for the Poisson equation are carried over to those for nonlinear (possibly degenerate) equations of $p$-Laplacian type
.  
Nowadays, such potential theoretic arguments, as well as the original idea of \cite{Min11}, yield a unified approach to finding optimal conditions for Lipschitz estimates in a general class of nonuniformly elliptic and nonautonomous problems \cite{BM20, DM21}.  
We also note that the pointwise estimates \eqref{wolff.u} and \eqref{Riesz.Du} 
have been extended to cover more general local operators with nonstandard growth, see  e.g.~\cite{CGZG-Wolff,CYZG-Wolff} and \cite{Ba15,BaHa,BY17,BY18,CKW}, respectively.  
For a comprehensive overview of nonlinear potential theory for local problems, see \cite{KuMiguide}. 

Nonlinear potential estimates for nonlocal measure data problems involving $s$-fractional $p$-Laplacian type operators were first developed in \cite{KuMiSi,KuMiSi18}. The main outcomes therein include existence results and pointwise upper and lower estimates for solutions via $\mathcal{W}^{|\mu|}_{s,p}$. In particular, the nonlocal contribution of solutions, which was analyzed in the excess decay estimates, eventually results in the appearance of so-called nonlocal tail like 
\[ \left(r^{sp}\int_{\mathbb{R}^{n}\setminus B_{r}(x_{0})}\frac{|u(y)|^{p-1}}{|y-x_{0}|^{n+sp}}\dy\right)^{1/(p-1)} \]
in the final potential estimates. We also mention the paper \cite{DiNo} that establishes oscillation estimates for solutions in terms of fractional maximal functions of the data.  These estimates are analogous to those in \cite{KuMiJFA2012}; thereby, they lead to the Calder\'on--Zygmund type estimates in the scale of fractional Sobolev spaces. 
In \cite{KuNoSi}, first-order differentiability of solutions and gradient potential estimates were proved for a class of nonlocal linear equations with coefficients. 
However, the question of extending such gradient regularity results to fractional $p$-Laplacian type equations still remains open. In fact, even the Lipschitz regularity for the homogeneous fractional $p$-Laplace equation is not known so far. 

As for the case of mixed local-nonlocal equations modelled on \eqref{eq:model}, 
research in \cite{ByunSong} studied the existence and $\mathcal{W}^{|\mu|}_{1,p}$-Wolff potential estimates for solutions. As a consequence, integrability results and continuity criteria for solutions were also obtained. We point out that the results in \cite{ByunSong} reflect both local and nonlocal characters of the mixed operator. Moreover, similarly to \cite{DFM-mixed}, the approach of \cite{ByunSong} considers the local term as a leading term in view of the scaling property, thereby concluding with estimates via $\mathcal{W}^{|\mu|}_{1,p}$. 
This is why the nonlocal tail used in the analysis of mixed problems (see \eqref{Tail} below), set to the scaling property of the local $p$-Laplacian, is slightly different from the above one considered in \cite{KuMiSi}. 

In this paper, by developing the approaches in \cite{ByunSong,DFM-mixed} and combining them with the methods of local nonlinear potential theory, we aim at supplying the nonlinear potential theory for mixed operators with pointwise gradient estimates via Riesz potentials. To our knowledge, our results are the first ones that deal with gradient regularity for mixed local-nonlocal nonlinear problems like \eqref{eq:model} with measure data. We are in a position to show our main accomplishments.

\subsection{Assumptions and results}\label{sec:main}
Let us present in detail our setting. We study the following problem 
\begin{equation}
    \label{eq:main}
\left\{
\begin{aligned}
-\mathrm{div}\,A(Du) + \mathcal{L}u & = \mu & \text{in } & \Omega\,, \\
u & = g & \text{in } & \mathbb{R}^{n}\setminus \Omega\,,
\end{aligned}
\right.
\end{equation}
where $\Omega \subset \mathbb{R}^{n}$ is a bounded domain, {$n\geq 2$}, $g:\mathbb{R}^{n} \rightarrow \mathbb{R}$ is a boundary data and $\mu$ is a signed Borel measure with finite total mass on $\rn$. 
The vector field $A: \mathbb{R}^{n} \to \mathbb{R}^{n}$ is assumed to be $C^{1}$-regular on $\mathbb{R}^{n}$ for $p \ge 2$ and on $\mathbb{R}^{n}\setminus\{0\}$ for $p<2$.
It also satisfies the following growth and ellipticity assumptions:
\begin{equation}\label{growth}
\left\{
\begin{aligned}
|A(z)| + |\partial A(z)| |z| &\le L_A |z|^{p-1}\,,\\
\nu_A |z|^{p-2} |\xi|^{2} &\le \partial A(z) \xi \cdot \xi\,, \\
\end{aligned}
\right.
\end{equation}
for every $z , \xi \in \mathbb{R}^{n}$, where $0<\nu_{A} \le L_{A} <\infty$ are fixed.
The nonlocal operator $\cL(\cdot)$ is defined by 
\begin{equation}
    \label{def:Lu}
    \cL u(x) \coloneqq P.V.\int_{\rn} |u(x)-u(y)|^{p-2}(u(x)-u(y))K(x,y)\dy\,, \qquad x \in \mathbb{R}\,,
\end{equation}
where, as customary, $P.V.$ stands for the Cauchy’s principal value, whereas $K:\rn \times \rn \rightarrow \mathbb{R}$ is a measurable, symmetric kernel satisfying
\begin{equation}\label{kernel.growth}
    \frac{\nu_K}{|x-y|^{n+sp}} \le K(x,y) = K(y,x) \le \frac{L_K}{|x-y|^{n+sp}}\qquad \text{for  a.e. $x,y \in \rn\,$,}
\end{equation}
 where $0<\nu_{K} \le L_{K} <\infty$ are fixed. In this paper, we assume
\begin{equation}\label{s-p-range}
s \in (0,1) \qquad\text{and}\qquad p > 2-\frac{1}{n}\,.
\end{equation}
For the case of referring to the dependence of some quantities on the parameters of the problem, we collect them as \[\data=\data(n,s,p,\nu_{K}, \nu_{A}, L_{K}, L_{A})\,.\] 
Our accomplishment is the following result.
\begin{theo}\label{theo:intro}
Let $u$ be a SOLA to \eqref{eq:main} under assumptions \eqref{growth}--\eqref{s-p-range}. For any $\sigma \in (0,1)$, there exists a positive radius $\bar{R} = \bar R(\data,\sigma) \in (0,1)$ such that the following holds: 
if the truncated Riesz potential $\cI^{|\mu|}_{1}(x_0,R)$ is finite for a ball  $B_{R}(x_{0}) \subset \Omega$ with $R \le \bar{R}$, then $x_0$ is a Lebesgue point of $Du$ in the sense that there exists the precise representative of $Du$ at $x_0$:
\[Du(x_0) \coloneqq \lim_{\vr\to 0} (Du)_{B_\vr(x_0)}  = \lim_{\vr\to 0} \mean{B_{\varrho}(x_{0})} Du \dx.\]
Moreover, there exists a constant $c = c(\data,\sigma)$ such that the following estimate holds with $q_0=\max\{p-1,1\}$:
\begin{align*}
|Du(x_0 )| & \leq c \left[\cI^{|\mu|}_{1} (x_0 ,R)\right]^{1/(p-1)}+c\left(\mean{B_R(x_0)} |Du|^{q_0}\dx\right)^{1/q_0} \notag \\
& \quad + c\left(R^{\sigma}\int_{\rn\setminus B_{R}(x_{0})}\frac{|u(y)-(u)_{B_{R}(x_{0})}|^{p-1}}{|y-x_{0}|^{n+sp}}\dy \right)^{1/(p-1)}\,.
\end{align*} 
\end{theo}

In fact, the above pointwise estimate is obtained as a consequence of the following oscillation estimates.
\begin{theo}\label{theo:main}
Under the assumptions of Theorem~\ref{theo:intro}, 
there exist constants $c \ge 1$ and $\kappa \in (0,1)$, both depending only on $\data$ and $\sigma$, such that the following estimates hold:
\begin{enumerate} [(i)]
\item If $p \geq 2$, then 
\begin{align}\label{theo:main-est1}
    |A(Du(x_{0})) - (A(Du))_{B_{R}(x_{0})}|  & \le c \cI^{|\mu|}_{1}(x_{0},R)+c\mean{B_{R}(x_{0})}|A(Du)-(A(Du))_{B_{R}(x_{0})}|\dx \notag\, \\
    & \quad + c R^{\kappa} \mean{B_{R}(x_0)} |A(Du)| \dx + cR^{\sigma}\int_{\rn\setminus B_{R}(x_{0})}\frac{|u(y)-(u)_{B_{R}(x_{0})}|^{p-1}}{|y-x_{0}|^{n+sp}}\dy\,.
\end{align}
\item If $2-1/n < p < 2$, then 
\begin{align}\label{theo:main-est2}
|Du(x_{0})-(Du)_{B_{R}(x_{0})}| & \le  c\left[\cI^{|\mu|}_{1}(x_{0},R)\right]^{1/(p-1)}+ c\left[\cI^{|\mu|}_{1}(x_{0},R)\right]\left(\mean{B_{R}(x_{0})}|Du|\dx\right)^{2-p} \notag\\
& \quad +c\mean{B_{R}(x_{0})}|Du-(Du)_{B_{R}(x_{0})}|\dx+ cR^{\kappa}\mean{B_{R}(x_{0})}|Du|\dx \notag \\
& \quad + c\left(R^{\sigma}\int_{\rn\setminus B_{R}(x_{0})}\frac{|u(y)-(u)_{B_{R}(x_{0})}|^{p-1}}{|y-x_{0}|^{n+sp}}\dy \right)^{1/(p-1)} \,.
\end{align}
\end{enumerate} 
\end{theo}
 
 We remark that, due to the same reasons as in \cite[Remark~1.13]{ByunSong}, our results of Theorems~\ref{theo:intro} and \ref{theo:main} continue to hold when the nonlocal operator $\mathcal{L}(\cdot)$ given in \eqref{def:Lu} is replaced by
\[ \mathcal{L}_{\Phi}u(x) \coloneqq P.V.\int_{\rn}\Phi(u(x)-u(y))K(x,y)\dy, \]
where  $K(\cdot,\cdot)$ is not necessarily symmetric and $\Phi:\mathbb{R}\rightarrow\mathbb{R}$ is a continuous function satisfying
\[ \nu_\Phi |t|^{p} \le \Phi(t)t \le L_\Phi |t|^{p} \qquad \text{for all }\ t\in\mathbb{R}\qquad \text{with fixed constants}\ 0<\nu_{\Phi} \le L_{\Phi} <\infty\,. \] 

\subsection{Gradient regularity as corollary}\label{sec:coro}
A remarkable consequence of Theorems~\ref{theo:intro} and \ref{theo:main} is a reduction of the problem of getting sharp gradient estimates for~\eqref{eq:main} to the study of Riesz potentials, whose properties are well-understood. In fact, this way, we can infer several transfers of regularity from data to the gradient of solutions.  Before presenting the definition of Lorentz spaces,  we recall the definition of decreasing rearrangement of a function due to~\cite{BeSa88}. 
Let $f:\Omega\to\R$ be a measurable function; we assume that $f$ is extended to the whole $\rn$ by letting $f\equiv0$ on $\rn\setminus\Omega$. 
The decreasing rearrangement $f^* : [0, \infty) \to [0, \infty]$ is given by 
\begin{equation*}
f^\ast(t) \coloneqq \sup\Big \{ s \geq 0 \colon \big|\{x\in \R^n:|f(x)|>s\}\big| > t  \Big\}\,,
\end{equation*}
where we assume that $\sup\emptyset=0$. Equivalently it can also be defined as the function which is right-continuous, non-increasing, and equimeasurable with $f$, i.e., $|\{x : |f(x)| > t\}| = |\{x: f^{*}(x) > t\}|$ for all $t > 0$. 

Let $0 < \gamma,q \leq \infty$. The Lorentz space $L^{\gamma,q}(\Omega)$ is defined as the set of measurable functions $f : \Omega \to \R$ such that $\|f\|_{L^{\gamma,q}(\Omega)} < \infty$, where
\begin{equation*}
\|f\|_{{L^{\gamma,q}}(\Omega)} \coloneqq 
\left\| (\cdot)^{1/\gamma-1/q}f^{\ast}(\cdot) \right\|_{L^{q}(0,\infty)} \,.
\end{equation*}
For $q=\infty$ this space is usually called the Marcinkiewicz space. 
If we consider the maximal rearrangement defined by
\begin{equation*}
f^{\ast\ast}(t) \coloneqq \frac 1t \int_0^t f^\ast(s) \, \mathrm{d}s\,, \qquad t>0\,, 
\end{equation*}
then for any $\gamma \in (1,\infty]$ we have \begin{equation*}
\|f\|_{L^{\gamma,q}(\Omega)} \approx 
\left\| (\cdot)^{1/\gamma-1/q}f^{\ast\ast}(\cdot) \right\|_{L^{q}(0,\infty)} \,,
\end{equation*}
where $\approx$ means that the two quantities are comparable up to constants depending on the parameters of the space. Moreover, the quantity on the right-hand side makes $L^{\gamma,q}(\Omega)$ a rearrangement-invariant Banach space when $\gamma \in (1,\infty)$, $q \in [1,\infty]$ or $\gamma=q=\infty$. We also have strict inclusions
\[ L^{\gamma} = L^{\gamma,\gamma} \subsetneq L^{t,q} \subsetneq L^{t,t} \subsetneq L^{t,\gamma} \subsetneq L^{q,q} = L^{q} \quad \text{for any}\quad 0<q<t<\gamma<\infty\,. \] 

We also consider the space $L\log  L(\Omega)$, which is the subset of integrable functions $f:\Omega\to\R$ such that 
\[ \|f\|_{L \log  L(\Omega)} \coloneqq \inf\left\{\lambda>0: \int_\Omega \frac{|f|}{\lambda}\log\left(e +\frac{|f|}{\lambda} \right)\dx\leq 1\right\} <\infty\,. \]

The following Calder\'on--Zygmund-type results are direct consequences of Theorem~\ref{theo:intro}. As the transfer of regularity is the same as in the local $p$-Laplacian type problems, we refer to \cite[Section~9]{KuMiguide} and \cite[Section~10]{KuMi2018jems} for detailed discussions. In fact, for  $u$ being a SOLA  to~\eqref{eq:main} under assumptions \eqref{growth}--\eqref{s-p-range}, the following implications hold true:
\begin{enumerate}[{\it (i)}]
    \item  If $1 < \gamma <n$, $0 <q\leq\infty$, and $\mu \in L^{\gamma,q}_{\mathrm{loc}}(\Omega)$, then $|Du|^{p-1}\in L^{n\gamma/(n-\gamma),q}_{\mathrm{loc}}(\Omega)$;
   \item   If locally $\mu \in \mathcal{M}_{b}(\Omega)$, then $|Du|^{p-1}\in L^{n/(n-1),\infty}_{\mathrm{loc}}(\Omega)$;
   \item  If locally $\mu \in L\log L (\Omega)$, then $|Du|^{p-1}\in L_{\mathrm{loc}}^{n/(n-1)}(\Omega)$.
\end{enumerate}
Note that the above results were known in the local problems but are new for the mixed problem~\eqref{eq:main}.

Another notable new result, which is a consequence of Theorem~\ref{theo:main}, is a borderline gradient continuity criterion. We also refer to \cite{KuMiCV14} for a different proof. 

\begin{coro}\label{coro:cont}
Let $u$ be a SOLA to~\eqref{eq:main}  under assumptions \eqref{growth}--\eqref{s-p-range}. There exists $\sigma = \sigma(n,p) \in (0,1)$ such that the following holds: if $B \Subset\Omega$ is a ball with radius less than $\bar{R}$ given in Theorem~\ref{theo:main} and
\begin{equation*}
\lim_{\vr\to 0}\sup_{x\in B}\cI^{|\mu|}_{1}(x,\vr)=0\,,
\end{equation*}
then $Du$ is continuous in $B$. In particular, if $\mu \in L^{n,1}_{\loc}(\Omega)$, then $Du$ is continuous in $\Omega$. 
\end{coro}

\subsection{Strategy and novelty} 
As in the case of local equations, our basic strategy for gradient potential estimates is to examine the decay property of certain excess functionals of the gradient. However, the presence of the nonlocal term in our problem \eqref{eq:main} gives rise to several difficulties in dealing with gradient estimates.
Indeed, even for homogeneous mixed problems, while the gradient H\"older regularity is proved in \cite{DFM-mixed, GarLin23}, an explicit form of excess decay estimate for the gradient is not known as far as we know. 
Therefore, our first aim is to obtain such an estimate for the homogeneous mixed equation
\[ -\mathrm{div}\,A(Dv) + \mathcal{L}v  = 0\,. \]
To this aim, we start by establishing a comparison estimate (see Lemma~\ref{mix.loc.comp})
with a relevant comparison problem involving the homogeneous local equation
\[ -\mathrm{div}\,A(Dw) = 0 \,. \] 
The basic idea is the same as in \cite[Lemma~4.2]{DFM-mixed}. However, in our measure data problems, we have to make use of tails ($L^{p-1}$-nonlocal averages) instead of snails ($L^{p}$-nonlocal averages) employed in \cite{DFM-mixed}. Note that the use of snails is restricted to the energy range only. 
We then combine such a comparison estimate between $Dv$ and $Dw$ with the known excess decay estimate for $Dw$ given in Lemma~\ref{localeq.ed}. 
Indeed, the main difficulty here is that in Lemma~\ref{lem:ex1}, the tail does not display the same rate of decay property as that of the gradient. To overcome this difficulty, we further restrict the radii of balls and then show a uniform density estimate for the gradient and tail terms over a sequence of concentric balls via an induction argument. Consequently, we are able to infer quantified H\"older estimates for $Dv$ (Theorem~\ref{theo:homo-mixed-hol}), which is also a new result of its own interest, of the following form 
\begin{align*}
\lefteqn{ \mean{B_{\rho}} |Dv - (Dv)_{B_{\rho}}|\dx } \\
&\leq c \left( \frac{\rho}{r} \right)^{\beta_{v}} \Bigg[ \mean{B_{r}} |Dv - (Dv)_{B_{r}}| \dx + r^{\bar a_1 - \ve_{1}} \left( \mean{B_{r}}|Dv|^{q_0} \dx \right)^{1/q_0}  + r^{\bar a_2 - \ve_{1}}\left[\frac{1}{r^{p'}}\tail(v-(v)_{B_{r}};r)\right] \Bigg]\,,
\end{align*} 
where $\beta_v\in(0,1)$ is a universal constant, $\bar a_1,\bar a_2>0$ are given in \eqref{bar a def}, and $\tail(\cdot)$ is defined in \eqref{Tail}. We note that in view of the scaling property of the problem (see also \cite[Theorem~1.4]{ByunSong}), this procedure yields a sub-optimal decay of tails in the sense that we cannot allow $\varepsilon_{1} = 0$ in Theorem~\ref{theo:homo-mixed-hol}, thereby prohibiting us from taking $\sigma=1$ in Theorems~\ref{theo:intro} and \ref{theo:main}. Nevertheless, we believe that the estimates we provide are the best ones that can be obtained by the applied method. 
In fact, Theorem~\ref{theo:homo-mixed-hol} brings a quantitative version of \cite[Theorem~5]{DFM-mixed} in the case $\gamma=p$ and $f\equiv 0$. On the other hand, the study of \cite{DFM-mixed} was carried out for a more general class of operators, as mentioned above.  

We next explain the process of transferring such decay property of $Dv$ to $Du$ via additional comparison estimates. 
Some basic estimates that we owe to \cite{ByunSong} are recalled in Lemma~\ref{lem:mix-comp2}. 
In this stage, as typical in the study of measure data problems involving $p$-Laplacian type operators (see e.g. \cite{AvKuMi, ByYoun, DuMiJFA2010, DuMiAJM2011, KuMiARMA2013, KuMiJEP2016}), the strategy splits into the superquadratic and subquadratic cases, since the problem exhibits different monotonicity properties according to the range of $p$. 

\subsubsection*{The superquadratic case} 
In this case, the main point is that we have to consider the degeneracy of the problem in order to obtain estimates via Riesz potentials, bypassing Wolff potentials. 
For this, we intend to establish all the comparison and excess estimates in terms of $A(Du)$ instead of $Du$, relying on an intrinsic linearization technique which was introduced in~\cite{AvKuMi} and applied to gradient potential estimates in \cite{BSY}. 
Consequently, the main result is concerned with an excess decay estimate for $A(Du)$. Such an estimate is obtained by first proving an excess decay estimate for $A(Dv)$ (contained in Proposition~\ref{prop:homo-mixed-hol2}) and then transferring such an estimate to $A(Du)$ by means of enhanced comparison estimates between $A(Du)$ and $A(Dv)$ (contained in Lemma~\ref{lem:A-comp}). 
To achieve the comparison estimate involving the $A(\cdot)$-map, we perform a two-scale alternative scheme involving the excess and mean value of $A(Du)$, that is
\[E(A(Du);B_{Mr}) \coloneqq \mean{B_{Mr}} |A(Du) - (A(Du))_{B_{Mr}}| \dx\qquad \text{and}\qquad |(A(Du))_{B_{r/M}}|\,, \]
with $M$ being a free parameter, which will be eventually fixed as a universal constant. 
We distinguish between the situation when $|(A(Du))_{B_{r/M}}|$ is small (resp. large) comparing to $E(A(Du);B_{Mr})$. 
Precisely, the alternative scheme splits the analysis into the following cases: 
\begin{enumerate}[{\it i)}]
    \item  {\it degenerate} -- $|(A(Du))_{B_{r/M}}|$ can be controlled from above by $E(A(Du);B_{Mr})$;
    \item {\it non-degenerate} -- $|(A(Du))_{B_{r/M}}|$ is much larger than $E(A(Du);B_{Mr})$. 
\end{enumerate} 
The non-degenerate case requires more careful treatment and is further divided into two subcases. If the mean value of $|Du|^{p-1}$ is larger than the measure term and the tail term, then we can further strengthen the comparison in an intrinsic form. This results from the fact that 
the oscillation of $Du$ is small in the relevant sense, which implies that the leading local part of the operator is non-degenerate. 
In this case, we consider an additional comparison map $v_{*}$ (defined on $B_{Mr}$ instead of $B_r$) and make use of another comparison estimate given in Lemma~\ref{lem:mix-comp3}. 
Finally, we can conclude with estimates for the excess of $A(Du)$ and related tail terms in 
Lemmas~\ref{lem:ADuexcess-pge2} and \ref{lem:u-tail}, respectively. Summing up and iterating the resulting estimates, we conclude with \eqref{theo:main-est1}.

\subsubsection*{The subquadratic case} This situation cannot be handled by the same reasoning as in the abovementioned case since it differs in the monotonicity properties of the vector field $A(\cdot)$. 
In the local case, the techniques from \cite{AvKuMi, BSY} can be applied in the subquadratic case (see also \cite{BSY2}), but this requires several reverse H\"older type estimates for the gradient.  In the mixed case, the proof of such estimates is not clear due to the presence of nonlocal tails. Therefore we use directly $Du$ in this case. Accordingly, the proof in the subquadratic case requires the control on the excess of $Du$, which is obtained as a result of basic comparison between $Du$ and $Dv$ given in Section~\ref{sec:comp-homo-mixed-meas-mixed}.  This requires the control of an additional term appearing in Lemma~\ref{lem:mix-comp2}, which is responsible for the second term on the right-hand side of \eqref{theo:main-est2}. Moreover, again, due to the monotonicity of the problem, in Theorem~\ref{theo:homo-mixed-hol}, the exponents related to the decay rate of the gradient and tail are different compared to the superquadratic case. Here, an additional free parameter $m \in (0,1-s)$ appears. It is eventually fixed as a universal constant in the later proof. 
Taking into account all these features, we conclude with excess decay estimates for $Du$ and tail estimates whose final forms are analogous to the ones obtained in the superquadratic case. Then, a similar iteration procedure as described above leads to \eqref{theo:main-est2}.

\subsection{Organization} 
Section~\ref{sec:prelim} introduces basic notation and known regularity results for reference problems. In Section~\ref{sec:quant-hold} we provide quantitative gradient H\"older regularity for homogeneous mixed equations via the comparison with a homogeneous local problem. Excess decay estimates for measure data mixed problems are presented in Section~\ref{sec:ex-dec-measure-data}. Section~\ref{sec.pf} is devoted to the proofs of the main result (Theorem~\ref{theo:main}) and its consequences.

\section{Preliminaries}\label{sec:prelim}
\subsection{Notation}

We shall adopt the customary convention of denoting by $c$ a constant greater than or equal to 1,  whose value may vary from line to line. To skip rewriting a constant, we use $\lesssim$. By $a\approx b$, we mean $a\lesssim b$ and $b\lesssim a$. 
By $B_R$ we denote a ball skipping prescribing its center, when {it} is not important. By $cB_R=B_{cR}$ we mean a ball with the same center as $B_R$, but with rescaled radius $cR$. 
In the sequel, unless otherwise mentioned, we always assume that every ball has a radius of less than one. For any $p>1$ we set $p'=p/(p-1)$.

Let $\mathcal{M}_b(U)$ for an open set $U\subset\rn$ denote the set of bounded signed Borel measures on $U$. For a~measurable set $U\subset \rn$  with finite and positive $n$-dimensional Lebesgue measure $|U|>0$ and $f\colon U\to \R^{k}$, $k\ge 1$ being a measurable map, we denote
\[ (f)_U \coloneqq \mean{U}f(x) \dx =\frac{1}{|U|}\int_{U}f(x) \dx\,. \]
We also denote 
\[ \osc_{U} f \coloneqq \sup_{x,y\in U}|f(x)-f(y)|\,. \]

Let us consider the tail space defined by
\begin{equation*}
L^{p-1}_{sp}(\mathbb{R}^{n}) \coloneqq \left\{f:\mathbb{R}^{n}\rightarrow\mathbb{R}  : \ \int_{\mathbb{R}^{n}}\frac{|f(x)|^{p-1}}{(1+|x|)^{n+sp}}\dx < \infty \right\}\,.
\end{equation*}
It is not difficult to see that, whenever $f \in L^{p-1}_{sp}(\rn)$, the nonlocal tail of $f$
\begin{equation}\label{Tail}
    {\rm Tail}(f;x_0,r) \coloneqq \left(r^p\int_{\rn\setminus B_r(x_0)}\frac{|f(x)|^{p-1}}{|x-x_0|^{n+sp}}\dx\right)^{1/(p-1)}
\end{equation}
is finite for any $x_0$ and $r>0$. Note that $L^{p}(\rn)\subset L^{p-1}_{sp}(\rn)$.

We consider the vector field $V:\rn\to\rn$ defined by
\begin{equation}\label{V}
V(z) \coloneqq |z|^\frac{p-2}{2}z \quad \text{for }\ z\in \rn\,.
\end{equation}

\begin{lem}\label{lem:mono}
Let $V$ be given by~\eqref{V} and $A$ satisfy \eqref{growth}. Then both $V$ and $A$ are locally bi-Lipschitz bijections of $\rn$. Moreover, the following inequalities hold for any $z_1,z_2 \in \rn$:
\begin{equation*}
\begin{aligned}
|A(z_{1})-A(z_{2})| & \approx (|z_{1}|+|z_{2}|)^{p-2}|z_{1}-z_{2}|\,, \\
(A(z_{1})-A(z_{2})) \cdot (z_{1}-z_{2}) & \approx (|z_{1}|+|z_{2}|)^{p-2}|z_{1}-z_{2}|^{2} \approx |V(z_{1})-V(z_{2})|^{2}\,.
\end{aligned}
\end{equation*}
\end{lem}

We next recall the definition and basic properties of fractional Sobolev spaces; a general reference is \cite{DPV}.
Let $p \ge 1$ and $s \in (0,1)$. For any open set $\mathcal{O}\subseteq \mathbb{R}^{n}$, the fractional Sobolev space $W^{s,p}(\mathcal{O})$ is the set of all functions $f \in L^{p}(\mathcal{O})$ satisfying
\begin{align*}
\lVert f \rVert_{W^{s,p}(\mathcal{O})} \coloneqq \lVert f \rVert_{L^{p}(\mathcal{O})} + [f]_{s,p;\mathcal{O}} \coloneqq \left(\int_{\mathcal{O}}|f|^{p}\dx\right)^{1/p} + \left(\int_{\mathcal{O}}\int_{\mathcal{O}}\frac{|f(x)-f(y)|^{p}}{|x-y|^{n+sp}}\dx\dy\right)^{1/p} < \infty\,.
\end{align*}
We define Sobolev space of functions with zero trace as follows
\[ \mathcal{X}^{1,p}_{0}(\Omega) \coloneqq \left\{ f \in W^{1,p}(\mathbb{R}^{n}) : \ \ f|_{\Omega} \in W^{1,p}_{0}(\Omega),\ f = 0\ \text{ a.e. in } \mathbb{R}^{n}\setminus\Omega \right\}\,. \]

Due to \cite[Lemma~2.2]{DFM-mixed} (see also \cite[Proposition~2.2]{DPV}), there holds the following fact.
\begin{lem}
    Let $1 \leq  p < \infty$, $s \in (0,1)$ and $B_{r} \subset \rn$ be a ball. If $h \in W^{1,p}_{0}(B_{r})$, then $h \in W^{s,p}(B_{r})$ and for a constant $c = c(n,p,s)>0$ it holds that
    \begin{equation}\label{Sob.ineq}
        \left( \int_{B_{r}}\mean{B_{r}} \frac{|h(x) - h(y)|^{p}}{|x-y|^{n+sp}} \dx \dy \right)^{1/p} \leq c r^{1-s} \left( \mean{B_{r}} |Dh|^{p} \dx \right)^{1/p}\,.
    \end{equation}
\end{lem}

\subsection{Weak solutions and SOLA}

When the datum is regular, we define weak solutions to~\eqref{eq:main} as follows.

\begin{defi}[Weak solution]\label{weak.sol}
Let $\mu \in W^{-1,p'}(\Omega)$ and $1<p<\infty$. We say that a function $u \in
W^{1,p}_{\loc}(\Omega) \cap W^{s,p}(\mathbb{R}^{n})$ is a~weak solution to the equation
\begin{equation}
    \label{eq-def-weak}
-\mathrm{div}\,A(Du) + \mathcal{L}u = \mu \quad \text{in } \Omega\,, \end{equation} 
under assumptions
\eqref{growth}--\eqref{kernel.growth} with $p > 1$ and $s\in(0,1)$,
if for any $\varphi \in C^{\infty}_{0}(\Omega)$ it holds
\begin{equation*}
\int_{\Omega}A(Du)\cdot D\varphi\dx +
\int_{\mathbb{R}^{n}}\int_{\mathbb{R}^{n}}|u(x)-u(y)|^{p-2}(u(x)-u(y))(\varphi(x)-\varphi(y))K(x,y)\dx\dy = \langle\mu,\varphi\rangle\,.
\end{equation*}
Moreover, given any boundary data $g \in W^{1,p}(\mathbb{R}^{n})$, we say that $u$ is a  weak solution to the problem~\eqref{eq:main}, 
if $u$ is a weak solution to~\eqref{eq-def-weak} and, in addition, $u-g \in \mathcal{X}^{1,p}_{0}(\Omega)$.
\end{defi}
The existence and uniqueness of weak solutions to \eqref{eq:main} can be proved by standard monotonicity methods, see \cite[Appendix]{ByunSong}.

Since weak solutions do not need to exist in the case of an arbitrary measure datum, we shall consider solutions obtained as a limit of approximation (SOLA for short). They were introduced in~\cite{bgSOLA} for local $p$-Laplace-type problems. In the case of mixed local-nonlocal $p$-Laplace-type problems, the existence of SOLA is provided in \cite{ByunSong} upon the following definition.

\begin{defi}[SOLA]\label{def.sola}
Let $\mu \in \mathcal{M}_{b}(\mathbb{R}^{n})$, $g \in W^{1,p}_{\mathrm{loc}}(\mathbb{R}^{n}) \cap L^{p-1}_{sp}(\mathbb{R}^{n})$, $2-1/n<p<\infty$.
We say that a~function $u:\mathbb{R}^{n}\rightarrow\mathbb{R}$ satisfying $u \in W^{1,q}(\Omega)$ for any $\max\{p-1,1\} \le q < \min\left\{\frac{n(p-1)}{n-1},p\right\}$ 
is a SOLA to \eqref{eq:main}, under assumptions
\eqref{growth}--\eqref{s-p-range}, if it is a distributional solution,
i.e.,
\begin{equation*}
    \int_{\Omega}A(Du)\cdot D\varphi\dx + \int_{\mathbb{R}^{n}}\int_{\mathbb{R}^{n}}|u(x)-u(y)|^{p-2}(u(x)-u(y))(\varphi(x)-\varphi(y))K(x,y)\dx\dy = \int_{\mathbb{R}^{n}}\varphi\,\mathrm{d}\mu
\end{equation*}
holds for any $\varphi \in C^{\infty}_{0}(\Omega)$, and $u=g$ a.e. in $\mathbb{R}^{n}\setminus\Omega$. Moreover, there exists a sequence of weak solutions $ \{u_{j}\} \subset W^{1,p}(\mathbb{R}^{n})$ to the Dirichlet problems
\begin{equation*}
\left\{
\begin{aligned}
-\mathrm{div}\,A(Du_{j}) + \mathcal{L}u_{j} & = \mu_{j} & \text{ in } & \Omega\,, \\
u_{j} & = g_{j} & \text{ in } & \mathbb{R}^{n}\setminus \Omega\,,
\end{aligned}
\right.
\end{equation*}
in the sense of Definition~\ref{weak.sol}, such that $u_{j}$ converges to $u$ a.e. in $\mathbb{R}^{n}$ and in $W^{1,q}(\Omega)$.  
Here the sequence $\{\mu_{j}\}\subset C^{\infty}_{0}(\mathbb{R}^{n})$ satisfies $ \mu_{j} \xrightharpoonup{*} \mu $ in the {sense} of measures and \[\limsup_{j\rightarrow\infty} |\mu_{j}|(B) \le |\mu|(\overline{B})\quad \text{
for every ball }B \subset \mathbb{R}^{n}\,.\] 
The sequence $\{g_{j}\} \subset C^{\infty}_{0}(\mathbb{R}^{n})$ converges to $g$ in the following sense: for any ball $B_{r} \equiv B_{r}(x_0)$, it holds that
\begin{equation*}
    \lim_{j\rightarrow\infty}\lVert g_{j} - g \rVert_{W^{1,p}(B_{r})} = 0 \qquad \textrm{and} \qquad \lim_{j\rightarrow\infty}\tail(g_{j}-g;x_0,r) = 0\,.
\end{equation*}
\end{defi}
The lower bound for $p$ in the definition of SOLA is needed to ensure that $u$ possesses a distributional gradient that is locally integrable.
The existence of $\{ \mu _{j} \} \subset C_{0}^{\infty}(\mathbb{R}^{n})$ in Definition~\ref{def.sola} is guaranteed by the standard mollification argument.

\begin{rem}
\rm In Definition~\ref{weak.sol}, we may use the function space
\[ \left\{ f \in W^{1,p}(\Omega_{0}) \cap L^{p-1}_{sp}(\mathbb{R}^{n}) : f|_{\Omega} \in W^{1,p}_{0}(\Omega), \ f = 0 \text{ a.e. in }\mathbb{R}^{n}\setminus \Omega \right\} \]
for a fixed bounded open set $\Omega_{0} \Supset \Omega$, instead of $\mathcal{X}^{1,p}_{0}(\Omega)$, thereby allowing boundary data in $W^{1,p}(\Omega_{0}) \cap L^{p-1}_{sp}(\mathbb{R}^{n})$, see for instance \cite{BKL,GarLin23}. Accordingly, in Definition~\ref{def.sola} we can consider boundary data in $W^{1,p}_{\loc}(\Omega_{0}) \cap L^{p-1}_{sp}(\mathbb{R}^{n})$. However, in this paper we will use $\mathcal{X}^{1,p}_{0}(\Omega)$ for simplicity.
\end{rem}

\subsection{Regularity for auxiliary homogeneous equations}
Here we recall several regularity estimates for homogeneous mixed equations and homogeneous local equations.

We first consider $v \in W^{1,p}_{\loc}(\Omega) \cap W^{s,p}(\rn)$ being a weak solution to the homogeneous mixed equation
\begin{equation}
    \label{eq:homo-mixed}
    -\dv\,A(Dv)+\cL v=0\qquad\text{in }\ \Omega\,.
\end{equation}
The combination of \cite[Lemma 4.2]{GarKin22} 
and \cite[Remark 6.12]{G03} yields the following local sup-estimate for $v$. 
\begin{lem}\label{lem:bdd}
Let $v$ be a weak solution to \eqref{eq:homo-mixed} under assumptions \eqref{growth}--\eqref{kernel.growth} with $p > 1$ and $s \in (0,1)$.
Then there exists $c = c(\data)$ such that for any $B_{r} \Subset \Omega$ and $k \in \R$, it holds
\[ \sup_{B_{r/2}} |v-k| \leq c \mean{B_{r}} |v-k| \dx + \tail(v-k; r/2)\,. \]
\end{lem} 
We also have the Caccioppoli-type estimate for \eqref{eq:homo-mixed} provided e.g., in \cite[Lemma~3.4]{ByunSong}.
\begin{lem}
Let $v$ be a weak solution to \eqref{eq:homo-mixed} under assumptions \eqref{growth}--\eqref{kernel.growth} with $p>1$ and $s \in (0,1)$.
Then there exists  $c = c(\data)$ such that for any $B_{r} \Subset \Omega$ and $k \in \R$, it holds
\begin{align}\label{CCP1.ineq}
 \mean{B_{r/2}} |Dv|^{p} \dx +  \int_{B_{r/2}} \mean{B_{r/2}} \frac{|v(x) - v(y)|^{p}}{|x-y|^{n+ sp}} \dx \dy \leq c \left[ \mean{B_{r}} \frac{|v-k|}{r} \dx + \frac{1}{r}\tail(v-k; r/2) \right]^{p}\,.
\end{align}
\end{lem}

Using the argument similar to \cite{KuMiSi}, one can modify the local H\"older regularity result in  \cite{GarKin22} to involve the exponent $1$ (instead of $p$) in the first term on the right hand side (cf. also \cite[Lemma~3.3]{ByunSong}).

\begin{lem}\label{lem:mix-hol}
Let $v$ be a weak solution to \eqref{eq:homo-mixed} under assumptions \eqref{growth}--\eqref{kernel.growth} with $p>1$ and $s \in (0,1)$. 
Then $v$ is locally H\"older continuous in $\Omega$.
In particular, there exist constants $\alpha \in (0,1)$ and $c \geq 1$, both depending only on $\data$, such that for any concentric balls $B_{\rho} \subset B_{2r} \Subset \Omega$ and $k \in \R$, it holds
\begin{equation}\label{mix.hol}
\osc_{B_{\rho}} v \leq c \left( \frac{\rho}{r} \right)^{\alpha} \left[ \mean{B_{2r}} |v- k| \dx + \tail(v-k; r) \right]\,.
\end{equation}
\end{lem}

We shall also make use of some regularity properties of weak solutions to the homogeneous local equation
\begin{equation}\label{eq:homo-local}
    -\dv\, A(Dw)=0\qquad\text{in }\ \Omega\,.
\end{equation}
\begin{lem}\label{localeq.ed}
    Let $w \in W^{1,p}_{\loc}(\Omega)$ be a weak solution to \eqref{eq:homo-local} under assumptions \eqref{growth} with $p>1$, and $B_{\rho} \subset B_{r} \Subset \Omega$ are concentric balls.
    Then for any $q \in [1, \infty)$, there exist constants $\beta_w = \beta_{w}(\data) \in (0,1)$ and $c= c(\data, q) \ge 1$ such that
    \begin{equation}\label{ex-decay}
        \mean{B_{\rho}}|Dw-(Dw)_{B_{\rho}}|^{q} \dx \le c\left(\frac{\rho}{r}\right)^{\beta_w q}\mean{B_{r}}|Dw-(Dw)_{B_{r}}|^{q} \dx\,.
    \end{equation}
\end{lem}
In fact, \eqref{ex-decay} for $q=1$ is obtained in \cite[Theorem~3.3]{DuMiJFA2010} for $p\geq 2$ and in \cite[Theorem~3.1]{DuMiAJM2011} for $1<p<2$. 
Then the Campanato embedding implies the H\"older regularity of $Dw$, and so \eqref{ex-decay} holds for every $q \geq 1$.
One can find further excess decay estimates regarding $A(Dw)$ and $V(Dw)$ and under natural structural assumptions in \cite[Theorems~4.3~and~4.4]{BSY}. 

\section{Quantitative gradient H\"older regularity for homogeneous mixed equations}\label{sec:quant-hold}
With $m \in (0,1-s)$ being a free parameter, let us define 
\begin{equation}
    \label{bar a def}
    \bar a_1 \coloneqq \begin{cases}\frac{1-s}{p-1}\quad& \text{if } p\geq 2\,, \\
        m\quad& \text{if } 1<p<2\,,
    \end{cases}\qquad\qquad
      \bar a_2 \coloneqq \begin{cases}\frac{1}{p-1}\quad& \text{if } p\geq 2\, ,\\
        \frac{1-m(2-p)}{p-1}\quad &\text{if }1<p<2\,,
    \end{cases}
\end{equation}
Note that $\bar a_1 < \bar a_2$. Later in the proof of Theorem~\ref{theo:homo-mixed-hol}, we will find a lower bound of $\bar{a}_{2}-\bar{a}_{1}>0$, which plays an important role in the control of tails over a sequence of concentric balls.

\subsection{Comparison between homogeneous mixed problem and homogeneous local problem} Let $v \in W^{1,p}_{\loc}(\Omega) \cap W^{s,p}(\rn)$ be a weak solution to \eqref{eq:homo-mixed} under assumptions \eqref{growth}--\eqref{kernel.growth} with $p > 1$ and $s \in (0,1)$. We fix a ball $B_{r/4} = B_{r/4}(x_{0}) \Subset \Omega$ and introduce the weak solution $w$ to
\begin{equation}\label{eq:homo-local-sec4}
\left\{
\begin{aligned}
- \mathrm{div}\, A(Dw) & = 0 &\text{in }& B_{r/4}\,,\\
w & = v &\text{on }& \partial B_{r/4}\,.
\end{aligned} 
\right.
\end{equation}
Then by \cite[Theorem~6.1]{G03}, $w$ is a quasi-minimizer of the $p$-Dirichlet functional. In other words, there exists $c = c(\data)$ such that
\begin{equation}\label{p.min}
    \int_{B_{r/4}} |Dw|^{p} \dx \leq c \int_{B_{r/4}} |D \varphi|^p \dx \qquad \text{for any } \varphi \in w + W^{1,p}_{0}(B_{r/4})\,.
\end{equation}

We are to establish a basic comparison estimate between \eqref{eq:homo-mixed} and \eqref{eq:homo-local-sec4}. The approach is quite similar to \cite[Lemma~4.2]{DFM-mixed}. 

\begin{lem}\label{mix.loc.comp}
    Let $v \in W^{1,p}_{\loc}(\Omega) \cap W^{s,p}(\rn)$ be a weak solution to \eqref{eq:homo-mixed} and $w \in v + W^{1,p}_{0}(B_{r/4})$ be the weak solution to \eqref{eq:homo-local-sec4} under assumptions \eqref{growth}--\eqref{kernel.growth} with $p > 1$ and $s \in (0,1)$. Let $m \in (0, 1-s)$ be arbitrary. Recall $q_0=\max\{p-1,1\}$ and $\bar a_1,\bar a_2$ from~\eqref{bar a def}.    Then there exists $c=c(\data)$ such that
    \begin{align}\label{mix.loc.est1}
        \mean{B_{r/4}(x_0)} |Dv-Dw|^p \dx
        & \leq c r^{{\bar a_1 p}} \left( \mean{B_{r}(x_0)} |Dv|^{q_0} \dx \right)^{p/q_0} + c r^{{\bar a_2 p}} \left[ \frac{1}{r^{p'}} \tail(v-(v)_{B_{r}(x_0)}; x_0, r) \right]^{p}\,.
    \end{align}  
\end{lem}
\begin{proof}
We first observe the elementary estimates
\begin{align*} 
    \left[ \tail(v-(v)_{B_{r}}; r/2) \right]^{p-1}
    & \leq c \left[ \tail(v-(v)_{B_{r}}; r) \right]^{p-1} + c r^{p} \int_{B_{r} \setminus B_{r/2}} \frac{|v(x) - (v)_{B_{r}}|^{p-1}}{|x-x_{0}|^{n+sp}} \dx \notag \\
    & \leq c \left[ \tail(v-(v)_{B_{r}}; r) \right]^{p-1} + c r^{2 p-1-sp} \mean{B_{r}} \left(\frac{|v - (v)_{B_{r}}|}{r} \right)^{p-1} \dx \notag \\
    & \leq c \left[ \tail(v-(v)_{B_{r}}; r) \right]^{p-1} + c r^{2p-1-sp} \left(\mean{B_{r}} |Dv|^{q_0} \dx \right)^{(p-1)/q_0}
\end{align*}
and
\[ [\tail(v-(v)_{B_{r/2}}; r/2)]^{p-1} \leq c \left[\tail(v-(v)_{B_{r}};r)\right]^{p-1} + c r^{2p-1-sp} \left( \mean{B_{r}} |Dv|^{q_0} \dx \right)^{(p-1)/q_0}. \]
We will use these estimates in the proof without mentioning them.

We extend $w$ to $\rn$ by letting $w=v$ on $\rn \setminus B_{r/4}$.
Testing \eqref{eq:homo-mixed} and \eqref{eq:homo-local-sec4} against $h \coloneqq v-w$, we get
\begin{align}\label{mix.loc.1}
    \mean{B_{r/4}} |V(Dv)- V(Dw)|^{2} \dx
    & \leq c \mean{B_{r/4}} ( A(Dv) - A(Dw) )\cdot (Dv - Dw) \dx \notag \\
    & = - \frac{c}{|B_{r/4}|} \int_{\rn} \int_{\rn} |v(x) - v(y)|^{p-2}(v(x)-v(y))(h(x)-h(y))K(x,y) \dx \dy \notag\\
    & = - c\int_{B_{r/2}} \mean{B_{r/2}} |v(x) - v(y)|^{p-2}(v(x)-v(y))(h(x)-h(y))K(x,y) \dx \dy \notag \\
    & \quad - 2c \int_{\rn \setminus B_{r/2}} \mean{B_{r/2}} |v(x) - v(y)|^{p-2}(v(x)-v(y))(h(x)-h(y))K(x,y) \dx \dy \notag \\
    & \eqqcolon (\mathbf{I}) + (\mathbf{II})\,.
\end{align}
H\"older's inequality, \eqref{CCP1.ineq} with $k = (v)_{B_{r}}$, and the fact that $h=0$ outside $B_{r/4}$ give
\begin{align}\label{mix.loc.2}
    |(\mathbf{I})|
    & \leq c \left( \int_{B_{r/2}} \mean{B_{r/2}} \frac{|v(x)-v(y)|^{p}}{|x-y|^{n+sp}} \dx \dy \right)^{1-1/p}
    \left( \int_{B_{r/2}} \mean{B_{r/2}} \frac{|h(x)-h(y)|^{p}}{|x-y|^{n+sp}} \dx \dy \right)^{1/p} \notag \\
    \overset{\eqref{CCP1.ineq}, \eqref{Sob.ineq}}&{\leq}
    c r^{1-s} \left( \frac{1}{r} \mean{B_r}|v-(v)_{B_r}| \dx + \frac{1}{r}\tail(v-(v)_{B_{r}}; r/2) \right)^{p-1}
    \left( \mean{B_{r/4}} |Dh|^{p} \dx \right)^{1/p} \notag \\
    & \leq c r^{1-s} \left[ \left(\mean{B_{r}} |Dv|^{q_0} \dx\right)^{1/q_0} + \frac{1}{r}\tail(v-(v)_{B_{r}}; r) \right]^{p-1}
    \left( \mean{B_{r/4}} |Dv - Dw|^{p} \dx \right)^{1/p}\,.
\end{align}
On the other hand, we can estimate
\begin{align}\label{mix.loc.3}
    |(\mathbf{II})|
    & \leq c \int_{\rn \setminus B_{r/2}} \mean{B_{r/4}} \frac{\max\{|v(x)-(v)_{B_{r/2}}|, |v(y)-(v)_{B_{r/2}}|\}^{p-1} |h(x)|}{|x-y|^{n+sp}} \dx \dy \notag \\
    & \leq c \int_{\rn \setminus B_{r/2}} \mean{B_{r/4}} \frac{\max\{|v(x)-(v)_{B_{r/2}}|, |v(y)-(v)_{B_{r/2}}|\}^{p-1} |h(x)|}{|y-x_0|^{n+sp}} \dx \dy \notag \\
    & \leq c r^{-sp} \mean{B_{r/2}} |v-(v)_{B_{r/2}}|^{p-1} |h| \dx + c \int_{\rn \setminus B_{r/2}} \frac{|v(y)-(v)_{B_{r/2}}|^{p-1}}{|y-x_0|^{n+sp}} \dy \mean{B_{r/4}} |h| \dx \notag \\
    \overset{\eqref{mix.hol},\eqref{Tail}}&{\leq} c r^{-sp} \left( \mean{B_{r}} |v - (v)_{B_{r}}| \dx + \tail(v-(v)_{B_{r}};  r/2) \right)^{p-1} \mean{B_{r/4}} |h| \dx \notag \\
    & \quad + c r^{-p} [ \tail(v-(v)_{B_{r/2}}; r/2) ]^{p-1} \mean{B_{r/4}} |h| \dx \notag \\
    & \leq c \left[ r^{(1-s)p'} \left( \mean{B_{r}} |Dv|^{q_0} \dx \right)^{1/q_0} + \frac{1}{r} \tail(v-(v)_{B_{r}}; r)  \right]^{p-1} \mean{B_{r/4}} |Dv - Dw| \dx\,,
\end{align}
where in the last line, we have applied Poincar\'e's inequality. 
Combining \eqref{mix.loc.1}, \eqref{mix.loc.2}, and \eqref{mix.loc.3}, we have
\begin{align}\label{mix.loc.4}
\lefteqn{ \mean{B_{r/4}} |V(Dv) - V(Dw)|^{2} \dx } \notag \\
    & \leq c \left[ r^{1-s} \left( \mean{B_{r}} |Dv|^{q_0} \dx \right)^{(p-1)/q_0} + r \left[ \frac{1}{r^{p'}} \tail(v-(v)_{B_{r}}; r) \right]^{p-1} \right] \left( \mean{B_{r/4}} |Dv - Dw|^{p}\dx \right)^{1/p}\,.
\end{align}

For $p \geq 2$, applying the inequality
\[  |Dv - Dw|^{p} \leq c|V(Dv) - V(Dw)|^{2} \]
and Young's inequality to \eqref{mix.loc.4}, we obtain the desired estimate \eqref{mix.loc.est1}.

On the other hand, if $p <2$, then we proceed as
\begin{align*}
    & \mean{B_{r/4}}|Dv - Dw|^{p} \dx \\
    & \leq \left( \mean{B_{r/4}} |V(Dv) - V(Dw)|^{2} \dx \right)^{p/2} \left( \mean{B_{r/4}} (|Dv|+|Dw|)^{p} \dx \right)^{(2-p)/2} \\
    \overset{\eqref{p.min}, \eqref{CCP1.ineq}}&{\leq} c \left( \mean{B_{r/4}} |V(Dv) - V(Dw)|^{2} \dx \right)^{p/2} \left( \frac{1}{r} \mean{B_r}|v-(v)_{B_r}| \dx + \frac{1}{r}\tail(v-(v)_{B_{r}}; r/4) \right)^{p(2-p)/2} \\
    \overset{\eqref{mix.loc.4}}&{\leq} c \left[ r^{1-s} \left( \mean{B_{r}} |Dv|^{q_0} \dx \right)^{(p-1)/q_0} + r \left[ \frac{1}{r^{p'}} \tail(v-(v)_{B_{r}}; r) \right]^{p-1} \right]^{p/2} \left( \mean{B_{r/4}} |Dv - Dw|^{p} \dx \right)^{1/2} \\
    & \qquad \cdot \left[ \left( \mean{B_{r}} |Dv|^{q_0} \dx \right)^{(p-1)/q_0} + r\left[\frac{1}{r^{p'}}\tail(v-(v)_{B_{r}}; r) \right]^{p-1} \right]^{p(2-p)/[2(p-1)]}\,.
\end{align*}
Hence, further applications of Young's inequality give \eqref{mix.loc.est1} as follows:
\begin{align*}
    \mean{B_{r/4}}|Dv - Dw|^{p} \dx
    & \leq c r^{(1-s)p} \left( \mean{B_{r}} |Dv|^{q_0} \dx \right)^{p/q_0} + c r^{p'} \left[\frac{1}{r^{p'}}\tail(v-(v)_{B_{r}}; r) \right]^{p} \\
    & \quad + c r^{p} \left( \mean{B_{r}} |Dv|^{q_0} \dx \right)^{p(2-p)/q_0} \left[ \frac{1}{r^{p'}} \tail(v-(v)_{B_{r}};r ) \right]^{p(p-1)} \\
    & \leq c r^{mp} \left( \mean{B_{r}} |Dv|^{q_0} \dx \right)^{p/q_0} + c r^{(1-m(2-p)) p'} \left[\frac{1}{r^{p'}}\tail(v-(v)_{B_{r}}; r) \right]^{p}\,. \qedhere
\end{align*}
\end{proof}

\begin{rem}
\rm 
Note that the constant $c$ in \eqref{mix.loc.est1} is stable as $p \nearrow 2$.
Moreover, by taking $m$ close enough to $1-s$ in \eqref{mix.loc.est1} when $p \nearrow 2$, the estimate coincide  with the limit one (for $p=2$).
Later in Section~\ref{sec.pf}, we will choose the constant $m$  depending on the constant $\sigma \in (0, 1)$ given in Theorem~\ref{theo:intro} to deal with the sum of the tails, in a way that $m \searrow 0$ as $\sigma \nearrow 1$.
\end{rem}

\subsection{Excess decay estimates for homogeneous mixed equations} Combining Lemmas~\ref{localeq.ed} and \ref{mix.loc.comp} gives the following:
\begin{lem}\label{lem:ex1}
Let $v$ be a weak solution to \eqref{eq:homo-mixed} under assumptions \eqref{growth}--\eqref{kernel.growth} with $p > 1$ and $s \in (0,1)$. Recall $\beta_w = \beta_w(\data) \in (0,1)$ determined in Lemma~\ref{localeq.ed} and $q_0=\max\{p-1,1\}$. 
Let $m \in (0,1-s)$, and $\bar a_1,\bar a_2$ defined in~\eqref{bar a def}. 
For any $q \in [1,p]$, there exist $c_{1}=c_{1}(\data)$ and $c_{2}=c_{2}(\data)$ such that, whenever $B_{\rho} \subset B_{r} \Subset \Omega$ are concentric balls,
\begin{align*}
 \left( \mean{B_{\rho}} |Dv - (Dv)_{B_{\rho}}|^{q} \dx \right)^{1/q}  & \leq c_{1} \left( \frac{\rho}{r} \right)^{\beta_w} \left( \mean{B_{r}} |Dv - (Dv)_{B_{r}}|^{q} \dx \right)^{1/q}\\
& \quad + c_{2} \left( \frac{r}{\rho} \right)^{n/q_0} \left[ r^{\bar a_1} \left( \mean{B_{r}} |Dv|^{q_0} \dx \right)^{1/q_0} +  r^{\bar a_2}  \left[ \frac{1}{r^{p'}} \tail(v-(v)_{B_{r}}; r) \right] \right] \,.
\end{align*}
\end{lem}

We also need the following estimate for the tail, whose proof is similar to that of \cite[Theorem~3.8]{ByunSong}. 
Nonetheless, in our analysis, we need the precise power of $r$ multiplied by the average of $|Dv|$ in order to control this term later in Theorem~\ref{theo:homo-mixed-hol}.
It is not required in \cite[Theorem~3.8]{ByunSong} since the authors considered regularity estimates for solutions only.
\begin{lem}\label{lem:tail-dec-v}
Let $v$ be a weak solution to \eqref{eq:homo-mixed} under assumptions \eqref{growth}--\eqref{kernel.growth} with $p > 1$ and $s \in (0,1)$.
Recall $\alpha\in(0,1)$ from Lemma~\ref{lem:mix-hol}. 
There exists $c_{3}=c_{3}(\data)$ such that for any concentric balls $B_\rho(x_0)\subset B_r(x_0) \Subset \Omega$, it holds
    \begin{align}\label{tail-dec-v}
         \left[ \frac{1}{\rho^{p'}}\tail(v-(v)_{B_{\rho}};x_0,\rho) \right] \notag &\leq c_{3} \left[ 1 + r^{(1-s)p} \left( \frac{r}{\rho} \right)^{(1-\alpha)(p-1)+1} \right]^{1/(p-1)}  \left[ \frac{1}{r^{p'}}\tail(v- (v)_{B_{r}}; x_0,r) \right] \notag \\
        & \quad + c_{3} r^{[(1-s)p-1]/(p-1)} \left( \frac{r}{\rho} \right)^{[(1-\alpha)(p-1)+1]/(p-1)} \left( \mean{B_r} |Dv|^{q_0} \dx \right)^{1/q_0}\,.
    \end{align}
\end{lem}
\begin{proof} 
If $r/4 \leq \rho \leq r$, then 
\begin{align*}
    & \left[ \frac{1}{\rho^{p'}} \tail(v-(v)_{B_{\rho}}; \rho) \right]^{p-1} \\
    & \leq c \int_{\rn \setminus B_{r}} \frac{|v(x)-(v)_{B_{r}}|^{p-1}}{|x-x_0|^{n+sp}} \dx + c \int_{B_{r} \setminus B_{\rho}} \frac{|v(x)-(v)_{B_{r}}|^{p-1}}{|x-x_0|^{n+sp}} \dx + c \int_{\rn \setminus B_{\rho}} \frac{|(v)_{B_{r}}-(v)_{B_{\rho}}|^{p-1}}{|x-x_0|^{n+sp}} \dx \\
    & \leq c \left[ \frac{1}{r^{p'}} \tail(v-(v)_{B_{r}}; r) \right]^{p-1} + \frac{c r^{n}}{\rho^{n+sp}} \mean{B_{r}} |v-(v)_{B_{r}}|^{p-1} \dx + \frac{c}{\rho^{sp}}|(v)_{B_{\rho}}-(v)_{B_{r}}|^{p-1} \\
    & \leq c \left[ \frac{1}{r^{p'}} \tail(v-(v)_{B_{r}}; r) \right]^{p-1} + c r^{-sp} \left( \mean{B_{r}}|Dv|^{q_0} \dx \right)^{(p-1)/q_0}.
\end{align*}
Now for $0< \rho < r/4$, we split as follows:
\begin{align*}
    \left[ \frac{1}{\rho^{p'}} \tail(v-(v)_{B_{\rho}}; \rho) \right]^{p-1}
    & = \int_{\rn \setminus B_{r/4}} \frac{|v(x) - (v)_{B_{\rho}}|^{p-1}}{|x-x_0|^{n+sp}}\dx + \int_{B_{r/4} \setminus B_{\rho}} \frac{|v(x) - (v)_{B_{\rho}}|^{p-1}}{|x-x_0|^{n+sp}}\dx \eqqcolon (\mathbf{I}) + (\mathbf{II})\,.
\end{align*}
Lemma~\ref{lem:bdd} gives
\begin{align*}
    |(v)_{B_{\rho}} - (v)_{B_{r}}|
    & \leq c \sup_{B_{r/2}} |v - (v)_{B_{r}}| \leq c \mean{B_r}|v-(v)_{B_r}| \dx + c \tail(v-(v)_{B_{r}}; r)\,.
\end{align*}
Then
\begin{align*}
    (\mathbf{I})
    & \leq c \int_{\rn \setminus B_{r}} \frac{|v(x) - (v)_{B_{r}}|^{p-1}}{|x-x_0|^{n+sp}} \dx + c r^{(1-s)p-1} \left[ \mean{B_r}\frac{|v-(v)_{B_r}|}{r} \dx \right]^{p-1}  + c r^{(1-s)p-1} \left[ \frac{|(v)_{B_{\rho}} - (v)_{B_{r}}|}{r} \right]^{p-1} \\
    & \leq c \left[ \frac{1}{r^{p'}} \tail(v-(v)_{B_{r}};r) \right]^{p-1} + c r^{(1-s)p-1} \left[ \frac{1}{r} \mean{B_r}|v-(v)_{B_r}| \dx \right]^{p-1}\,.
\end{align*}
On the other hand, to estimate $(\mathbf{II})$ we use Lemma~\ref{lem:mix-hol} as follows:
\begin{align*}
    (\mathbf{II})
    & \leq c r^{(1-s) p} \int_{\rho}^{r/4} \frac{1}{t^p} \left( \osc_{B_{t}} v \right)^{p-1} \frac{{\rm d}t}{t} \\
    \overset{\eqref{mix.hol}}&{\leq} c r^{(1-s)p} \left[ \mean{B_r}|v-(v)_{B_r}| \dx + \tail(v-(v)_{B_{r}}; r) \right]^{p-1} \int_{\rho}^{r/4} \frac{1}{t^{p}}\left( \frac{t}{r}\right)^{\alpha(p-1)} \frac{{\rm d}t}{t} \\
    & \leq c r^{(1-s)p-1} \left( \frac{r}{\rho} \right)^{(1-\alpha)(p-1)+1} \left[ \frac{1}{r} \mean{B_r}|v-(v)_{B_r}|\dx + \frac{1}{r} \tail(v-(v)_{B_{r}};r) \right]^{p-1}\,.
\end{align*}
Summing up the above estimates, we conclude \eqref{tail-dec-v}.
\end{proof}

We further upgrade Lemma~\ref{lem:ex1} to the following decay estimate.
\begin{theo}\label{theo:homo-mixed-hol}
Let  $v$ be a weak solution to \eqref{eq:homo-mixed} under assumptions \eqref{growth}--\eqref{kernel.growth} with $p > 1$ and $s \in (0,1)$. 
Let $m \in (0,1-s)$, $\ve_{1} \in (0, (1-s)/p)$ and  $\bar a_1$, $\bar a_2$ be as in~\eqref{bar a def}. 
Then there exist constants $R_{0} \in (0,1)$, $\beta_{v} \in (0, \min\{ 1/(p-1), \beta_{w}\})$ and $c \ge 1$, all depending only on $\data,m,\ve_{1}$, such that for any concentric balls $B_{\rho} \subset B_{r} \Subset \Omega$ with $r \le R_{0}$ it holds
\begin{align*}
\lefteqn{ \mean{B_{\rho}} |Dv - (Dv)_{B_{\rho}}| \dx} \\
& \leq c \left( \frac{\rho}{r} \right)^{\beta_{v}} \left[ \mean{B_{r}} |Dv - (Dv)_{B_{r}}| \dx + r^{\bar a_1 - \ve_{1}} \left( \mean{B_{r}}|Dv|^{q_0} \dx \right)^{1/q_0}   + r^{\bar a_2 - \ve_{1}} \left[ \frac{1}{r^{p'}} \tail(v-(v)_{B_{r}}; r) \right] \right].
\end{align*}
Furthermore, for $p\geq 2$ the constants $R_0, \beta_{v}, c$ can be taken independent of $m$.
\end{theo}
\begin{proof} 
Recall that $\beta_w$ is the constant introduced in Lemma~\ref{localeq.ed}.
Without loss of generality, we may assume $\beta_{w} < 1/(p-1)$.
We take
\begin{equation}\label{ve-choice}
    \ve = \min \left\{ \frac{\bar{a}_{1}}{8}, \frac{\ve_{1}}{4}, 
    \frac{\beta_{w}}{4} \right\}\,.
\end{equation}
{Note that when $p\geq 2$, then $\bar a_1$ and (consequently) $\ve$ do not depend on $m$.} Moreover, the constants $c_{1},c_{2}$ are from Lemma~\ref{lem:ex1} and $c_{3}$ is from Lemma~\ref{lem:tail-dec-v}, where all the constants depend only on $\data$.
We take $\tau = \tau(\data, m, \ve_{1}) \in (0, 1/8)$ small enough to satisfy
\begin{equation}\label{tau-choice}
   2c_{1}\tau^{\beta_{w}} \le 1 \,, \quad 2\tau^{\ve}\leq 1\,, \quad 2^{1/(p-1)} c_{3} \tau^{4\varepsilon} \leq 1/2 \, ,
\end{equation}
and $R_{0}  \in (0,1)${, depending on $(\data, m, \ve_{1})$ when $1<p<2$ and on $(\data, \ve_{1})$ only when $p\geq 2$,} to satisfy 
\begin{equation}\label{Rzero}
    c_{2} R_{0}^{\ve} \leq \tau^{10n}\,, \quad c_{3}\tau^{-p'}R_{0}^{1-s} \leq 1/2 \,.
\end{equation}

We fix a ball $B_{r} = B_{r}(x_{0}) \Subset \Omega$ with $r \le R_{0}$ as in the statement, and set
\[ r_k \coloneqq \tau^{k} r\,, \quad \bk{k} \coloneqq B_{r_k}(x_0) \quad \text{for every }\ k = 0, 1, 2, \dots\,. \]
Lemma~\ref{lem:ex1} with $q= q_0$ and~\eqref{tau-choice} give
\begin{align}\label{hom.ex1}
\left( \mean{\bk{k+1}} |Dv - (Dv)_{\bk{k+1}}|^{q_0} \dx \right)^{1/q_0}
& \leq \frac{1}{2} \left( \mean{\bk{k}} |Dv - (Dv)_{\bk{k}}|^{q_0} \dx \right)^{1/q_0} + \tau^{9n} r_k^{\bar a_1 - \ve} \left(\mean{\bk{k}} |Dv|^{q_0} \dx \right)^{1/q_0} \notag \\
& \quad + \tau^{9n} r_k^{\bar a_2 - \ve} \left[ \frac{1}{r_k^{p'}} \tail(v-(v)_{\bk{k}}; r_{k}) \right]\,.
\end{align}
We now set
\begin{equation}\label{hom.ex2}
\lambda_{v} \coloneqq \tau^{-8n} r_{0}^{\bar a_1 - 4 \ve} \left( \mean{\bk{0}} |Dv|^{q_0} \dx \right)^{1/q_0} + \tau^{-8n} r_{0}^{\bar a_2 - 4 \ve} \left[\frac{1}{r_{0}^{p'}} \tail(v-(v)_{\bk{0}};  r_{0})\right]\,.
\end{equation}
To proceed further, we employ an induction argument to demonstrate that for any $k \in \N$
\begin{align}
r_{k}^{\bar{a}_{1}-4\varepsilon} \sum_{i=0}^{k} \left( \mean{\bk{i}} |Dv - (Dv)_{\bk{i}}|^{q_0} \dx \right)^{1/q_0} \leq \tau^{4n} \lambda_{v} \,, \label{hom.ex3} \\
r_{k}^{\bar a_2 - 4 \ve} \left[\frac{1}{r_{k}^{p'}} \tail(v-(v)_{\bk{k}}; r_{k})\right] \leq \lambda_{v}\,, \label{hom.ex4}
\end{align}
and
\begin{equation}\label{hom.ex5}
r_{k}^{\bar a_1 - 4 \ve} \left( \mean{\bk{k}} |Dv|^{q_0} \dx \right)^{1/q_0} \leq \lambda_{v}\,.
\end{equation}
By the choice of $\lambda_{v}$ in \eqref{hom.ex2}, one can directly check that
\begin{equation}\label{hom.ex6}
r_{0}^{\bar a_1 - 4 \ve}\left( \mean{\bk{0}} |Dv|^{q_0} \dx \right)^{1/q_0} + r_{0}^{\bar a_2 - 4 \ve} \left[\frac{1}{r_{0}^{p'}} \tail(v-(v)_{\bk{0}}; r_{0})\right] \leq \tau^{8n} \lambda_{v}\,.
\end{equation}
We now assume that there exists $k \in \N$ such that \eqref{hom.ex3}, \eqref{hom.ex4}, and \eqref{hom.ex5} hold for any $i \in \{0, 1, \dots, k\}$.
Then \eqref{ve-choice}, \eqref{tau-choice}, \eqref{hom.ex1}, and \eqref{hom.ex3}--\eqref{hom.ex6} imply 
\begin{align}\label{hom.ex7}
\lefteqn{ r_{k+1}^{\bar{a}_{1}-4\varepsilon}\sum_{i=0}^{k+1} \left( \mean{\bk{i}} |Dv-(Dv)_{\bk{i}}|^{q_0} \dx \right)^{1/q_0} } \notag \\
&\leq r_{k+1}^{\bar{a}_{1}-4\varepsilon}\left( \mean{\bk{0}} |Dv-(Dv)_{\bk{0}}|^{q_0} \dx \right)^{1/q_0} + \tau^{\bar{a}_{1}-4\ve}r_{k}^{\bar{a}_{1}-4\varepsilon} \sum_{i=0}^{k} \left( \mean{\bk{i}} |Dv-(Dv)_{\bk{i}}|^{q_0} \dx \right)^{1/q_0} \notag\\
&\quad + \tau^{9n} \sum_{i=0}^{k} r_i^{3 \ve} \left( \mean{\bk{i}} |Dv|^{q_0} \dx \right)^{1/q_0} \notag + \tau^{9n} \sum_{i=0}^{k} r_{i}^{\bar a_2 - \ve} \left[ \frac{1}{r_{i}^{p'}} \tail(v-(v)_{\bk{i}};  r_{i}) \right] \notag\\
&\leq 2\tau^{8n}\lambda_{v} + \tau^{4\varepsilon+4n} \lambda_{v} + 2\tau^{9n}\frac{r_{0}^{3\varepsilon}}{1- \tau^{3\ve }} \lambda_{v} \leq \tau^{4n} \lambda_{v}\,.
\end{align}
By Minkowski's inequality, we get
\begin{align*}
r_{k+1}^{\bar a_1 - 4 \ve}\left( \mean{\bk{k+1}} |Dv|^{q_0} \dx \right)^{1/q_0}
& \leq r_{k+1}^{\bar a_1 - 4 \ve} \sum_{i=0}^{k} \left( \mean{\bk{i+1}}|Dv-(Dv)_{\bk{i}}|^{q_0} \dx \right)^{1/q_0} + r_{k+1}^{\bar a_1 - 4 \ve} |(Dv)_{\bk{0}}| \notag \\
&  \leq \tau^{-n/q_0} r_{k}^{\bar{a}_{1}-4\varepsilon}\sum_{i=0}^{k} \left( \mean{\bk{i}}|Dv-(Dv)_{\bk{i}}|^{q_0} \dx\right)^{1/q_0} + r_{0}^{\bar a_1 - 4 \ve} |(Dv)_{\bk{0}}| \\
\overset{\eqref{hom.ex6},\eqref{hom.ex7}}&{\leq} \tau^{3n} \lambda_{v} + \tau^{8n} \lambda_{v} \notag 
 \leq \lambda_{v}\,.
\end{align*}
Note that if $p \geq 2$, then $\bar a_2 - \bar a_1 = s ( p-1)$.
On the other hand, if $p \in (1,2)$, then $\bar a_2 - \bar a_1 = (1-m)/(p-1) \in (s/(p-1), 1/(p-1))$.
Hence, for any $p >1$, it holds
\[ (\bar a_2- \bar a_1)(p-1) + (1-s)p - 1 \geq (1-s)(p-1) >0\,. \]
Using this fact and Lemma~\ref{lem:tail-dec-v}, we have
\begin{align*}
\lefteqn{ r_{k+1}^{\bar{a}_{2}-4\varepsilon}\left[\frac{1}{r_{k+1}^{p'}}\tail(v-(v)_{B^{k+1}};r_{k+1})\right] }\\
& \le c_{3}\tau^{\bar{a}_{2}-4\varepsilon}\left[1+r_{k}^{(1-s)p}\tau^{-(1-\alpha)(p-1)-1}\right]^{1/(p-1)}r_{k}^{\bar{a}_{2}-4\varepsilon}\left[\frac{1}{r_{k}^{p'}}\tail(v-(v)_{B^{k}};r_{k})\right] \\
& \quad + c_{3}\tau^{\bar{a}_{2}-4\varepsilon}r_{k}^{\bar{a}_{2}-\bar{a}_{1}+[(1-s)p-1]/(p-1)}\tau^{-[(1-\alpha)(p-1)+1]/(p-1)}r_{k}^{\bar{a}_{1}-4\varepsilon}\left(\mean{B^{k}}|Dv|^{q_{0}}\dx\right)^{1/q_{0}} \\
& \le 2^{1/(p-1)}c_{3}\tau^{4\varepsilon}\lambda_{v} + c_{3}\tau^{-p'}r_{k}^{1-s}\lambda_{v}  \le \lambda_{v}\,.
\end{align*}
In summary, by the induction argument \eqref{hom.ex3}, \eqref{hom.ex4}, and \eqref{hom.ex5} hold for every $k \in \N \cup \{0\}$.

This time Lemma~\ref{lem:ex1} with $q= 1$ and~\eqref{tau-choice} give
\begin{align*}
\mean{B^{i+1}}|Dv-(Dv)_{B^{i+1}}|\dx 
& \le \frac{1}{2}\mean{B^{i}}|Dv-(Dv)_{B^{i}}|\dx + cr_{i}^{\bar{a}_{1}}\left(\mean{B^{i}}|Dv|^{q_{0}}\dx\right)^{1/q_{0}} \\
& \quad + cr_{i}^{\bar{a}_{2}}\left[\frac{1}{r_{i}^{p'}}\tail(v-(v)_{B^{i}};r_{i})\right]
\end{align*}
for a constant $c$ ultimately depending on $(\data,m,\ve_{1})$ when $1<p<2$ and on $(\data,\ve_{1})$ only when $p \ge 2$. Iterating this inequality and then using \eqref{hom.ex4} and \eqref{hom.ex5}, we have
\begin{align*}
\mean{B^{k}}|Dv-(Dv)_{B^{k}}|\dx 
& \le \frac{1}{2^{k}}\mean{B^{0}}|Dv-(Dv)_{B^{0}}|\dx + c\sum_{i=1}^{k}\frac{1}{2^{k-i}}r_{i}^{\bar{a}_{1}}\left(\mean{B^{i}}|Dv|^{q_{0}}\dx\right)^{1/q_{0}} \\
& \quad + c\sum_{i=1}^{k}\frac{1}{2^{k-i}}r_{i}^{\bar{a}_{2}}\left[\frac{1}{r_{i}^{p'}}\tail(v-(v)_{B^{i}};r_{i})\right] \\
& \le \frac{c}{2^{k}}\left(\mean{B^{0}}|Dv-(Dv)_{B^{0}}|\dx + 2\lambda_{v}\sum_{i=1}^{k}2^{i}r_{i}^{4\varepsilon}\right) \\
& \le \frac{c}{2^{k}}\left(\mean{B^{0}}|Dv-(Dv)_{B^{0}}|\dx + \lambda_{v}\right)\,.
\end{align*}
Hence, there is $\beta_{v}$, depending on $(\data,m,\ve_{1})$ when $1<p<2$ and on $(\data,\ve_{1})$ only when $p\ge2$, such that
\[ \mean{\bk{k}}|Dv-(Dv)_{\bk{k}} \dx \leq c\left( \frac{r_{k}}{r_0}\right)^{\beta_{v}} \left( \mean{\bk{0}} |Dv - (Dv)_{\bk{0}}| \dx + \lambda_v \right) \, . \]
Now, for any $0 < \rho \leq r$, there exists $k_0 \in \N \cup \{0\}$ such that $r_{k_{0}+1} < \rho \leq r_{k_{0}}$. Then, by enlarging the radius of the ball, one has
\begin{align*}
\mean{B_{\rho}(x_0)} |Dv - (Dv)_{B_{\rho}(x_0)}| \dx
& \leq \mean{B_{\rho}(x_0)} |Dv - (Dv)_{\bk{k_0}}| \dx + |(Dv)_{B_{\rho}(x_0)} - (Dv)_{\bk{k_0}}| \\
& \leq c \mean{\bk{k_0}} |Dv - (Dv)_{\bk{k_0}}| \dx  \\
& \leq c \left( \frac{\rho}{r} \right)^{\beta_{v}} \left( \mean{B_{r}(x_0)} |Dv - (Dv)_{B_{r}}| \dx + \lambda_{v} \right)\,.
\end{align*}
Recalling the definition of $\lambda_{v}$ in \eqref{hom.ex2}, an elementary manipulation gives the desired estimate.
\end{proof}
As a direct consequence of Theorem~\ref{theo:homo-mixed-hol}, we have the following local Lipschitz bound for $v$.
\begin{coro}Under the assumptions of Theorem~\ref{theo:homo-mixed-hol},   we have that
\begin{equation}\label{homo-mixed-lip}
    \sup_{B_{r/2}(x_0)} |Dv| \leq c \left( \mean{B_{r}(x_0)} |Dv|^{q_0} \dx \right)^{1/q_0} + cr^{\bar{a}_{2} - \ve_{1}} \left[ \frac{1}{r^{p'}} \tail(v-(v)_{B_{r}(x_0)} ; x_0, r) \right]\,.
\end{equation}\end{coro}
Let us enhance the comparison estimate in the superquadratic case, i.e., for $p \geq 2$.
It is worth mentioning that in the next proposition, {due to Theorem~\ref{theo:homo-mixed-hol},} none of the constants depend on $m$ since we consider the superquadratic case only. 
\begin{prop}\label{prop:homo-mixed-hol2}
Let $v$ be a weak solution to \eqref{eq:homo-mixed} under assumptions \eqref{growth}--\eqref{kernel.growth} with $p\geq 2$ and $s \in (0,1)$, and let $\ve_{1} \in (0, (1-s)/p)$ be arbitrary. Assume $R_{0}=R_{0}(\data, \ve_{1})\in(0, 1)$ and
$\beta_{v} = \beta_{v}(\data, \ve_{1}) \in (0, \min\{ 1/(p-1), \beta_{w}\})$ are given by Theorem~\ref{theo:homo-mixed-hol}. 
Then there exist $c \ge 1$ and $\wt R_{0} \in (0,R_{0})$, both depending only on $\data$ and $\ve_{1}$, such that for any concentric balls $B_{\rho} \subset B_{r} \Subset \Omega$ with $r<\wt R_0$ we have
\begin{align*}
\lefteqn{ \mean{B_{\rho}} |A(Dv) - (A(Dv))_{B_{\rho}}| \dx }\\
& \leq c  \left( \frac{\rho}{r} \right)^{\beta_{v}} \Bigg[ \mean{B_{r}} |A(Dv) - (A(Dv))_{B_{r}}| \dx + r^{\ve_{1}} \mean{B_{r}}|Dv|^{p-1} \dx + r^{1 - (p-1)\ve_{1}} \left[ \frac{1}{r^{p'}} \tail(v-(v)_{B_{r}}; r) \right]^{p-1} \Bigg]\,.
\end{align*}
\end{prop}
\begin{proof}
We may assume that $\rho \leq r/2$ without loss of generality.  
We have
\begin{align} \label{homo-mixed-hol-1}
\lefteqn{ \mean{B_{\rho}} |A(Dv) - (A(Dv))_{B_{\rho}}| \dx \notag  \leq c \mean{B_{\rho}} |A(Dv) - A((Dv)_{B_{\rho}})| \dx } \notag \\
& \qquad\qquad \leq c \left(\sup_{B_{r/2}} |Dv|\right)^{p-2} \mean{B_{\rho}} |Dv - (Dv)_{B_{\rho}}| \dx \notag \\
& \qquad\qquad \leq c \left(\inf_{B_{r/2}} |Dv|\right)^{p-2} \mean{B_{\rho}} |Dv - (Dv)_{B_{\rho}}| \dx + c \left( \osc_{B_{r/2}} Dv \right)^{p-2} \mean{B_{\rho}} |Dv - (Dv)_{B_{\rho}}| \dx \notag \\
& \qquad \qquad \eqqcolon (\mathbf{I})  + (\mathbf{II})\, . 
\end{align}
First, using Theorem~\ref{theo:homo-mixed-hol} with $\varepsilon_{1}$ replaced by $\varepsilon_{1}/(p-1)$, one can estimate $ (\mathbf{I})$ as follows: 
    \begin{align*}
        (\mathbf{I}) & \leq c \left( \frac{\rho}{r} \right)^{\beta_{v}} \left(\inf_{B_{r/2}} |Dv|\right)^{p-2} \Bigg[ \mean{B_{r/2}} |Dv - (Dv)_{B_{r/2}}| \dx + r^{(1 - s - \ve_{1})/(p-1)} \left( \mean{B_{r}} |Dv|^{p-1} \dx \right)^{1/(p-1)} \\
        & \qquad \qquad \qquad \qquad \qquad \qquad \quad + r^{(1 - \ve_{1})/(p-1)} \left[ \frac{1}{r^{p'}} \tail(v-(v)_{B_{r}}; r) \right] \Bigg] \\
        & \leq c \left( \frac{\rho}{r} \right)^{\beta_{v}} \Bigg[ \mean{B_{r/2}} |Dv|^{p-2}|Dv - \xi| \dx + r^{(1 - s - \ve_{1})/(p-1)} \mean{B_{r}} |Dv|^{p-1} \dx \\
        & \qquad \qquad \qquad + r^{\ve_{1}(p-2)/(p-1)} \left( \mean{B_{r}} |Dv|^{p-1} \dx \right)^{(p-2)/(p-1)} r^{1/(p-1) - \ve_{1}} \left[ \frac{1}{r^{p'}} \tail(v-(v)_{B_{r}}; r) \right] \Bigg] \\
        & \leq c \left( \frac{\rho}{r} \right)^{\beta_{v}} \Bigg[ \mean{B_{r/2}} |Dv|^{p-2}|Dv - \xi| \dx + r^{\ve_{1}} \mean{B_{r}} |Dv|^{p-1} \dx + r^{1 - (p-1)\ve_{1}} \left[ \frac{1}{r^{p'}} \tail(v-(v)_{B_{r}}; r) \right]^{p-1} \Bigg]\,,
    \end{align*}
    where $\xi$ is an arbitrary $n$-dimensional vector.
    Taking $\xi \in \rn$ to satisfy $A(\xi) = (A(Dv))_{B_{r}}$, we have
    \begin{align*}
        (\mathbf{I}) & \leq c \left( \frac{\rho}{r} \right)^{\beta_{v}} \Bigg[ \mean{B_{r}} |A(Dv) - (A(Dv))_{B_{r}}| \dx + r^{\ve_{1}} \mean{B_{r}} |Dv|^{p-1} \dx + r^{1 - (p-1)\ve_{1}} \left[ \frac{1}{r^{p'}} \tail(v-(v)_{B_{r}}; r) \right]^{p-1} \Bigg] \,.
    \end{align*}
    Next, applying Theorem~\ref{theo:homo-mixed-hol} twice, 
    we infer that
    \begin{align*}
        (\mathbf{II}) 
        & \leq c \left( \osc_{B_{r/2}} Dv \right)^{p-2} \left( \frac{\rho}{r} \right)^{\beta_{v}} \Bigg[ \mean{B_{r}} |Dv - (Dv)_{B_{r}}| \dx + r^{(1-s)/(p-1)-\varepsilon_{1}} \left( \mean{B_{r}}|Dv|^{p-1} \dx \right)^{1/(p-1)}  \notag \\
        & \qquad\qquad\qquad\qquad\qquad\qquad\quad + r^{1/(p-1)-\varepsilon_{1}} \left[ \frac{1}{r^{p'}} \tail(v-(v)_{B_{r}};  r) \right] \Bigg]  \notag \\
        & \leq c \left( \frac{\rho}{r} \right)^{\beta_{v}} \Bigg[ \mean{B_{r}} |Dv - (Dv)_{B_{r}}|^{p-1} \dx + r^{1-s - (p-1)\ve_{1}} \mean{B_{r}}|Dv|^{p-1} \dx  \notag \\
        & \qquad \qquad \qquad  + r^{1 - (p-1)\ve_{1}} \left[ \frac{1}{r^{p'}} \tail(v-(v)_{B_{r}}; r) \right] ^{p-1} \Bigg]\,. \notag
    \end{align*}
From the choice of $\xi$, we have
\[ \mean{B_{r}} |Dv - (Dv)_{B_{r}}|^{p-1} \dx
    \leq c \mean{B_{r}} |Dv - \xi|^{p-1} \dx
    \leq c \mean{B_{r}} |A(Dv) - (A(Dv))_{B_{r}}| \dx\]
and therefore
\begin{align*}
(\mathbf{II})
        & \leq c \left( \frac{\rho}{r} \right)^{\beta_{v}} \Bigg[ \mean{B_{r}} |A(Dv) - (A(Dv))_{B_{r}}| \dx + r^{\ve_{1}} \mean{B_{r}} |Dv|^{p-1} \dx  \notag \\
        & \qquad \qquad \qquad   + r^{1 - (p-1)\ve_{1}} \left[ \frac{1}{r^{p'}} \tail(v-(v)_{B_{r}};  r) \right]^{p-1} \Bigg]\,.
\end{align*}
Connecting the estimates found for $(\mathbf{I})$ and $(\mathbf{II})$ to \eqref{homo-mixed-hol-1}, we conclude with the desired estimate.
\end{proof}

\section{Excess decay estimates for measure data mixed problems}\label{sec:ex-dec-measure-data}
\subsection{Comparison between measure data mixed problem and homogeneous mixed problem}\label{sec:comp-homo-mixed-meas-mixed}
Here, we establish several comparison estimates between \eqref{eq:main} and a homogeneous mixed problem.
These estimates, along with those in the previous section, imply excess decay estimates for \eqref{eq:main}.
We first consider $u$ being the weak solution to \eqref{eq:main} under the additional assumption
\begin{equation}\label{regular-data}
\mu\in C_0^\infty(\Omega)\qquad\text{and} \qquad g\in W^{1,p}(\rn)\,,
\end{equation}
which will be eventually removed in Section~\ref{est.sola} below.
 
For any ball $B_{r} \subset \Omega$, let us consider $v$ being a weak solution to
\begin{equation}\label{eq:homo-mixed-2}
\left\{
\begin{aligned}
-\dv A(Dv)+\cL v & = 0 &\text{in }& B_{r}\,, \\
v & = u &\text{in }& \rn\setminus B_{r}\,. 
\end{aligned} 
\right.
\end{equation} 
We will employ the basic comparison estimates between \eqref{eq:main} and \eqref{eq:homo-mixed-2} obtained in \cite[Lemmas~4.1 and 4.2]{ByunSong}. We stress that the second one does not hold true for $1<p\leq 2-1/n$ with $q\geq 1$.
Similar results for $p$ close to~$1$ can be found in \cite{NP23ARMA}.
\begin{lem}\label{lem:mix-comp1}
Let $u$ and $v$ be the weak solutions to \eqref{eq:main} and \eqref{eq:homo-mixed-2}, respectively, under assumptions \eqref{growth}--\eqref{kernel.growth} with $p>1$, $s\in(0,1)$ and~\eqref{regular-data}. Then there exists a constant $c = c(\data)$ such that, for any $d>0$ and $\gamma >1$, 
\begin{align*}
\mean{B_{r}} \frac{|V(Du) - V(Dv)|^{2}}{(d+|u-v|)^{\gamma}} \dx \leq c \frac{d^{1-\gamma}}{\gamma - 1} \left[ \frac{|\mu|(B_{r})}{r^{n}} \right]\,.
\end{align*}
\end{lem}

\begin{lem}\label{lem:mix-comp2}  Let 
$u$ and $v$ be the weak solutions to \eqref{eq:main} and \eqref{eq:homo-mixed-2}, respectively,  under assumptions \eqref{growth}--\eqref{s-p-range} and~\eqref{regular-data}.  For any $1 \leq q < \min \left\{  \frac{n(p-1)}{n-1},p  \right\}$, there exists a constant $c = c(\data,q)$ satisfying
    \begin{align*}
        \mean{B_{r}} |Du - Dv|^{q} \dx \leq c \left[ \frac{|\mu|(B_{r})}{r^{n-1}} \right]^{q/(p-1)} + c \chi_{\{p<2\}} \left[ \frac{|\mu|(B_{r})}{r^{n-1}} \right]^{q} \left( \mean{B_{r}} |Du|^{q} \dx \right)^{2-p}\,.
    \end{align*}
\end{lem}

To proceed further, we choose a ball $B_{Mr} = B_{Mr}(x_{0}) \subset \Omega$ with $M \ge 8$ to be chosen later and consider the following comparison map
\begin{equation}\label{eq:homo-mixed-3}
\left\{
\begin{aligned}
-\dv A(D v_{*})+\cL v_{*} & = 0 & \text{in }& B_{M r}\,, \\
v_{*} & = u & \text{in }& \rn\setminus B_{M r}\,.
\end{aligned} 
\right.
\end{equation}
We assume $Mr \leq R_{0}<1$, where $R_{0}$ is the radius from Theorem \ref{theo:homo-mixed-hol} (cf. \eqref{Rzero}).
Once we have Lemmas~\ref{lem:mix-comp1} and \ref{lem:mix-comp2}, we can also obtain the following comparison estimate, whose proof is exactly the same as that of \cite[Lemma~3]{KuMiguide}. 
\begin{lem}\label{lem:mix-comp3}
Let $u$, $v$, and $v_{*}$ be the weak solutions to \eqref{eq:main}, \eqref{eq:homo-mixed-2}, and \eqref{eq:homo-mixed-3}, respectively, under assumptions \eqref{growth}--\eqref{kernel.growth} with  $p \geq 2$, $s\in(0,1)$ and~\eqref{regular-data}. Let $M$ be as in \eqref{eq:homo-mixed-3}, and assume that there are constants $H \geq 1$ and $\Lambda >0$ such that
\begin{equation*}
\frac{|\mu|(B_{Mr})}{(Mr)^{n-1}} \leq H\Lambda^{p-1}
\end{equation*}
and 
\begin{equation*}
\frac{\Lambda}{H} \leq |Dv_{*}| \leq H \Lambda \qquad \text{in } B_{r}\,.
\end{equation*}
Then there exists $c=c(\data,M,H)$ such that 
\begin{equation*}
\mean{B_{r}} |Du - Dv| \dx \leq c \Lambda^{2-p} \left[ \frac{|\mu|(B_{Mr})}{(Mr)^{n-1}} \right]\,.
\end{equation*}
\end{lem}

\subsection{Comparison estimates involving the $A(\cdot)$-map 
in the case $p \geq 2$}\label{sec:alternatives}

This section is devoted to enhancing the comparison estimate for $A(Du)$ and $A(Dv)$.
To achieve this we will follow the principal idea of employing the two-scale alternatives as in \cite{AvKuMi}.
In the rest of this section, we fix the constant
\begin{equation}\label{Lambda}
    \Lambda \coloneqq \left( \mean{B_{r/M}} |Du|^{p-1} \dx \right)^{1/(p-1)}\,.
\end{equation}
Recall that as we focus on $p \geq 2$ in this section, the constants defined in \eqref{bar a def} are given by
\[ \bar a_1 = \frac{1-s}{p-1} \qquad\text{and}\qquad \bar a_2 = \frac{1}{p-1}\,. \]
We assume $\ve_{1} \in (0, (1-s)/p)$ and $0< M r \leq \wt R_{0}$ as in Proposition~\ref{prop:homo-mixed-hol2}.

We will consider first the \underline{degenerate alternative}
\begin{equation}\label{deg-alt}
    \mean{B_{Mr}} |A(Du)- (A(Du))_{B_{Mr}}| \dx \geq \theta |(A(Du))_{B_{r/M}}|\,,
\end{equation}
for a free parameter $\theta \in (0,1)$ whose value is to be chosen later in \eqref{theta-choice}. Next, we will focus on the \underline{non-degenerate alternative}
\begin{equation}\label{nondeg-alt}
    \mean{B_{Mr}} |A(Du)- (A(Du))_{B_{Mr}}| \dx < \theta |(A(Du))_{B_{r/M}}|\,.
\end{equation}
Observe that under \eqref{nondeg-alt}, for any $\rho \in [r/M, Mr]$ we have
\begin{align*}
    \lefteqn{ \mean{B_{\rho}}|Du|^{p-1} \dx
    \leq c\mean{B_{\rho}}|A(Du) - (A(Du))_{B_{r/M}}| \dx + c|(A(Du))_{B_{r/M}}| } \\
    & \leq c \mean{B_{\rho}} |A(Du) - (A(Du))_{B_{Mr}}| \dx + c \mean{B_{r/M}} |A(Du) - (A(Du))_{B_{Mr}}| \dx + c |(A(Du))_{B_{r/M}}| \\
    & \leq c M^{2n} \mean{B_{Mr}} |A(Du) - (A(Du))_{B_{Mr}}| \dx + c |(A(Du))_{B_{r/M}}| \\
    & \leq c (1 + \theta M^{2n}) |(A(Du))_{B_{r/M}}|\,.
\end{align*}
Assume that $\theta = \theta(n,M) \in (0,1)$ is small enough to satisfy
\begin{equation}\label{theta-choice}
    \theta M^{2n} \leq 1\,.
\end{equation}
Then, for any $\rho \in [r/M, Mr]$ and $\Lambda$ from~\eqref{Lambda}, there exists $c$ depending only on $\data$ such that
\begin{equation}\label{nondeg-Dubdd}
    \mean{B_{\rho}}|Du|^{p-1} \dx \leq c \Lambda^{p-1}\,.
\end{equation}
Finally, we will consider \underline{further alternatives}, i.e., we will assume that either for some $\theta_{*} \in(0,1)$
\begin{equation}\label{nondeg-large-mu}
    \theta_{*} \Lambda^{p-1} \leq \left[ \frac{|\mu|(B_{Mr})}{(Mr)^{n-1}} \right] + (Mr)^{1 - (p-1) \ve_{1}} \left[ \frac{1}{(Mr)^{p'}} \tail(u-(u)_{B_{Mr}}; Mr) \right]^{p-1}
\end{equation}
or
\begin{equation}\label{nondeg-small-mu}
    \theta_{*} \Lambda^{p-1} > \left[ \frac{|\mu|(B_{Mr})}{(Mr)^{n-1}} \right] + (Mr)^{1 - (p-1) \ve_{1}} \left[ \frac{1}{(Mr)^{p'}} \tail(u-(u)_{B_{Mr}}; Mr) \right]^{p-1}\,.
\end{equation}
Note that we will determine the constants $M$ and $\theta_{*}$ in Lemma~\ref{lem:Dv-low-up}.

 We start with the comparison estimate under the degenerate alternative assumption \eqref{deg-alt}.
\begin{lem}\label{lem:deg-comp}
Let $u$ and $v$ be the weak solutions to \eqref{eq:main} and \eqref{eq:homo-mixed-2}, respectively, under assumptions \eqref{growth}--\eqref{kernel.growth} with $p \geq 2$, $s\in(0,1)$ and \eqref{regular-data}. 
Assume further that $M$ is as in \eqref{eq:homo-mixed-3} and $\theta \in(0,1)$ is as in \eqref{theta-choice}. 
Under the degenerate alternative assumption \eqref{deg-alt}, there exists  $c=c(\data)$ such that for any $\ve_{2} \in (0,1)$ {and $B_{Mr}\subset\Omega$} we have
\begin{align*}
\mean{B_{r}}|A(Du) - A(Dv)| \dx
& \leq c \ve_{2} \left( M^{2n} + \frac{1}{\theta} \right) \mean{B_{Mr}} |A(Du) - (A(Du))_{B_{Mr}}| \dx \notag + c \ve_{2}^{2-p} M^{n-1} \left[ \frac{|\mu|(B_{Mr})}{(Mr)^{n-1}} \right]\,.
\end{align*} 
\end{lem}

\begin{proof}
From assumption \eqref{deg-alt}, we have
    \begin{align*}
        \mean{B_{r}}|Du|^{p-1} \dx
        & \leq c \mean{B_{r}} |A(Du) - (A(Du))_{B_{r/M}}| \dx + c |(A(Du))_{B_{r/M}}| \\
        & \leq c \mean{B_{r}} |A(Du) - (A(Du))_{B_{Mr}}| \dx + c |(A(Du))_{B_{Mr}} - (A(Du))_{B_{r/M}}| + c |(A(Du))_{B_{r/M}}| \\
        & \leq c \left( M^{2n} + \frac{1}{\theta} \right) \mean{B_{Mr}} |A(Du) - (A(Du))_{B_{Mr}}| \dx\,.
    \end{align*}
    Then by Lemmas~\ref{lem:mono} and \ref{lem:mix-comp2} we deduce
    \begin{align*}
        \mean{B_{r}} |A(Du) - A(Dv)| \dx
        & \leq c \mean{B_{r}} |Du|^{p-2}|Du - Dv| \dx + c \mean{B_{r}} |Du - Dv|^{p-1} \dx \\
        & \leq c \ve_{2} \mean{B_{r}} |Du|^{p-1} \dx + c \ve_{2}^{2-p} \mean{B_{r}} |Du-Dv|^{p-1} \dx \\
        & \leq c \ve_{2} \left( M^{2n} + \frac{1}{\theta} \right) \mean{B_{Mr}} |A(Du) - (A(Du))_{B_{Mr}}| \dx + c \ve_{2}^{2-p} \left[ \frac{|\mu|(B_{r})}{r^{n-1}} \right] \\
        & \leq c \ve_{2} \left( M^{2n} + \frac{1}{\theta} \right) \mean{B_{Mr}} |A(Du) - (A(Du))_{B_{Mr}}| \dx + c \ve_{2}^{2-p} M^{n-1} \left[ \frac{|\mu|(B_{Mr})}{(Mr)^{n-1}} \right]\,.
    \end{align*}
    This completes the proof.
\end{proof}

In the non-degenerate alternative case \eqref{nondeg-alt}, we first establish the following estimate under \eqref{nondeg-large-mu}.
\begin{lem}\label{lem:nondeg-large-mu}
Let $u$ and $v$ be the weak solutions to \eqref{eq:main} and \eqref{eq:homo-mixed-2},  respectively, under assumptions \eqref{growth}--\eqref{kernel.growth} with $p \geq 2$, $s\in(0,1)$ and \eqref{regular-data}. Suppose $M$ is as in \eqref{eq:homo-mixed-3} and $B_{Mr}\subset\Omega$. 
Assume \eqref{nondeg-alt} and \eqref{nondeg-large-mu} are satisfied for some $\theta_* \in(0,1)$.
Then there exists a constant $c=c(\data)$ such that
\begin{equation*}
    \mean{B_{r/M}} |A(Du) - A(Dv)| \dx
    \leq \frac{c}{\theta_{*}} \left[ \left[ \frac{|\mu|(B_{Mr})}{(Mr)^{n-1}} \right] + (Mr)^{1 - (p-1)\ve_{1}} \left[ \frac{1}{(Mr)^{p'}} \tail(u-(u)_{B_{Mr}};Mr) \right]^{p-1} \right]\,.
\end{equation*}
\end{lem}

\begin{proof}
By Lemma~\ref{lem:mono} and Young's inequality, we have
\begin{align*}
\mean{B_{r/M}} |A(Du) - A(Dv)| \dx
& \leq c \mean{B_{r/M}} |Du|^{p-2}|Du - Dv| \dx + c \mean{B_{r/M}} |Du - Dv|^{p-1} \dx \notag \\
& \leq c \mean{B_{r/M}} |Du|^{p-1} \dx + c \mean{B_{r/M}} |Du - Dv|^{p-1} \dx\,.
\end{align*}
Then \eqref{nondeg-large-mu} and Lemma~\ref{lem:mix-comp2} give the desired estimate.
\end{proof}

From now on, we will discuss the non-degenerate alternative case \eqref{nondeg-alt} with \eqref{nondeg-small-mu}.
Under these underlying assumptions, where the mean oscillation of $A(Du)$ is small, and the influence of the inhomogeneous term and the nonlocal tail term is sufficiently small, we can provide pointwise upper and lower bounds for $|Dv|$, as stated in the following lemma. This result is crucial for the application of the comparison estimate of Lemma~\ref{lem:mix-comp3}.
In the proof of the lemma below, we will determine the constants mentioned earlier: $M$ and $\theta_{*}$, ensuring that these constants depend only on $\data$ and $\ve_{1}$.
\begin{lem}\label{lem:Dv-low-up}
Let $v$ and $v_{*}$ be the weak solutions to \eqref{eq:homo-mixed-2} and \eqref{eq:homo-mixed-3}, respectively, under assumptions \eqref{growth}--\eqref{kernel.growth} with $p \geq 2$, $s\in(0,1)$ and \eqref{regular-data}, and let $\ve_{1} \in (0, (1-s)/p)$. 
Recall $\Lambda$ from~\eqref{Lambda} and $\wt R_0$ from Theorem~\ref{theo:homo-mixed-hol}.
There exist $R_{1}=R_{1}(\data, \ve_{1}) \leq \wt R_{0}$, $M=M(\data, \ve_{1}) \geq 8$, and $\theta_{*} = \theta_{*}(\data, \ve_{1}) \in (0,1)$ such that if \eqref{nondeg-alt}, \eqref{theta-choice}, and \eqref{nondeg-small-mu} hold for a ball $B_{Mr}\subset\Omega$ with $r \leq R_{1}$, then for some $H$ depending only on $\data$ and $\ve_{1}$ it holds
\begin{equation}\label{Dv-low-up-1}
\frac{\Lambda}{H} \leq |Dv_{*}| \leq  H \Lambda \quad \text{in } B_{r}
\end{equation}
and
\begin{equation}\label{Dv-low-up-2}
\frac{\Lambda}{H} \leq |Dv| \leq  H \Lambda \quad \text{in } B_{r/M}\,.
\end{equation}
\end{lem}
\begin{proof}
To show \eqref{Dv-low-up-1}, we first give the bound for the tail. Lemma~\ref{lem:mix-comp2} and \eqref{nondeg-small-mu} imply
\begin{align*}
\lefteqn{ (Mr)^{1 - (p-1) \ve_{1}} \left[ \frac{1}{(Mr)^{p'}} \tail(u-(v_{*})_{B_{Mr}}; Mr) \right]^{p-1} } \\
& \leq c (Mr)^{1 - (p-1) \ve_{1}} \left[ \left[ \frac{1}{(Mr)^{p'}} \tail(u-(u)_{B_{Mr}}; Mr) \right]^{p-1}  + |(u-v_{*})_{B_{Mr}}|^{p-1} \int_{\rn \setminus B_{Mr}} \frac{1}{|x-x_{0}|^{n+sp}} \dx \right] \notag \\
& \leq c (Mr)^{1 - (p-1) \ve_{1}} \left[ \frac{1}{(Mr)^{p'}} \tail(u-(u)_{B_{Mr}}; Mr) \right]^{p-1} + c(Mr)^{(1-s)p - (p-1)\ve_{1}} \left(\mean{B_{Mr}}|Du-Dv_{*}| \dx \right)^{p-1} \nonumber \\
& \le c\left[(Mr)^{1-(p-1)\varepsilon_{1}} + (Mr)^{(1-s)p+(p-1)\varepsilon_{1}} \right]\theta_{*}\Lambda^{p-1}\,.
\end{align*}
Choosing $R_{1} = R_{1}(\data,\varepsilon,M)$ so small that 
\begin{equation}\label{R1.cond1}
MR_{1} \le \widetilde{R}_{0} \le 1\,,
\end{equation}
we have
\begin{align}\label{Dv-low-up-3}
(Mr)^{1- (p-1) \ve_{1}} \left[ \frac{1}{(Mr)^{p'}} \tail(u-(v_{*})_{B_{Mr}}; Mr) \right]^{p-1}
\leq c \theta_{*} \Lambda^{p-1}\,.
\end{align}
    
To obtain the upper bound of $|Dv_{*}|$ in \eqref{Dv-low-up-1}, we apply \eqref{homo-mixed-lip} to see
\begin{align}\label{Dv-low-up-4}
\sup_{B_{Mr/2}} |Dv_{*}|^{p-1}
& \leq c \mean{B_{Mr}} |Dv_{*}|^{p-1} \dx + c (Mr)^{1 - (p-1) \ve_{1}} \left[ \frac{1}{(Mr)^{p'}} \tail(v_{*}-(v_{*})_{B_{Mr}}; Mr) \right]^{p-1} \notag \\
& \leq c \mean{B_{Mr}} |Du|^{p-1} \dx + c \mean{B_{Mr}}|Du-Dv_{*}|^{p-1} \dx \notag \\
& \qquad + c (Mr)^{1 - (p-1) \ve_{1}} \left[ \frac{1}{(Mr)^{p'}} \tail(u-(v_{*})_{B_{Mr}}; Mr) \right]^{p-1}\,.
\end{align}
Combining \eqref{Dv-low-up-3} and \eqref{Dv-low-up-4}, and then using  Lemma~\ref{lem:mix-comp2}, \eqref{nondeg-Dubdd}, and \eqref{nondeg-small-mu}, we obtain the upper bound in \eqref{Dv-low-up-1}:
\begin{equation}\label{Dv*.upper}
\sup_{B_{Mr/2}} |Dv_{*}| \leq c \Lambda\,.
\end{equation}
    
As for the lower bound of $|Dv_{*}|$ in \eqref{Dv-low-up-1}, recalling \eqref{Lambda}, we choose $c_{u} = c_{u}(\data)$ such that
\begin{equation*} 
\frac{\Lambda^{p-1}}{c_{u}} \le (|A(Du)|)_{B_{r/M}} \le c_{u}\Lambda^{p-1}\,.
\end{equation*}
Then we have
\begin{align*}
(|A(Dv_{*})|)_{B_{r/M}} & \ge (|A(Du)|)_{B_{r/M}} - \left|(|A(Dv_{*})|)_{B_{r/M}}-(|A(Du)|)_{B_{r/M}}\right| \\
& = (|A(Du)|)_{B_{r/M}} - \left|\mean{B_{r/M}}(|A(Dv_{*})|-|A(Du)|)\dx\right| \\
& \ge \frac{\Lambda^{p-1}}{c_{u}} - \mean{B_{r/M}}|A(Du)-A(Dv_{*})|\dx\,,
\end{align*}
where the last integral is estimated as
\begin{align*}
\mean{B_{r/M}}|A(Du)-A(Dv_{*})|\dx 
& \le cM^{2n}\mean{B_{Mr}}|Du-Dv_{*}|^{p-1}\dx + cM^{2n}\Lambda^{p-2}\mean{B_{Mr}}|Du-Dv_{*}|\dx \\
& \le cM^{2n}\left[\frac{|\mu|(B_{Mr})}{(Mr)^{n-1}}\right] + cM^{2n}\Lambda^{p-2}\left[\frac{|\mu|(B_{Mr})}{(Mr)^{n-1}}\right]^{1/(p-1)} \\
& \le \tilde{c}_{*}M^{2n}\left[\theta_{*} + \theta_{*}^{1/(p-1)}\right]\Lambda^{p-1}\,
\end{align*}
with $\tilde{c}_{*} = \tilde{c}_{*}(\data)$. 
Choosing $\theta_{*} = \theta_{*}(\data,M)$ such that
\begin{equation}\label{theta*.cond}
\tilde{c}_{*}M^{2n}\left[ \theta_{*} + \theta_{*}^{1/(p-1)} \right] \le \frac{1}{2c_{u}}\,,
\end{equation}
we have
\begin{equation*}
\mean{B_{r/M}}|A(Du)-A(Dv_{*})|\dx \le \frac{\Lambda^{p-1}}{2c_{u}}
\end{equation*}
and therefore
\begin{equation*}
(|A(Dv_{*})|)_{B_{r/M}} \ge \frac{\Lambda^{p-1}}{2c_{u}}\,.
\end{equation*}
In particular, there exists $\bar{x}_{*} \in B_{r/M}$ such that
\begin{equation*}
|A(Dv_{*}(\bar{x}_{*}))| \ge \frac{\Lambda^{p-1}}{2c_{u}}\,. 
\end{equation*}
In light of Proposition~\ref{prop:homo-mixed-hol2} and \eqref{R1.cond1}, we have
\begin{align*}
\osc_{B_{r}} A(Dv_{*})
& \leq \frac{c}{M^{\beta_{v}}} \Bigg[ \mean{B_{Mr}} |A(Dv_{*})-(A(Dv_{*}))_{B_{Mr}}| \dx + (Mr)^{\ve_{1}}\mean{B_{Mr}} |Dv_{*}|^{p-1} \dx  \notag \\
& \qquad\qquad\qquad + (Mr)^{1-(p-1)\ve_{1}} \left[ \frac{1}{(Mr)^{p'}} \tail(v_{*}-(v_{*})_{B_{Mr}}; Mr) \right]^{p-1} \Bigg] \notag  \\
& \leq \frac{c}{M^{\beta_{v}}} \Bigg[ \mean{B_{Mr}} |Du - Dv_{*}|^{p-1} \dx  + \mean{B_{Mr}} |Du|^{p-1} \dx  \notag\\
& \qquad\qquad\qquad + (Mr)^{1-(p-1)\ve_{1}} \left[ \frac{1}{(Mr)^{p'}} \tail(u-(v_{*})_{B_{Mr}}; Mr) \right] ^{p-1} \Bigg] \notag \\
& \leq \frac{c_{*}}{M^{\beta_{v}}} \Lambda^{p-1}
\end{align*}
for a constant $c_{*}=c_{*}(\data,\varepsilon_{1})$. 
Choosing $M = M(\data,\varepsilon_{1})$ such that
\begin{equation}\label{M.cond1}
\frac{c_{*}}{M^{\beta_{v}}} \le \frac{1}{4c_{u}}\,, 
\end{equation}
we obtain the lower bound in \eqref{Dv-low-up-1}:
\begin{equation*}
|A(Dv_{*}(x))| \ge |A(Dv_{*}(\bar{x}_{*}))| - \osc_{B_{r}}A(Dv_{*}) \ge \frac{\Lambda^{p-1}}{4c_{u}} \qquad \forall\; x\in B_{r}\,.
\end{equation*}
    
Now we show \eqref{Dv-low-up-2}. 
For this, we choose $R_{1} = R_{1}(\data,\varepsilon_{1},M) \le \tilde{R}_{0}$ such that
\begin{equation}\label{R1.cond2}
M^{n+p-1}R_{1}^{(1-s)p + (p-1)\varepsilon_{1}} \le 1\,.
\end{equation} 
Using \eqref{nondeg-Dubdd}, \eqref{nondeg-small-mu}, \eqref{R1.cond2} and Lemma~\ref{lem:mix-comp2}, we have
\begin{align}\label{Dv-low-up-tail}
\lefteqn{ r^{1 - (p-1) \ve_{1}} \left[ \frac{1}{r^{p'}} \tail(u-(v)_{B_{r}}; r) \right]^{p-1} } \notag \\
& \leq c r^{1 - (p-1) \ve_{1}} \left[ \frac{1}{r^{p'}} \tail(u-(u)_{B_{r}}; r) \right]^{p-1} + c r^{(1-s)p - (p-1) \ve_{1}} \mean{B_{r}}|Du-Dv|^{p-1} \dx  \notag\\
& \leq c M^{n+p -1} r^{(1 - s) p - (p-1) \ve_{1}} \mean{B_{Mr}} |Du|^{p-1} \dx \notag \\
& \quad + c (Mr)^{1 - (p-1) \ve_{1}} \left[ \frac{1}{(Mr)^{p'}} \tail(u-(u)_{B_{Mr}};Mr) \right]^{p-1} + c M^{n-1}r^{(1-s)p-(p-1)\varepsilon_{1}} \left[ \frac{|\mu|(B_{Mr})}{(Mr)^{n-1}} \right] \notag \\
& \leq  c\Lambda^{p-1}\,.
\end{align}
Similarly as in \eqref{Dv*.upper}, we use Lemma~\ref{lem:mix-comp2}, together with  \eqref{homo-mixed-lip}, \eqref{nondeg-Dubdd}, \eqref{nondeg-small-mu}, and \eqref{Dv-low-up-tail}, to obtain the upper bound
\begin{align*} 
\sup_{B_{r/2}} |Dv|^{p-1}
& \leq c \mean{B_{r}} |Du|^{p-1} \dx + c \mean{B_{r}}|Du-Dv|^{p-1} \dx + c r^{1 - (p-1) \ve_{1}} \left[ \frac{1}{r^{p'}} \tail(u-(v)_{B_{r}}; r) \right]^{p-1} \notag \\
& \leq c \left( 1 + M^{n-1}\theta_{*} \right) \Lambda^{p-1} 
\overset{\eqref{theta*.cond}}{\le} c\Lambda^{p-1}\,.
\end{align*}
Now we turn our attention to the lower bound in \eqref{Dv-low-up-2}. By completely similar calculations as above, we have
\begin{equation*}
(|A(Dv)|)_{B_{r/M}} \ge \frac{\Lambda^{p-1}}{c_{u}} - \mean{B_{r/M}}|A(Du)-A(Dv)|\dx\,,
\end{equation*}    
and this time we estimate
\begin{align*}
\mean{B_{r/M}}|A(Du)-A(Dv)|\dx 
& \le cM^{n}\mean{B_{r}}|Du-Dv|^{p-1}\dx + cM^{n}\Lambda^{p-2}\mean{B_{r}}|Du-Dv|\dx \\
& \le cM^{2n-1}\left[\frac{|\mu|(B_{Mr}}{(Mr)^{n-1}}\right] + cM^{(np-1)/(p-1)}\Lambda^{p-2}\left[\frac{|\mu|(B_{Mr})}{(Mr)^{n-1}}\right]^{1/(p-1)} \\
& \le \tilde{c}\left[M^{2n-1}\theta_{*} + \left(M^{np-1}\theta_{*}\right)^{1/(p-1)}\right]\Lambda^{p-1}\,
\end{align*}
for a constant $\tilde{c} = \tilde{c}(\data)$. Choosing $\theta_{*}=\theta_{*}(\data,M)$ so small that
\begin{equation}\label{theta*.cond2}
\tilde{c}\left[M^{2n-1}\theta_{*} + \left(M^{np-1}\theta_{*}\right)^{1/(p-1)}\right] \le \frac{1}{2c_{u}}\,, 
\end{equation}
we have
\begin{equation*}
(|A(Dv)|)_{B_{r/M}} \ge \frac{\Lambda^{p-1}}{2c_{u}}\,.
\end{equation*}
In turn, there exists $\bar{x} \in B_{r/M}$ such that
\begin{equation*}
|A(Dv(\bar{x}))| \ge \frac{\Lambda^{p-1}}{2c_{u}}\,.
\end{equation*}
We use Proposition~\ref{prop:homo-mixed-hol2} and Lemma~\ref{lem:mix-comp2}. Then \eqref{nondeg-small-mu}, \eqref{Dv-low-up-tail} and \eqref{R1.cond2} imply
\begin{align*}
\osc_{B_{r/M}} A(Dv)
& \leq \frac{c}{M^{\beta_{v}}} \Bigg[ \mean{B_{r}} |A(Dv) - (A(Dv))_{B_{r}}| \dx + r^{\ve_{1}} \mean{B_{r}} |Dv|^{p-1} \dx  \notag \\
& \qquad \qquad \qquad  + r^{1-(p-1)\ve_{1}} \left[ \frac{1}{r^{p'}} \tail(v-(v)_{B_{r}}; r) \right]^{p-1} \Bigg] \notag \\
& \leq \frac{c}{M^{\beta_{v}}} \Bigg[ \mean{B_{r}} |Du - Dv|^{p-1} \dx  +  \mean{B_{r}} |Du|^{p-1} \dx  \notag \\
& \qquad \qquad \qquad + r^{1 -(p-1)\ve_{1}} \left[ \frac{1}{r^{p'}} \tail(u-(v)_{B_{r}}; r) \right]^{p-1} \Bigg] \notag \\
& \leq \frac{c}{M^{\beta_{v}}} \left[ M^{n-1}\theta_{*} + 1 \right]\Lambda
\overset{\eqref{theta*.cond2}}{\le} \frac{c}{M^{\beta_{v}}}\Lambda^{p-1}
\end{align*}
whenever $r \le R_{1}$, where $c = c(\data,\varepsilon_{1})$. 
Choosing $M = M(\data,\varepsilon_{1})$ such that
\begin{equation}\label{M.cond2}
\frac{c}{M^{\beta_{v}}} \le \frac{1}{4c_{u}}\,,
\end{equation}
we conclude with the lower bound in \eqref{Dv-low-up-2}.
\end{proof}

\begin{rem}\rm
The procedure of fixing the parameters $M,\theta,\theta_{*}$ and $R_{1}$ is summarized as follows. 
We first choose $M = M(\data,\varepsilon_{1}) \ge 8$ as in Lemma~\ref{lem:Dv-low-up} such that \eqref{M.cond1} and \eqref{M.cond2} hold. Then we accordingly fix $\theta = \theta(\data,\varepsilon_{1})$ as in \eqref{theta-choice}.
Finally, we determine $\theta_{*}=\theta_{*}(\data,\varepsilon_{1})$ satisfying \eqref{theta*.cond}, \eqref{theta*.cond2} and $R_{1}=R_{1}(\data,\varepsilon_{1})$ satisfying \eqref{R1.cond1}, \eqref{R1.cond2}. As a result, we have chosen the values of $M,\theta,\theta_{*}$ and $R_{1}$ depending only on $\data$ and $\ve_{1}$, and we will use these values of the parameters in the rest of the paper.
\end{rem}

\begin{lem}\label{lem:nondeg-comp} 
Let $u$ and $v$ be the weak solutions to \eqref{eq:main} and \eqref{eq:homo-mixed-2},  respectively, under assumptions \eqref{growth}--\eqref{kernel.growth} with 
$p \geq 2$, $s\in(0,1)$ and \eqref{regular-data}, and let $\ve_{1} \in (0, (1-s)/p)$. 
Assume that {$B_{Mr}\subset\Omega$, while} the constants $R_{1}$, $M$, $\theta$, and $\theta_{*}$, depending on $\ve_1$ and $\data$ only, are chosen as in Lemma~\ref{lem:Dv-low-up} and \eqref{theta-choice} under the assumptions \eqref{nondeg-alt} and \eqref{nondeg-small-mu}. 
Then there exists a constant $c=c(\data,\ve_{1})$ such that
\begin{align*}
\mean{B_{r/M}}|A(Du)-A(Dv)| \dx \leq c \left[ \frac{|\mu|(B_{Mr})}{(Mr)^{n-1}} \right]\,.
\end{align*} 
\end{lem}

\begin{proof}
Using Lemma~\ref{lem:mono} and \eqref{Dv-low-up-2}, we infer that
\begin{align*}
\mean{B_{r/M}} |A(Du) - A(Dv)| \dx
& \leq c \mean{B_{r/M}} |Dv|^{p-2}|Du - Dv| \dx + c \mean{B_{r/M}} |Du - Dv|^{p-1} \dx \\
& \leq c\Lambda^{p-2}\mean{B_{r/M}}|Du-Dv|\dx + c\mean{B_{r/M}}|Du-Dv|^{p-1}\dx.
\end{align*}
Note that we also have \eqref{Dv-low-up-1}. Thus, we can estimate the first and second integrals on the right-hand side by using Lemmas~\ref{lem:mix-comp3} and \ref{lem:mix-comp2}, respectively, which leads to the conclusion.
\end{proof}

\subsection{Excess decay estimates for for $2-1/n<p<\infty$}

Note that when $p \geq 2$, all the constants $M$, $\theta$, $\theta_{*}$, and $R_{1}$ were chosen depending only on $\data$ and $\ve_{1} \in (0, (1-s)/p)$ in Section~\ref{sec:alternatives}.

Summing up Lemmas~\ref{lem:deg-comp}, \ref{lem:nondeg-large-mu}, and \ref{lem:nondeg-comp}, we have the following comparison lemma.
\begin{lem}\label{lem:A-comp}
Let $u$ and $v$ be the weak solutions to \eqref{eq:main} and \eqref{eq:homo-mixed-2}, respectively, under assumptions \eqref{growth}--\eqref{kernel.growth} with 
$p \geq 2$, $s\in(0,1)$ and\eqref{regular-data}. 
Assume that  $\ve_{1} \in (0, (1-s)/p)$ and the constants $R_{1}$ and $M$ are chosen as in Lemma~\ref{lem:Dv-low-up} {and $B_{Mr}\subset\Omega$}. 
Then there exists $c=c(\data, \ve_{1})$ such that for any $\ve_{2} \in (0,1)$ we have
\begin{align*} 
\mean{B_{r/M}} |A(Du) - A(Dv)| \dx
& \leq c \ve_{2} \mean{B_{Mr}} |A(Du) - (A(Du))_{B_{Mr}}| \dx + c \ve_{2}^{2-p} \left[ \frac{|\mu|(B_{Mr})}{(Mr)^{n-1}} \right] \notag \\
& \quad + c (Mr)^{1 - (p-1)\ve_{1}} \left[ \frac{1}{(Mr)^{p'}} \tail(u-(u)_{B_{Mr}};Mr) \right]^{p-1}\,.
\end{align*} 
\end{lem}

Now we are ready to prove the excess decay estimate for $A(Du)$ when $p \geq 2$.
\begin{lem}\label{lem:ADuexcess-pge2}
Under the assumptions in Lemma~\ref{lem:A-comp}, for any $\rho \in (0, Mr]$, we have
\begin{align}\label{ADuexcess-pge2}
\mean{B_{\rho}}|A(Du) - (A(Du))_{B_{\rho}}| \dx 
& \leq c_{4} \left( \frac{\rho}{Mr} \right)^{\beta_{v}} \left[\mean{B_{M r}} |A(Du) - (A(Du))_{B_{Mr}}| \dx + (Mr)^{\ve_{1}} \mean{B_{Mr}} |A(Du)| \dx \right] \notag \\
& \quad + c_{5}\left(\frac{Mr}{\rho}\right)^{n}(Mr)^{1-(p-1)\ve_{1}}\left[ \frac{1}{(Mr)^{p'}} \tail(u-(u)_{B_{Mr}}; Mr) \right]^{p-1}  \nonumber \\
& \quad + c_{5} \left( \frac{Mr}{\rho} \right)^{\eta} \left[ \frac{|\mu|(B_{Mr})}{(Mr)^{n-1}} \right] 
\end{align} 
for some $c_{4}$, $c_{5}$ and $\eta$ depending only on $\data$ and $\ve_{1}$.
\end{lem}
\begin{proof}
If $\rho \in (r/M, Mr]$, then \eqref{ADuexcess-pge2} follows directly from the fact that $M$ depends only on $\data$ and $\ve_{1}$. Hence, we assume $\rho \in (0, r/M]$.
Applying Proposition~\ref{prop:homo-mixed-hol2} with concentric balls $B_{\rho} \subset B_{r/M}$, we have
\begin{align*}
\lefteqn{ \mean{B_{\rho}}|A(Du)-(A(Du))_{B_{\rho}}|\dx } \\
& \le c\mean{B_{\rho}}|A(Dv)-(A(Dv))_{B_{\rho}}|\dx + c\mean{B_{\rho}}|A(Du)-A(Dv)|\dx \\
& \le c\left(\frac{M\rho}{r}\right)^{\beta_{v}}\Bigg[ \mean{B_{r/M}}|A(Dv)-(A(Dv))_{B_{r/M}}|\dx + \left(\frac{r}{M}\right)^{\ve_{1}}\mean{B_{r/M}}|Dv|^{p-1}\dx \\
& \qquad\qquad\qquad\qquad + \left(\frac{r}{M}\right)^{1-(p-1)\ve_{1}}\left[\left(\frac{M}{r}\right)^{p'}\tail(v-(v)_{B_{r/M}};r/M)\right]^{p-1} \Bigg] \\
& \quad + c\left(\frac{r}{M\rho}\right)^{n}\mean{B_{r/M}}|A(Du)-A(Dv)|\dx.
\end{align*}
We then estimate the tail term as
    \begin{align*}
        & \left(\frac{r}{M}\right)^{1 - (p-1) \ve_{1}} \left[ \left(\frac{M}{r}\right)^{p'} \tail(v-(v)_{B_{r/M}}; r/M) \right]^{p-1}
         = \left(\frac{r}{M}\right)^{1 - (p-1) \ve_{1}} \left[ \left(\frac{M}{r}\right)^{p'} \tail(u-(v)_{B_{r/M}}; r/M) \right]^{p-1}  \\
        & \leq c \mean{B_{r/M}} |Du - Dv|^{p-1} \dx + c \left(\frac{r}{M}\right)^{1 - (p-1) \ve_{1}} \left[ \left(\frac{M}{r}\right)^{p'} \tail(u-(u)_{B_{r/M}}; r/M) \right]^{p-1} \\
        & \leq c \mean{B_{r/M}} |Du - Dv|^{p-1} \dx + c (Mr)^{(1-s)p - (p-1) \ve_{1}} \mean{B_{M r}} |Du|^{p-1} \dx \\
        & \qquad + c (Mr)^{1 - (p-1) \ve_{1}} \left[ \frac{1}{(Mr)^{p'}} \tail(u-(u)_{B_{Mr}}; Mr) \right]^{p-1}\,,
\end{align*}
thereby getting
\begin{align*}
\lefteqn{ \mean{B_{\rho}}|A(Du)-(A(Du))_{B_{\rho}}|\dx } \\
& \le c\left(\frac{\rho}{Mr}\right)^{\beta_{v}}\Bigg[\mean{B_{r/M}}|A(Du)-(A(Du))_{B_{r/M}}|\dx + (Mr)^{\ve_{1}}\mean{B_{Mr}}|Du|^{p-1}\dx \\
& \qquad\qquad\qquad\qquad + (Mr)^{1-(p-1)\ve_{1}}\left[\frac{1}{(Mr)^{p'}}\tail(u-(u)_{B_{Mr}};Mr)\right]^{p-1} \Bigg] \\
& \quad + c\left(\frac{Mr}{\rho}\right)^{n}\mean{B_{r/M}}|A(Du)-A(Dv)|\dx\,.
\end{align*}
To estimate the last integral, we use Lemma~\ref{lem:A-comp} with the choice $\ve_{2} = (\rho/Mr)^{n+\beta_{v}}$, which gives
\begin{align*}
\lefteqn{ \mean{B_{r/M}}|A(Du)-A(Dv)|\dx }\\
& \le c\left(\frac{\rho}{Mr}\right)^{n+\beta_{v}}\mean{B_{Mr}}|A(Du)-(A(Du))_{B_{Mr}}|\dx + c\left(\frac{Mr}{\rho}\right)^{n(p-1)+\beta_{v}(p-2)}\left[\frac{|\mu|(B_{Mr})}{(Mr)^{n-1}}\right]  \\
& \quad + c(Mr)^{1-(p-1)\ve_{1}}\left[\frac{1}{(Mr)^{p'}}\tail(u-(u)_{B_{Mr}};Mr)\right]^{p-1}\,.
\end{align*}
Combining the last two displays, we obtain the desired estimate \eqref{ADuexcess-pge2}.
\end{proof}
As for the case $2-1/n < p < 2$, we have the following estimate resulting from Theorem~\ref{theo:homo-mixed-hol} and Lemma~\ref{lem:mix-comp2}.

\begin{lem}
Let $u$ be the weak solution to \eqref{eq:main} under assumptions \eqref{growth}--\eqref{kernel.growth} with $2-1/n < p < 2$, $s\in(0,1)$ and \eqref{regular-data}. For any $m \in (0,1-s)$ and $\varepsilon_{1} \in (0,(1-s)/p)$, there exist $R_{0} \in (0,1)$ and $\beta_{v} \in (0,\min\{1/(p-1),\beta_{w})$, $c_{6},c_{7} \ge 1$, all depending only on $\data$, $m$ and $\varepsilon_{1}$, such that for all concentric balls $B_{\rho}\subset B_{r} \subset \Omega$ with $r \le R_{0}$ it holds
\begin{align}\label{Duexcess}
\mean{B_{\rho}}|Du-(Du)_{B_{\rho}}|\dx 
& \le c_{6}\left(\frac{\rho}{r}\right)^{\beta_{v}}\Bigg[\mean{B_{r}}|Du-(Du)_{B_{r}}|\dx + r^{m-\varepsilon_{1}}\mean{B_{r}}|Du|\dx \nonumber \\
& \qquad\qquad\qquad + r^{[1-m(2-p)]/(p-1)-\varepsilon_{1}}\left[\frac{1}{r^{p'}}\tail(u-(u)_{B_{r};r})\right] \Bigg] \nonumber \\
& \quad + c_{7}\left(\frac{r}{\rho}\right)^{n}\left[\frac{|\mu|(B_{r})}{r^{n-1}}\right]^{1/(p-1)} + c_{7}\left(\frac{r}{\rho}\right)^{n}\left[\frac{|\mu|(B_{r})}{r^{n-1}}\right]\left(\mean{B_{r}}|Du|\dx\right)^{2-p}\,.
\end{align}
\end{lem} 

\subsection{A tail estimate for $2-1/n<p<\infty$}
We provide the following tail estimate for $u$.

\begin{lem}\label{lem:u-tail} 
Let $u$ be the weak solution to \eqref{eq:main} under assumptions \eqref{growth}--\eqref{s-p-range} and \eqref{regular-data}. Recall $q_0=\max\{p-1,1\}$. Then there exists $c_{8}=c_{8}(\data)$ such that, for every concentric balls $B_{\rho} \subset B_{r} \subset \Omega$,
\begin{align}\label{u-tail}
\left[ \frac{1}{\rho^{p'}} \tail(u-(u)_{B_{\rho}};\rho) \right]^{q_0}
& \leq c_{8} \left[ 1 + r^{(1-s)p}\left( \frac{r}{\rho} \right)^{(1-\alpha)(p-1) + 1} \right]^{q_{0}/(p-1)} \left[ \frac{1}{r^{p'}} \tail(u-(u)_{B_{r}};r) \right]^{q_0} \notag \\
& \quad + c_{8} \left[ r^{(1-s)p-1} \left( \frac{r}{\rho} \right)^{n+p} \right]^{q_{0}/(p-1)} \left( \mean{B_{r}}|Du|^{q_0} \dx + \left[ \frac{|\mu|(B_{r})}{r^{n-1}} \right]^{q_0/(p-1)} \right)\,.
\end{align}
\end{lem}
\begin{proof}
    First, we consider the case $p \geq 2$, i.e., $q_0 = p-1$.
    By a straightforward calculation, we have
    \begin{align*}
        & \left[ \frac{1}{\rho^{p'}} \tail(u-(u)_{B_{\rho}}; \rho) \right]^{p-1} \\
        & \leq c \left[ \frac{1}{\rho^{p'}} \tail(v-(v)_{B_{\rho}}; \rho) \right]^{p-1}  + c \int_{\rn \setminus B_{\rho}}\frac{|(u-v)_{B_{\rho}}|^{p-1}}{|x-x_{0}|^{n+sp}} \dx
        + c \int_{\rn \setminus B_{\rho}} \frac{|u(x)-v(x)|^{p-1}}{|x-x_{0}|^{n+sp}} \dx \\
        &  \leq c \left[ \frac{1}{\rho^{p'}} \tail(v-(v)_{B_{\rho}}; \rho) \right]^{p-1} + \frac{c}{\rho^{sp}} |(u-v)_{B_{\rho}}|^{p-1} + \frac{c}{\rho^{n+sp}} \int_{B_{r}} |u-v|^{p-1} \dx \\
        & \leq c \left[ \frac{1}{\rho^{p'}} \tail(v-(v)_{B_{\rho}}; \rho) \right]^{p-1} + c r^{-sp} \left( \frac{r}{\rho} \right)^{n+sp} \mean{B_{r}} |u-v|^{p-1} \dx
    \end{align*}
    and similarly
    \begin{align*}
        \left[ \frac{1}{r^{p'}} \tail(v-(v)_{B_{r}}; r) \right]^{p-1}
        & \leq c \left[ \frac{1}{r^{p'}} \tail(u-(u)_{B_{r}}; r) \right]^{p-1} + c r^{-sp} \mean{B_{r}} |u-v|^{p-1} \dx.
    \end{align*}
    Then by Lemma \ref{lem:tail-dec-v},
    \begin{align*}
         \left[ \frac{1}{\rho^{p'}}\tail(v-(v)_{B_{\rho}};\rho) \right]^{p-1}
         &\leq c \left[ 1 + r^{(1-s)p} \left( \frac{r}{\rho} \right)^{(1-\alpha)(p-1)+1} \right]  \left[ \frac{1}{r^{p'}}\tail(v- (v)_{B_{r}}; r) \right]^{p-1} \notag \\
        & \quad + c r^{(1-s)p-1} \left( \frac{r}{\rho} \right)^{(1-\alpha)(p-1)+1} \left[ \frac{1}{r} \mean{B_r}|v-(v)| \dx \right]^{p-1} \notag \\
        &\leq c \left[ 1 + r^{(1-s)p} \left( \frac{r}{\rho} \right)^{(1-\alpha)(p-1)+1} \right]  \left[ \frac{1}{r^{p'}}\tail(u- (u)_{B_{r}}; r) \right]^{p-1} \notag \\
        & \quad + c r^{-sp} \left( \frac{r}{\rho} \right)^{(1-\alpha)(p-1)+1} \left[ \mean{B_{r}} |u-v|^{p-1} \dx + \mean{B_{r}} |u-(u)_{B_{r}}|^{p-1} \dx \right] \,.
    \end{align*}
Finally, Sobolev--Poincar\'e inequality and Lemma \ref{lem:mix-comp2} give
\begin{align*}
\left[ \frac{1}{\rho^{p'}} \tail(u-(u)_{B_{\rho}}; \rho) \right]^{p-1}
& \leq c \left[ 1 + r^{(1-s)p} \left( \frac{r}{\rho} \right)^{(1-\alpha)(p-1)+1} \right]  \left[ \frac{1}{r^{p'}}\tail(u- (u)_{B_{r}}; r) \right]^{p-1} \\
& \quad + c r^{(1-s)p-1} \left( \frac{r}{\rho} \right)^{n+p} \left ( \left[ \frac{|\mu|(B_{r})}{r^{n-1}} \right] + \mean{B_{r}} |Du|^{p-1} \dx \right)\, .
\end{align*}
In the last line, we have used $n+sp < n + p$, $(1-\alpha)(p-1)+1 <p$, and $R_{1} < 1$. 
This completes the proof of the lemma for the case $p \geq 2$.
The proof for the case $p<2$ is completely similar; the only difference arises from the appearance of an additional term due to Lemma~\ref{lem:mix-comp2}.
Nonetheless, one can handle it by employing Young's inequality. 
\end{proof}

\subsection{Estimates for SOLA}\label{est.sola}
So far, all the estimates in this section were obtained for weak solutions to \eqref{eq:main} under the additional assumption on the regularity of $\mu$ and $g$ \eqref{regular-data}.  
Now we consider a SOLA $u$ to \eqref{eq:main} and let $ \{u_{j}\}$ be an approximating sequence for $u$ being solutions to problems with measures $\mu_j$ and boundary data $g_j$, as described in Definition~\ref{def.sola}. 
Once we have the a priori estimates \eqref{ADuexcess-pge2}, \eqref{Duexcess} and \eqref{u-tail} for $u_{j}$, then the convergence properties of $u_{j}$, $\mu_{j}$ and $g_{j}$, along with a standard approximation argument (see for instance the proof of \cite[Lemma~6.1]{ByunSong}), show that the same estimates are valid for $u$ as well. We summarize the conclusion as follows. 
\begin{coro}
Let $u$ be a SOLA to \eqref{eq:main} under assumptions \eqref{growth}--\eqref{s-p-range}.
Then the estimates \eqref{ADuexcess-pge2}, \eqref{Duexcess}, and \eqref{u-tail} continue to hold for $u$.
\end{coro}

\section{Proof of Theorem~\ref{theo:main}}\label{sec.pf}

\subsection{The case $p \ge 2$}
In this section, we will prove Theorem~\ref{theo:main} for the case when $p \geq 2$.

\begin{proof}[Proof of Theorem \ref{theo:main} {\it (i)}] 
Let us recall important constants which we will use throughout this proof: 
$M$ and $R_{1}$ determined in Lemma~\ref{lem:Dv-low-up}, $\beta_{v} \in (0,1)$ in Proposition~\ref{prop:homo-mixed-hol2} (originally in Theorem~\ref{theo:homo-mixed-hol}), $c_{4},c_{5}$, and $\eta$ in Lemma~\ref{lem:ADuexcess-pge2}, while $c_8$ in Lemma~\ref{lem:u-tail}. 
Under the condition $p \geq 2$, in light of Lemma~\ref{lem:mono}, the estimate from Lemma~\ref{lem:u-tail} will be applied in the following form: for every concentric balls $B_{\rho} \subset B_{r} \subset \Omega$ we have
\begin{align}\label{u-tail-2}
        \left[ \frac{1}{\rho^{p'}} \tail(u-(u)_{B_{\rho}};\rho) \right]^{p-1}
        & \leq c_{8} \left[ 1 + r^{(1-s)p}\left( \frac{r}{\rho} \right)^{(1-\alpha)(p-1) + 1} \right] \left[ \frac{1}{r^{p'}} \tail(u-(u)_{B_{r}};r) \right]^{p-1} \notag \\
        & \quad + c_{8} r^{(1-s)p-1} \left( \frac{r}{\rho} \right)^{n+p} \left( \mean{B_{r}}|A(Du)| \dx + \left[ \frac{|\mu|(B_{r})}{r^{n-1}} \right] \right)\,.
    \end{align}
Here, the constant $c_{8}$ in the above estimate is distinct from the one in Lemma~\ref{lem:u-tail}, since $A(z)$ is not exactly $|z|^{p-2}z$.
However, we are going to denote it using the same notation for the sake of convenience.

Let $\sigma\in(0,1)$ be an arbitrary number as in the statement. We observe that it is enough to provide a proof under the condition 
\begin{equation}\label{sigma.add}
\sigma > \max\{1-(1-s)p,0\}\,,
\end{equation}
as then the remaining case follows immediately. 
Hence, in the rest of this section, we assume \eqref{sigma.add}. 

We now fix
\[ \ve_{1} = \frac{1}{4p} \min \{ 1-\sigma, 1-s \} \]
and then choose $\tau \in (0, 1/M)$ to satisfy
\begin{equation}\label{sec6-tau}
2c_{4} \tau^{\beta_{v}} \leq 1\,, \quad 2\tau^{\ve_{1}/p} \leq 1\,, \quad 8 c_{8} \tau^{\sigma} \leq 1\,, \quad 2 \tau^{\sigma + (1-s)p -1} \le 1\,.
\end{equation}
Note that as we have chosen $\ve_{1}$ depending only on $\data$ and $\sigma$ and considering the case $p \geq 2$, all the constants introduced so far also depend only on $\data$ and $\sigma$. 
Further, we take some small radius $\bar R \in (0, R_{1})$, where $R_{1}$ comes from Lemma~\ref{lem:Dv-low-up}, to satisfy 
\begin{equation}\label{R.choice.ge2}
c_{5}\bar R^{\ve_{1}/4} \leq \tau^{10n} \,, \quad 8 c_{8} \bar R^{\sigma + (1-s)p -1} \leq \tau^{n+p}\,.
\end{equation}
For any $r \in (0, \bar R/2]$ such that $B_{2r}(x_{0}) \subset \Omega$, let us consider a sequence of concentric balls
\[ \bk{k} = B_{r_{k}}(x_0), \quad r_{k} = \tau^{k} r, \quad \text{for every} \quad k = 0, 1, 2, \dots\, . \]
Then Lemma~\ref{lem:ADuexcess-pge2} gives 
\begin{align*} 
\mean{\bk{i+1}} |A(Du) - (A(Du))_{\bk{i+1}}| \dx 
\overset{\eqref{sec6-tau}_{1},\eqref{R.choice.ge2}_{1}}&{\leq} \frac{1}{2} \mean{\bk{i}}|A(Du) - (A(Du))_{\bk{i}}| \dx + \tau^{9n} r_{i}^{3 \ve_{1}/4} \mean{\bk{i}} |A(Du)| \dx \notag \\
& \qquad + \tau^{9n} r_{i}^{\sigma + \ve_{1}} \left[ \frac{1}{r_{i}^{p'}} \tail(u-(u)_{\bk{i}};r_{i}) \right]^{p-1} + c_{5} \tau^{-\eta} \left[ \frac{|\mu|(\bk{i})}{r_{i}^{n-1}} \right]\,.
\end{align*}
Choose any natural numbers $k_{1} < k_{2}$. Summing up \eqref{ADu-pot-1} over $i \in \{k_{1},\ldots k_{2}-1\}$, we obtain
\begin{align*}
\lefteqn{ \sum_{i=k_{1}+1}^{k_{2}}\mean{B^{i}}|A(Du)-(A(Du))_{B^{i}}|\dx }  \\
& \le \frac{1}{2}\sum_{i=k_{1}}^{k_{2}-1}\mean{B^{i}}|A(Du)-(A(Du))_{B^{i}}|\dx + \tau^{9n}\sum_{i=k_{1}}^{k_{2}-1}r_{i}^{3\varepsilon_{1}/4}\mean{B^{i}}|A(Du)|\dx  \\
& \quad + \tau^{9n}\sum_{i=k_{1}}^{k_{2}-1}r_{i}^{\sigma+\varepsilon_{1}}\left[\frac{1}{r_{i}^{p'}}\tail(u-(u)_{B^{i}};r_{i})\right]^{p-1} + c_{5}\tau^{-\eta}\sum_{i=k_{1}}^{k_{2}-1}\left[\frac{|\mu|(B^{i})}{r_{i}^{n-1}}\right]\,,
\end{align*}
which in turn gives
\begin{align}\label{ADu-pot-1}
\lefteqn{ \sum_{i=k_{1}}^{k_{2}}\mean{B^{i}}|A(Du)-(A(Du))_{B^{i}}|\dx } \nonumber \\
& \le 2\mean{B^{k_{1}}}|A(Du)-(A(Du))_{B^{0}}|\dx + 2\tau^{9n}\sum_{i=k_{1}}^{k_{2}-1}r_{i}^{3\varepsilon_{1}/4}\mean{B^{i}}|A(Du)|\dx \nonumber \\
& \quad + 2\tau^{9n}\sum_{i=k_{1}}^{k_{2}-1}r_{i}^{\sigma+\varepsilon_{1}}\left[\frac{1}{r_{i}^{p'}}\tail(u-(u)_{B^{i}};r_{i})\right]^{p-1} + 2c_{5}\tau^{-\eta}\sum_{i=k_{1}}^{k_{2}-1}\left[\frac{|\mu|(B^{i})}{r_{i}^{n-1}}\right]\,.
\end{align}
    
To proceed further, we set
\begin{align}\label{ADu-pot-2}
\lambda_{A}^{p-1} \coloneqq \tau^{-8n} \mean{\bk{0}} |A(Du)| \dx + \tau^{-8n} r_{0}^{\sigma} \left[ \frac{1}{r_{0}^{p'}} \tail(u-(u)_{\bk{0}} ; r_{0})\right]^{p-1} + 2 c_{5} \tau^{-8n-\eta} \cI^{|\mu|}_{1}(x_0, 2r_{0})\,,
\end{align}
which is finite by the assumptions of Theorem~\ref{theo:main}.
Note that, since $\tau < 1/4$, it holds
\begin{align}\label{Riesz-sum}
\sum_{i=0}^{\infty} \left[ \frac{|\mu|(\bk{i})}{r_{i}^{n-1}} \right] & \leq \sum_{i=1}^{\infty} \tau^{-(n-1)} (- \log \tau) \int_{r_{i}}^{r_{i-1}} \frac{|\mu|(B_{\rho}(x_{0}))}{\rho^{n-1}} \frac{{\rm d} \rho}{\rho} + 2^{(n-1)} \log 2 \int_{r_{0}}^{2 r_{0}} \frac{|\mu|(B_{\rho}(x_0))}{\rho^{n-1}} \frac{{\rm d} \rho}{\rho} \nonumber \\
&\leq  \tau^{-n} \cI^{|\mu|}_{1} (x_0, 2r_{0}) \leq \frac{\tau^{7n+\eta}}{2 c_{5}} \lambda_{A}^{p-1}\,.
\end{align}

We claim that for every integer $k \geq 0$ the following estimates hold true
\begin{align}\label{ADu-pot-ind1}
\sum_{i=0}^{k} \mean{\bk{i}} |A(Du) - (A(Du))_{B^{i}}| \dx \leq \tau^{4n} \lambda_{A}^{p-1}\,,
\end{align}
\begin{align}\label{ADu-pot-ind2}
\sum_{i=0}^{k} r_{i}^{\sigma} \left[ \frac{1}{r_{i}^{p'}} \tail(u-(u)_{B^{i}} ; r_{i}) \right]^{p-1} \leq \lambda_{A}^{p-1}\,,
\end{align}
and
\begin{align}\label{ADu-pot-ind3}
\mean{\bk{k}}|A(Du)| \dx \leq \lambda_{A}^{p-1}\,.
\end{align}
The proof of the claim is somewhat different from that of Theorem~\ref{theo:homo-mixed-hol} since we have to deal with the sum of excess functionals, taking into account additional terms related to $\mu$. 
By definition \eqref{ADu-pot-2}, we have \eqref{ADu-pot-ind2}, \eqref{ADu-pot-ind3} for $k=0$ and
\begin{align*}
\mean{\bk{0}} |A(Du) - (A(Du))_{\bk{0}} | \dx \leq  2 \mean{\bk{0}} |A(Du)| \dx \leq 2 \tau^{8n} \lambda_{A}^{p-1}\,.
\end{align*}
We now assume that there is $k \in \N$ such that \eqref{ADu-pot-ind1}--\eqref{ADu-pot-ind3} hold for every $i=0, 1, \dots, k$.
Applying \eqref{ADu-pot-1} and then \eqref{Riesz-sum}--\eqref{ADu-pot-ind3}, we have
\begin{align*}
& \sum_{i=0}^{k+1} \mean{\bk{i}} |A(Du) - (A(Du))_{\bk{i}}| \dx \notag \\
& \leq 2\mean{\bk{0}} |A(Du) - (A(Du))_{\bk{0}} | \dx + 2\tau^{9n} \sum_{i=0}^{k} r_{i}^{3 \ve_{1}/4} \mean{\bk{i}} |A(Du)| \dx \notag \\
& \qquad + 2\tau^{9n} \sum_{i=0}^{k} r_{i}^{\sigma + \ve_{1}} \left[ \frac{1}{r_{i}^{p'}} \tail (u-(u)_{\bk{i}} ; r_{i}) \right]^{p-1} + 2c_{5} \tau^{-\eta} \sum_{i=0}^{k} \left[ \frac{|\mu|(\bk{i})}{ r_{i}^{n-1}} \right] \notag \\
& \leq 4 \tau^{8n} \lambda_{A}^{p-1} + 2\tau^{9n}\frac{r_{0}^{3\ve_{1}/4}}{1-\tau^{3 \ve_{1}/4}} \lambda_{A}^{p-1} + 2 \tau^{9n} \lambda_{A}^{p-1}  + \tau^{7n} \lambda_{A}^{p-1} 
\overset{\eqref{sec6-tau}_{1,2}}{\leq} \tau^{4n} \lambda_{A}^{p-1}\,.
\end{align*}
On the other hand, \eqref{ADu-pot-2} and \eqref{ADu-pot-ind1} yield
\begin{align*}
\mean{\bk{k+1}} |A(Du)| \dx
& \leq \tau^{-n} \sum_{i=0}^{k} \mean{\bk{i}}|A(Du)-(A(Du))_{\bk{i}}| \dx + \mean{\bk{0}} |A(Du)| \dx \notag \notag \\
& \leq \tau^{3n} \lambda_{A}^{p-1} + \tau^{8n} \lambda_{A}^{p-1} \leq \lambda_{A}^{p-1}\,.
\end{align*}
Recalling \eqref{u-tail-2}, we use $\eqref{sec6-tau}_{2,3}$, \eqref{ADu-pot-2}, \eqref{ADu-pot-ind2}, \eqref{ADu-pot-ind3}, and \eqref{Riesz-sum} in order to estimate 
\begin{align*}
\lefteqn{ \sum_{i=0}^{k+1} r_{i}^{\sigma} \left[ \frac{1}{r_{i}^{p'}} \tail(u-(u)_{B^{i}} ; r_{i}) \right]^{p-1} } \notag \\
& \leq r_{0}^{\sigma} \left[ \frac{1}{r_{0}^{p'}} \tail(u-(u)_{B^{0}} ; r_{0}) \right]^{p-1} + c_{8} \sum_{i=0}^{k} \left[ 1 + r_{i}^{(1-s)p} \tau^{-(1-\alpha)(p-1) - 1} \right] r_{i+1}^{\sigma} \left[ \frac{1}{r_{i}^{p'}} \tail(u-(u)_{B^{i}} ; r_{i}) \right]^{p-1} \notag \\
& \quad + c_{8} \sum_{i=0}^{k} \left[ \tau^{{\sigma}-n-p} r_{i}^{\sigma + (1-s)p-1} \right] \left( \mean{\bk{i}}|A(Du)| \dx + \left[ \frac{|\mu|(\bk{i})}{r_{i}^{n-1}} \right] \right) \notag \\
& \leq r_{0}^{\sigma} \left[ \frac{1}{r_{0}^{p'}} \tail(u-(u)_{B^{0}} ; r_{0}) \right]^{p-1} + 2c_{8}\tau^{\sigma} \sum_{i=0}^{k} r_{i}^{\sigma} \left[ \frac{1}{r_{i}^{p'}} \tail(u-(u)_{B^{i}} ; r_{i}) \right]^{p-1}\\
& \quad + c_{8}\tau^{-n-p} \sum_{i=0}^{k} r_{i}^{\sigma + (1-s)p - 1}\mean{\bk{i}}|A(Du)| \dx + c_{8}\tau^{-n-p}r_{0}^{\sigma+(1-s)p-1} \sum_{i=0}^{k} \left[ \frac{|\mu|(\bk{i})}{r_{i}^{n-1}} \right] \notag \\
& \leq \tau^{8n} \lambda_{A}^{p-1}
+ 2c_{8}\tau^{\sigma}\lambda_{A}^{p-1}
+ c_{8}\tau^{-n-p}\frac{r_{0}^{\sigma + (1-s)p-1}}{1-\tau^{\sigma + (1-s)p-1}}\lambda_{A}^{p-1}
+ c_{8}\tau^{-n-p}r_{0}^{\sigma + (1-s)p-1} \lambda_{A}^{p-1} 
\overset{\eqref{sec6-tau}_{4}}{\leq} \lambda_{A}^{p-1}\,.
\end{align*}
This completes the proof of the claim. 
In particular, the uniform bound \eqref{ADu-pot-ind1} implies
\[ \lim_{i\to\infty}\mean{B^{i}}|A(Du)-(A(Du))_{B^{i}}|\dx = 0, \]
and then an easy manipulation shows that $A(Du)$ is VMO at $x_{0}$. 
    
Our next goal is to establish that $x_0$ is a Lebesgue point of $A(Du)$.
Take any $k_1, k_2 \in \N$ with $k_1 < k_2$.
Then \eqref{ADu-pot-1} along with \eqref{ADu-pot-ind3} leads to
\begin{align}\label{ADu-pot-final}
|(A(Du))_{B^{k_{1}}} - (A(Du))_{B^{k_{2}}}| & \le \sum_{i=k_{1}}^{k_{2}-1}\left|(A(Du))_{B^{i}}-(A(Du))_{B^{i+1}}\right| 
 \le \sum_{i=k_{1}}^{k_{2}-1}\tau^{-n}\mean{B^{i}}|A(Du)-(A(Du))_{B^{i}}|\dx  \nonumber \\
& \le c\mean{B^{k_{1}}}|A(Du)-(A(Du))_{B^{k_{1}}}|\dx + c\lambda_{A}^{p-1}\sum_{i=k_{1}}^{k_{2}-1}r_{i}^{3\varepsilon_{1}/4} \nonumber \\
& \quad  + cr_{0}^{\ve_{1}}\sum_{i=k_{1}}^{k_{2}-1}r_{i}^{\sigma}\left[\frac{1}{r_{i}^{p'}}\tail(u-(u)_{B^{i}};r_{i})\right]^{p-1} + c\sum_{i=k_{1}}^{k_{2}-1}\left[\frac{|\mu|(B^{i})}{r_{i}^{n-1}}\right]\,.
\end{align}
Due to \eqref{ADu-pot-ind2} and the facts that $\cI^{|\mu|}_{1}(x_{0},2r_{0})<\infty$ and $A(Du)$ is VMO at $x_{0}$, $\{ (A(Du))_{\bk{i}} \}$ is a Cauchy sequence, and then an elementary manipulation shows that $x_0$ is a Lebesgue point of $A(Du)$. 
Moreover, letting $\{\xi_{i}\} \subset\rn$ be the sequence of vectors satisfying $A(\xi_{i}) = (A(Du))_{B^{i}}$, the continuity of $A^{-1}$ implies that $\lim_{i\rightarrow\infty}\xi_{i}$ exists. 
We also have 
\begin{align*}
|\xi_{i} - (Du)_{B^{i}}| & \le \mean{B^{i}}|Du-\xi_{i}|\dx + \mean{B^{i}}|Du-(Du)_{B^{i}}|\dx \\
& \le c\mean{B^{i}}|Du-\xi_{i}|\dx \le c\left(\mean{B^{i}}|A(Du)-(A(Du))_{B^{i}}|\dx\right)^{1/(p-1)} \xrightarrow[i \rightarrow \infty]{} 0\,,
\end{align*}
thereby obtaining the existence of $\lim_{i\to\infty}(Du)_{B^{i}}$. 
This, after a standard manipulation, shows that $x_{0}$ is also a Lebesgue point of $Du$, and the continuity of $A$ implies
\[ A(Du(x_{0})) = A\left(\lim_{i\to\infty}(Du)_{B^{i}}\right) = A\left(\lim_{i\to\infty}\xi_{i}\right) = \lim_{i\to\infty}A(\xi_{i}) = A(Du)(x_{0}). \]

To show the pointwise estimate \eqref{theo:main-est1}, 
we let $k_{1} = 0$ and $k_{2}\rightarrow\infty$ in \eqref{ADu-pot-final}, which gives
\begin{align*} 
\left|A(Du(x_{0}))-(A(Du))_{B^{0}}\right| \le c\mean{B^{0}}|A(Du)-(A(Du))_{B^{0}}|\dx + cr_{0}^{\kappa}\lambda_{A}^{p-1} + \cI^{|\mu|}_{1}(x_{0},2r_{0})\,,
\end{align*}
where both $c \ge 1$ and $\kappa \in (0,1)$ are depending only on $\data$ and $\sigma$. 
Recalling the definition of $\lambda_{A}$ in \eqref{ADu-pot-2}, we complete the proof of Theorem \ref{theo:main} in the case $p \geq 2$.
\end{proof}

\subsection{The case $2-1/n < p < 2$}
Throughout this section, we assume that $2-1/n < p < 2$. 
\begin{proof}[Proof of Theorem~\ref{theo:main} {\it (ii)}] 
Let $\sigma \in (0,1)$ and, as in the proof of Theorem~\ref{theo:main}~{\it (i)} without loss of generality, we assume \eqref{sigma.add}.  Then we choose 
\[ m = \frac{1-\sigma}{4}\min\left\{\frac{1}{2-p},\frac{1}{p-1},1-s\right\}\qquad\text{ and } \qquad \varepsilon_{1} = \frac{m}{2} \]
so that both are depending only on $\data$ and $\sigma$. Recall that for $p<2$, due to their definitions in~\eqref{bar a def}, we have $\bar a_1=m$ and $\bar a_2=(1-m(2-p))/(p-1)$. In particular 
\[\frac{\sigma}{p-1} + \varepsilon_{1}< \frac{1-m(2-p)}{p-1} - \varepsilon_{1} = \bar a_2 - \varepsilon_{1}  \qquad \text{ and } \qquad \varepsilon_{1} < m=\bar a_1. \]
We choose $\tau = \tau(\data,\sigma)$ so small that
\begin{equation}\label{Du-pot-tau}
2c_{6} \tau^{\beta_{v}} \le 1, \quad  2 \tau^{\ve_{1}} \le 1, \quad 2^{1/(p-1)}c_{8}\tau^{\sigma/(p-1)} \le 1/4, \quad 2\tau^{[\sigma+(1-s)p-1]/(p-1)} \le 1\,.
\end{equation}
Then, recalling $R_{0} = R_{0}(\data,\sigma) \in (0,1)$ determined in Theorem~\ref{theo:homo-mixed-hol}, we choose $\bar{R} = \bar{R}(\data,\sigma) \in (0, R_{0})$ so small that
\begin{equation}\label{R.choice.le2}
\bar{R}^{\varepsilon_{1}/4} \le \tau^{9n}, \quad \bar{R}^{[1-m(2-p)]/(p-1)-\varepsilon_{1}} \le \tau^{9n}\bar{R}^{\sigma/(p-1)+\varepsilon_{1}}, \quad 2c_{8}\left[\tau^{-(n+p)}\bar{R}^{\sigma+(1-s)p-1}\right]^{1/(p-1)} \le 1/4\,.
\end{equation}
In the following, with $r \le \bar{R}/2$ being any radius such that $B_{2r}(x_{0}) \subset \Omega$, we consider a sequence of concentric balls
\[ B^i \coloneqq B_{r_i}(x_0), \quad r_i \coloneqq \tau^i r \qquad \text{for every }\ i \in \N \cup \{0\}. \]

Applying \eqref{Duexcess} to the concentric balls $B^{i+1} \subset B^{i}$, and then using \eqref{Du-pot-tau} and \eqref{R.choice.le2}, we have
\begin{align}\label{Du-pot-1}
 \mean{B^{i+1}}|Du-(Du)_{B^{i+1}}|\dx 
& \le \frac{1}{2}\mean{B^{i}}|Du-(Du)_{B^{i}}|\dx + \tau^{9n}r_{i}^{3\varepsilon_{1}/4}\mean{B^{i}}|Du|\dx \nonumber \\
& \quad + \tau^{9n}r_{i}^{\sigma/(p-1) + \varepsilon_{1}}\left[\frac{1}{r_{i}^{p'}}\tail(u-(u)_{B^{i}};r_{i})\right] \nonumber \\
& \quad + c_{7}\tau^{-n}\left[\frac{|\mu|(B^{i})}{r_{i}^{n-1}}\right]^{1/(p-1)} + c_{7}\tau^{-n}\left[\frac{|\mu|(B^{i})}{r_{i}^{n-1}}\right]\left(\mean{B^{i}}|Du|\dx\right)^{2-p}\,.
\end{align}
Take any numbers $k_{1} < k_{2}$. Summing up \eqref{Du-pot-1} over $i \in \{k_{1},\ldots,k_{2}-1\}$, we have
\begin{align*}
\sum_{i=k_{1}+1}^{k_{2}}\mean{B^{i}}|Du-(Du)_{B^{i}}|\dx 
& \le \frac{1}{2}\sum_{i=k_{1}}^{k_{2}-1}\mean{B^{i}}|Du-(Du)_{B^{i}}|\dx + \tau^{9n}\sum_{i=k_{1}}^{k_{2}-1}r_{i}^{3\varepsilon_{1}/4}\mean{B^{i}}|Du|\dx \\
& \quad + \tau^{9n}\sum_{i=k_{1}}^{k_{2}-1}r_{i}^{\sigma/(p-1)+\varepsilon_{1}}\left[\frac{1}{r_{i}^{p'}}\tail(u-(u)_{B^{i}};r_{i})\right] \\
& \quad + c_{7}\tau^{-n}\sum_{i=k_{1}}^{k_{2}-1}\left[\frac{|\mu|(B^{i})}{r_{i}^{n-1}}\right]^{1/(p-1)} + c_{7}\tau^{-n}\sum_{i=k_{1}}^{k_{2}-1}\left[\frac{|\mu|(B^{i})}{r_{i}^{n-1}}\right]\left(\mean{B^{i}}|Du|\dx\right)^{2-p}
\end{align*}
and therefore
\begin{align}\label{Du-pot-2}
\sum_{i=k_{1}}^{k_{2}}\mean{B^{i}}|Du-(Du)_{B^{i}}|\dx & \le 2\mean{B^{k_{1}}}|Du-(Du)_{B^{k_{1}}}|\dx + 2\tau^{9n}\sum_{i=k_{1}}^{k_{2}-1}r_{i}^{3\varepsilon_{1}/4}\mean{B^{i}}|Du|\dx \nonumber \\
& \quad +2\tau^{9n}\sum_{i=k_{1}}^{k_{2}-1} r_{i}^{\sigma/(p-1)+\varepsilon_{1}}\left[\frac{1}{r_{i}^{p'}}\tail(u-(u)_{B^{i}};r_{i})\right] \nonumber \\
& \quad + 2c_{7}\tau^{-n}\sum_{i=k_{1}}^{k_{2}-1}\left[\frac{|\mu|(B^{i})}{r_{i}^{n-1}}\right]^{1/(p-1)} + 2c_{7}\tau^{-n}\sum_{i=k_{1}}^{k_{2}-1}\left[\frac{|\mu|(B^{i})}{r_{i}^{n-1}}\right]\left(\mean{B^{i}}|Du|\dx\right)^{2-p}\,.
\end{align}

We now set
\begin{align}\label{lambda.u} 
\lambda_{u} & \coloneqq \tau^{-8n}\mean{B^{0}}|Du|\dx + \tau^{-8n}r_{0}^{\sigma/(p-1)}\left[\frac{1}{r_{0}^{p'}}\tail(u-(u)_{B^{0}};r_{0})\right] \nonumber \\
& \qquad + (2c_{7}\tau^{-8n})^{1/(p-1)}\left[\cI^{|\mu|}_{1}(x_{0},2r_{0})\right]^{1/(p-1)}
\end{align}
and recall \eqref{Riesz-sum}, which together with the fact that $1/(p-1) > 1$ implies
\begin{equation}\label{Riesz-sum2} 
\sum_{i=0}^{\infty} \left[\frac{|\mu|(B^{i})}{r_{i}^{n-1}}\right]^{1/(p-1)} \le \left[\sum_{i=0}^{\infty}\frac{|\mu|(B^{i})}{r_{i}^{n-1}}\right]^{1/(p-1)} \le \tau^{-n/(p-1)}\left[\cI^{|\mu|}_{1}(x_{0},2r_{0})\right]^{1/(p-1)} \le \left(\frac{\tau^{7n}}{2c_{7}}\right)^{1/(p-1)}\lambda_{u}\,.
\end{equation}
We claim that
\begin{equation}\label{Du-pot-ind1}
\sum_{i=0}^{k}\mean{B^{i}}|Du-(Du)_{B^{i}}|\dx \le \tau^{4n}\lambda_{u}\,,
\end{equation}
\begin{equation}\label{Du-pot-ind2}
\sum_{i=0}^{k}r_{i}^{\sigma/(p-1)}\left[\frac{1}{r_{i}^{p'}}\tail(u-(u)_{B^{i}};r_{i})\right] \le \lambda_{u}\,,
\end{equation}
and
\begin{equation}\label{Du-pot-ind3}
\mean{B^{i}}|Du|\dx \le \lambda_{u}
\end{equation}
hold for every $k \in \mathbb{N}\cup\{0\}$.
The proof again goes by the induction. It is straightforward to see that \eqref{Du-pot-ind1}--\eqref{Du-pot-ind3} hold for $i=0$. We now assume that they hold for every $i = 0, \ldots ,  k$, and then show that they hold for $i=k+1$. 
We first apply \eqref{Du-pot-2} with $k_{1}=0$ and $k_{2}=k+1$. Then, by the induction hypothesis along with $\eqref{Du-pot-tau}_{1}$, \eqref{lambda.u}, and \eqref{Riesz-sum2}, we deduce that
\begin{align*}
\lefteqn{ \sum_{i=0}^{k+1}\mean{B^{i}}|Du-(Du)_{B^{i}}|\dx }\\
& \le 2\mean{B^{0}}|Du-(Du)_{B^{i}}|\dx + 2\tau^{9n}\sum_{i=0}^{k}r_{i}^{3\varepsilon_{1}/4}\mean{B^{i}}|Du|\dx + 2\tau^{9n}\sum_{i=0}^{k}r_{i}^{\sigma/(p-1)+\varepsilon_{1}}\left[\frac{1}{r_{i}^{p'}}\tail(u-(u)_{B^{i}};r_{i})\right] \\
& \quad + 2c_{7}\tau^{-n}\sum_{i=0}^{k}\left[\frac{|\mu|(B^{i})}{r_{i}^{n-1}}\right]^{1/(p-1)} + 2c_{7}\tau^{-n}\sum_{i=0}^{k}\left[\frac{|\mu|(B^{i})}{r_{i}^{n-1}}\right]\left(\mean{B^{i}}|Du|\dx\right)^{2-p} \\
& \le 2\tau^{8n}\lambda_{u} + \frac{2\tau^{9n}r_{0}^{3\varepsilon_{1}/4}}{1-\tau^{3\varepsilon_{1}/4}}\lambda_{u} + 2\tau^{9n}\lambda_{u} + 2c_{7}\tau^{-n-n/(p-1)}\left[\cI^{|\mu|}_{1}(x_{0},2r_{0})\right]^{1/(p-1)} + 2c_{7}\tau^{-2n}\lambda_{u}^{2-p}\cI^{|\mu|}_{1}(x_{0},2r_{0}) \\
& \le \tau^{4n}\lambda_{u}\,.
\end{align*}
Next, a direct calculation shows 
\[
\mean{B^{k+1}}|Du|\dx \le \tau^{-n}\sum_{i=0}^{k}\mean{B^{i}}|Du-(Du)_{B^{i}}|\dx + \mean{B^{0}}|Du|\dx 
\le \tau^{3n}\lambda_{u} + \tau^{8n}\lambda_{u} \le \lambda_{u}\,.
\]
As for the tail, Lemma~\ref{lem:u-tail}, along with \eqref{Du-pot-ind2} and \eqref{Du-pot-ind3}, yields
\begin{align*}
\lefteqn{ \sum_{i=0}^{k+1} r_{i}^{\sigma/(p-1)} \left[ \frac{1}{r_{i}^{p'}} \tail(u-(u)_{B^{i}} ; r_{i}) \right] } \notag \\
& \leq r_{0}^{\sigma/(p-1)} \left[ \frac{1}{r_{0}^{p'}} \tail(u-(u)_{B^{0}} ; r_{0}) \right] \\
& \quad + c_{8} \sum_{i=0}^{k} \left[ 1 + r_{i}^{(1-s)p} \tau^{-(1-\alpha)(p-1) - 1} \right]^{1/(p-1)} r_{i+1}^{\sigma/(p-1)} \left[ \frac{1}{r_{i}^{p'}} \tail(u-(u)_{B^{i}} ; r_{i}) \right] \notag \\
& \quad + c_{8} \sum_{i=0}^{k} \left[ \tau^{-n-p} r_{i}^{\sigma + (1-s)p-1} \right]^{1/(p-1)} \left( \mean{\bk{i}}|Du| \dx + \left[ \frac{|\mu|(\bk{i})}{r_{i}^{n-1}} \right]^{1/(p-1)} \right) \notag \\
\overset{\eqref{R.choice.le2}_{3}}& {\leq} r_{0}^{\sigma/(p-1)} \left[ \frac{1}{r_{0}^{p'}} \tail(u-(u)_{B^{0}} ; r_{0}) \right] + 2^{1/(p-1)} c_{8}\tau^{\sigma/(p-1)}  \sum_{i=0}^{k} r_{i}^{\sigma/(p-1)} \left[ \frac{1}{r_{i}^{p'}} \tail(u-(u)_{B^{i}} ; r_{i}) \right] \\
& \quad + c_{8}\tau^{-(n+p)/(p-1)}\sum_{i=0}^{k} r_{i}^{[\sigma + (1-s)p - 1]/(p-1)}\mean{\bk{i}}|Du| \dx + \frac{1}{8}\sum_{i=0}^{k} \left[ \frac{|\mu|(\bk{i})}{r_{i}^{n-1}} \right]^{1/(p-1)} \notag \\
\overset{\eqref{R.choice.le2}_{3},\eqref{lambda.u},\eqref{Riesz-sum2}}& {\leq} \lambda_{u}\,.
\end{align*}
Consequently, we have proved that \eqref{Du-pot-ind1}--\eqref{Du-pot-ind3} hold for every $k \in \mathbb{N}\cup\{0\}$.

From the uniform bound \eqref{Du-pot-ind1}, 
it follows that $Du$ is VMO at $x_{0}$. 
We now proceed as
\begin{align*}
\lefteqn{ | (Du)_{B^{k_{1}}} - (Du)_{B^{k_{2}}} | \le \sum_{i=k_{1}}^{k_{2}-1}| (Du)_{B^{i}} -(Du)_{B^{i+1}}| 
\le \sum_{i=k_{1}}^{k_{2}-1} \tau^{-n}\mean{B^{i}}|Du-(Du)_{B^{i}}|\dx } \nonumber \\
\overset{\eqref{Du-pot-2}}&{\le} c\mean{B^{k_{1}}}|Du-(Du)_{B^{k_{1}}}|\dx + c\lambda_{u}\sum_{i=k_{1}}^{k_{2}-2}r_{i}^{3\ve_{1}/4}  + c\sum_{i=k_{1}}^{k_{2}-2}r_{i}^{\sigma/(p-1)+\varepsilon_{1}}\left[\frac{1}{r_{i}^{p'}}\tail(u-(u)_{B^{i}};r_{i})\right] \nonumber \\
& \quad + c\sum_{i=k_{1}}^{k_{2}-2}\left[\frac{|\mu|(B^{i})}{r_{i}^{n-1}}\right]^{1/(p-1)} + c\sum_{i=k_{1}}^{k_{2}-2}\left[\frac{|\mu|(B^{i})}{r_{i}^{n-1}}\right]\left(\mean{B^{i}}|Du|\dx\right)^{2-p}\,.
\end{align*}
Then the fact that $Du$ is VMO at $x_{0}$, \eqref{Du-pot-ind2}, \eqref{Du-pot-ind3} and the assumption of the theorem imply that $\{(Du)_{B^{i}}\}$ is a Cauchy sequence. Hence, $x_{0}$ is a Lebesgue point of $Du$.

Finally, we let $k_{1}=0$ and $k_{2} \rightarrow \infty$ in the last display. Then, recalling the definition of $\lambda_{u}$ given in \eqref{lambda.u}, we conclude with \eqref{theo:main-est2}:
\begin{align*}
|Du(x_{0})-(Du)_{B^{0}}|
& \le c\mean{B^{0}}|Du-(Du)_{B^{0}}|\dx + cr_{0}^{3 \ve_{1}/4}\lambda_{u} + c\left[\cI^{|\mu|}_{1}(x_{0},2r_{0})\right]^{1/(p-1)} + c\lambda_{u}^{2-p}\cI^{|\mu|}_{1}(x_{0},2r_{0}) \\
& \le c\mean{B^{0}}|Du-(Du)_{B^{0}}|\dx + cr_{0}^{3 \ve_{1}/4}\mean{B^{0}}|Du|\dx + r_{0}^{\sigma/(p-1)}\left[\frac{1}{r_{0}^{p'}}\tail(u-(u)_{B^{0}};,r_{0})\right] \\
& \quad + c\left[\cI^{|\mu|}_{1}(x_{0},2r_{0})\right]^{1/(p-1)} + c\left[\cI^{|\mu|}_{1}(x_{0},2r_{0})\right]\left(\mean{B^{0}}|Du|\dx\right)^{2-p}\,. \qedhere
\end{align*}
\end{proof}

\begin{rem}
\rm
In our analysis, we need to control the summation of nonlocal tail terms over a sequence of concentric balls to show that $A(Du)$ or $Du$ is VMO at $x_{0}$. 
Moreover, in the proof of \eqref{ADu-pot-ind2} and \eqref{Du-pot-ind2}, we already have used $\cI^{|\mu|}_{1}(x_{0},2r)$ on the right-hand side of the inequality. 
However, in the local case, one can show a similar VMO-regularity under the weaker condition
\[ \lim_{\rho\rightarrow0}\frac{|\mu|(B_{\rho}(x_{0}))}{\rho^{n-1}} = 0 \]
instead of the $\cI^{|\mu|}_{1}(x_{0},2r)<\infty$. See for instance \cite[Theorem~19]{KuMiguide}.
\end{rem}

{\subsection{Proof of the continuity criterion}
\begin{proof}[Proof of Corollary~\ref{coro:cont}]
We start by choosing $\sigma = \sigma(s,p) \in (0,1)$ close enough to 1 in order to satisfy $\sigma+p-1-sp > 0$. 
We first consider the case $2-1/n < p < 2$, and aim to show that for every $\delta>0$ and $x_{0} \in B$, there exists a positive radius $r_{\delta} < \mathrm{dist}(B,\partial\Omega)$ such that
\begin{equation}\label{C1} 
\osc_{B_{r_{\delta}}(x_{0})}Du < \delta.
\end{equation}
Let $t$ be so small that $B_{2t}(x_{0}) \subset B$. We take any point $x_{1} \in B_{t}(x_{0})$.  Then $B_{t}(x_{1}) \subset B_{2t}(x_{0})$. Note that Theorem~\ref{theo:main} and the assumption of the corollary imply that both $x_{0}$ and $x_{1}$ are Lebesgue points of $Du$. Also, a straightforward calculation gives that
\begin{align*}
\mean{B_{t}(x_{1})}|Du-(Du)_{B_{t}(x_{1})}|\dx \le c\mean{B_{2t}(x_{0})}|Du-(Du)_{B_{2t}(x_{0})}|\dx
\end{align*}
and
\begin{align*}
\lefteqn{ \int_{\mathbb{R}^{n}\setminus B_{t}(x_{1})}\frac{|u(y)-(u)_{B_{t}(x_{1})}|^{p-1}}{|y-x_{1}|^{n+sp}}\dy }\\
& \le c\int_{\mathbb{R}^{n}\setminus B_{2t}(x_{0})}\frac{|u(y)-(u)_{B_{2t}(x_{0})}|^{p-1}}{|y-x_{0}|^{n+sp}}\dy + ct^{p-1-sp}\left(\mean{B_{2t}(x_{0})}|Du|\dx\right)^{p-1}.
\end{align*}

We apply \eqref{theo:main-est2} to the balls $B_{2t}(x_{0})$ and $B_{t}(x_{1})$. Then, taking into account the above observations, we have
\begin{align}\label{Du-cont}
|Du(x_{1})-Du(x_{0})| 
& \le |Du(x_{0})-(Du)_{B_{2t}(x_{0})}| + |Du(x_{1})-(Du)_{B_{t}(x_{1})}| + |(Du)_{B_{2t}(x_{0})} - (Du)_{B_{t}(x_{1})}| \notag \\
& \le c\left[\cI^{|\mu|}_{1}(x_{0},2t)\right]^{1/(p-1)} + c\left[\cI^{|\mu|}_{1}(x_{0},2t)\right]\left(\mean{B_{2t}(x_{0})}|Du|\dx\right)^{2-p} \notag \\
& \quad + c\left[\cI^{|\mu|}_{1}(x_{1},t)\right]^{1/(p-1)} + c\left[\cI^{|\mu|}_{1}(x_{1},t)\right]\left(\mean{B_{2t}(x_{0})}|Du|\dx\right)^{2-p} \notag \\
& \quad + c\mean{B_{2t}(x_{0})}|Du-(Du)_{B_{2t}(x_{0})}|\dx + ct^{\tilde{\kappa}}\mean{B_{2t}(x_{0})}|Du|\dx \notag \\
& \quad + c(2t)^{\sigma/(p-1)}\left[\frac{1}{(2t)^{p'}}\tail(u-(u)_{B_{2t}(x_{0})};x_{0},2t)\right]\,,
\end{align}
where $c \ge 1$ and $\tilde{\kappa} = \min\{\kappa,(\sigma+p-1-sp)/(p-1)\} \in (0,1)$ are depending only on $\data$. 
Moreover, revisiting the settings in the proof of Theorem~\ref{theo:main} and recalling \eqref{Du-pot-ind2}, we have in particular
\[ \lim_{i\to\infty}r_{i}^{\sigma/(p-1)}\left[\frac{1}{r_{i}^{p'}}\tail(u-(u)_{B^{i}};r_{i})\right] = 0 \,, \]
which, with an elementary manipulation, implies
\[ \lim_{\varrho\to0}\varrho^{\sigma/(p-1)}\left[\frac{1}{\varrho^{p'}}\tail(u-(u)_{B_{\varrho}(x_{0})};x_{0},\varrho)\right] = 0\,. \]
We also have that $Du$ is VMO at $x_{0}$. 
Consequently, for any $\delta > 0$, we can choose $t>0$ so small that the right-hand side of \eqref{Du-cont} is less than $\delta$. Taking $r_{\delta}$ to be this value of $t$, we conclude with \eqref{C1}. 

The proof for $p \ge 2$ is essentially the same as above, so we omit to present details. In this case, we can obtain similar estimates in terms of $A(Du)$, using \eqref{theo:main-est1}, \eqref{ADu-pot-ind2}, and the fact that $A(Du)$ is VMO at $x_{0}$.  
\end{proof} 
}

\section*{Acknowledgements}
I.~Chlebicka is supported by NCN grant no. 2019/34/E/ST1/00120. K.~Song is supported by a KIAS individual grant (MG091701) at the Korea Institute for Advanced Study. Y.~Youn is supported by the National Research Foundation, NRF-2020R1C1C1A01009760. A.~Zatorska-Goldstein is supported by NCN grant no. 2019/33/B/ST1/00535.

\section*{Declarations}

\subsection*{Conflict of interest} The authors declare that they have no conflict of interest.

\subsection*{Data Availability}
Data sharing not applicable to this article as no datasets were generated or analysed during
the current study.


\begin{thebibliography}{10}

\bibitem{AvKuMi}
B.~Avelin, T.~Kuusi, and G.~Mingione.
\newblock Nonlinear {C}alder\'{o}n-{Z}ygmund theory in the limiting case.
\newblock {\em Arch. Ration. Mech. Anal.}, 227(2):663--714, 2018.

\bibitem{Balci}
A.~Kh. Balci.
\newblock Nonlocal and mixed models with {L}avrentiev gap.
\newblock {\em arXiv:2312.05604}, 2023.

\bibitem{BarIm}
G.~Barles and C.~Imbert.
\newblock Second-order elliptic integro-differential equations: viscosity
  solutions' theory revisited.
\newblock {\em Ann. Inst. H. Poincar\'{e} C Anal. Non Lin\'{e}aire},
  25(3):567--585, 2008.

\bibitem{Ba15}
P.~Baroni.
\newblock Riesz potential estimates for a general class of quasilinear
  equations.
\newblock {\em Calc. Var. Partial Differential Equations}, 53(3-4):803--846,
  2015.

\bibitem{BaJDE22}
P.~Baroni.
\newblock Gradient continuity for $p(x)$-{L}aplacian systems under minimal
  conditions on the exponent.
\newblock {\em J. Differential Equations}, 367:415--450, 2023.

\bibitem{BaHa}
P.~Baroni and J.~Habermann.
\newblock Elliptic interpolation estimates for non-standard growth operators.
\newblock {\em Ann. Acad. Sci. Fenn. Math.}, 39(1):119--162, 2014.

\bibitem{BM20}
L.~Beck and G.~Mingione.
\newblock Lipschitz bounds and nonuniform ellipticity.
\newblock {\em Comm. Pure Appl. Math.}, 73(5):944--1034, 2020.

\bibitem{BeSa88}
C.~Bennett and R.~Sharpley.
\newblock {\em Interpolation of {O}perators}.
\newblock Academic Press, Inc., Boston, MA, 1988.

\bibitem{BDVV}
S.~Biagi, S.~Dipierro, E.~Valdinoci, and E.~Vecchi.
\newblock Mixed local and nonlocal elliptic operators: regularity and maximum
  principles.
\newblock {\em Comm. Partial Differential Equations}, 47(3):585--629, 2022.

\bibitem{BDVV23}
S.~Biagi, S.~Dipierro, E.~Valdinoci, and E.~Vecchi.
\newblock A {F}aber--{K}rahn inequality for mixed local and nonlocal operators.
\newblock {\em J. Anal. Math.}, 150(2):405--448, 2023.

\bibitem{BDVV21}
S.~Biagi, S.~Dipierro, E.~Valdinoci, and E.~Vecchi.
\newblock A {H}ong--{K}rahn--{S}zeg\"{o} inequality for mixed local and
  nonlocal operators.
\newblock {\em Mathematics in Engineering}, 5(1):1--25, 2023.

\bibitem{bgSOLA}
L.~Boccardo and T.~Gallou\"et.
\newblock Nonlinear elliptic and parabolic equations involving measure data.
\newblock {\em J. Funct. Anal.}, 87(1):149--169, 1989.

\bibitem{BKL}
S.-S. Byun, D.~Kumar, and H.-S. Lee.
\newblock Global gradient estimates for the mixed local and nonlocal problems
  with measurable nonlinearities.
\newblock {\em Calc. Var. Partial Differential Equations}, 63(2):Paper No. 27, 48, 2024.

\bibitem{BLS}
S.-S. Byun, H.-S. Lee, and K.~Song.
\newblock Regularity results for mixed local and nonlocal double phase
  functionals.
\newblock {\em arXiv:2301.06234}, 2023.

\bibitem{ByunSong}
S.-S. Byun and K.~Song.
\newblock Mixed local and nonlocal equations with measure data.
\newblock {\em Calc. Var. Partial Differential Equations}, 62(1):Paper No. 14,
  35, 2023.

\bibitem{BSY}
S.-S. Byun, K.~Song, and Y.~Youn.
\newblock Potential estimates for elliptic measure data problems with irregular
  obstacles.
\newblock {\em Math. Ann.}, 387(1-2):745--805, 2023.

\bibitem{BSY2}
S.-S. Byun, K.~Song, and Y.~Youn.
\newblock Singular elliptic measure data problems with irregular obstacles.
\newblock {\em arXiv:2309.09835}, 2023.

\bibitem{BY17}
S.-S. Byun and Y.~Youn.
\newblock Optimal gradient estimates via {R}iesz potentials for
  {$p(\cdot)$}-{L}aplacian type equations.
\newblock {\em Q. J. Math.}, 68(4):1071--1115, 2017.

\bibitem{BY18}
S.-S. Byun and Y.~Youn.
\newblock Riesz potential estimates for a class of double phase problems.
\newblock {\em J. Differential Equations}, 264(2):1263--1316, 2018.

\bibitem{ByYoun}
S.-S. Byun and Y.~Youn.
\newblock Potential estimates for elliptic systems with subquadratic growth.
\newblock {\em J. Math. Pures Appl. (9)}, 131:193--224, 2019.

\bibitem{CKSV}
Z.-Q. Chen, P.~Kim, R.~Song, and Z.~Vondra\v{c}ek.
\newblock Boundary {H}arnack principle for {$\Delta+\Delta^{\alpha/2}$}.
\newblock {\em Trans. Amer. Math. Soc.}, 364(8):4169--4205, 2012.

\bibitem{ChKu2010}
Z.-Q. Chen and T.~Kumagai.
\newblock A priori {H}\"{o}lder estimate, parabolic {H}arnack principle and
  heat kernel estimates for diffusions with jumps.
\newblock {\em Rev. Mat. Iberoam.}, 26(2):551--589, 2010.

\bibitem{CGZG-Wolff}
I.~Chlebicka, F.~Giannetti, and Zatorska-Goldstein A.
\newblock Wolff potentials and local behaviour of solutions to measure data
  elliptic problems with {O}rlicz growth.
\newblock {\em Adv. Calc. Var.}, https://doi.org/10.1515/acv-2023-0005
\bibitem{CKW}
I.~Chlebicka, M.~Kim, and M.~Weidner.
\newblock Gradient {R}iesz potential estimates for a general class of measure
  data quasilinear systems, arxiv:2307.15525, 2023.

\bibitem{CYZG-Wolff}
I.~Chlebicka, Y.~Youn, and A.~Zatorska-Goldstein.
\newblock Wolff potentials and measure data vectorial problems with {O}rlicz
  growth.
\newblock {\em Calc. Var. Partial Differential Equations}, 62(2):Paper No. 64,
  41, 2023.

\bibitem{DM21}
C.~De~Filippis and G.~Mingione.
\newblock Lipschitz bounds and nonautonomous integrals.
\newblock {\em Arch. Ration. Mech. Anal.}, 242(2):973--1057, 2021.

\bibitem{DFM-mixed}
C.~De~Filippis and G.~Mingione.
\newblock Gradient regularity in mixed local and nonlocal problems.
\newblock {\em Math. Ann.}, https://doi.org/10.1007/s00208-022-02512-7

\bibitem{DKP14}
A.~Di~Castro, T.~Kuusi, and G.~Palatucci.
\newblock Nonlocal {H}arnack inequalities.
\newblock {\em J. Funct. Anal.}, 267(6):1807--1836, 2014.

\bibitem{DKP16}
A.~Di~Castro, T.~Kuusi, and G.~Palatucci.
\newblock Local behavior of fractional {$p$}-minimizers.
\newblock {\em Ann. Inst. H. Poincar\'{e} Anal. Non Lin\'{e}aire},
  33(5):1279--1299, 2016.

\bibitem{DPV}
E.~Di~Nezza, G.~Palatucci, and E.~Valdinoci.
\newblock Hitchhiker's guide to the fractional {S}obolev spaces.
\newblock {\em Bull. Sci. Math.}, 136(5):521--573, 2012.

\bibitem{DiNo}
L.~Diening and S.~Nowak.
\newblock {C}alder\'on-{Z}ygmund estimates for the fractional $p$-{L}aplacian.
\newblock {\em arXiv:2303.02116}, 2023.

\bibitem{DPLV}
S.~Dipierro, E.~Proietti~Lippi, and E.~Valdinoci.
\newblock ({N}on)local logistic equations with {N}eumann conditions.
\newblock {\em Ann. Inst. H. Poincar\'{e} C Anal. Non Lin\'{e}aire},
  40(5):1093--1166, 2023.

\bibitem{diva}
S.~Dipierro and E.~Valdinoci.
\newblock Description of an ecological niche for a mixed local/nonlocal
  dispersal: an evolution equation and a new {N}eumann condition arising from
  the superposition of {B}rownian and {L}\'{e}vy processes.
\newblock {\em Phys. A}, 575:Paper No. 126052, 20, 2021.

\bibitem{DuMiJFA2010}
F.~Duzaar and G.~Mingione.
\newblock Gradient estimates via linear and nonlinear potentials.
\newblock {\em J. Funct. Anal.}, 259(11):2961--2998, 2010.

\bibitem{DuMiAJM2011}
F.~Duzaar and G.~Mingione.
\newblock Gradient estimates via non-linear potentials.
\newblock {\em Amer. J. Math.}, 133(4):1093--1149, 2011.

\bibitem{GarKin22}
P.~Garain and J.~Kinnunen.
\newblock On the regularity theory for mixed local and nonlocal quasilinear
  elliptic equations.
\newblock {\em Trans. Amer. Math. Soc.}, 375(8):5393--5423, 2022.

\bibitem{GarLin23}
P.~Garain and E.~Lindgren.
\newblock Higher {H}\"{o}lder regularity for mixed local and nonlocal
  degenerate elliptic equations.
\newblock {\em Calc. Var. Partial Differential Equations}, 62(2):Paper No. 67,
  36, 2023.

\bibitem{gauk}
P.~Garain and A.~Ukhlov.
\newblock Mixed local and nonlocal {S}obolev inequalities with extremal and
  associated quasilinear singular elliptic problems.
\newblock {\em Nonlinear Anal.}, 223:Paper No. 113022, 35, 2022.

\bibitem{G03}
E.~Giusti.
\newblock {\em Direct methods in the calculus of variations}.
\newblock World Scientific Publishing Co., Inc., River Edge, NJ, 2003.

\bibitem{KiMa94}
T.~Kilpel\"{a}inen and J.~Mal\'{y}.
\newblock The {W}iener test and potential estimates for quasilinear elliptic
  equations.
\newblock {\em Acta Math.}, 172(1):137--161, 1994.

\bibitem{KuMiJFA2012}
T.~Kuusi and G.~Mingione.
\newblock Universal potential estimates.
\newblock {\em J. Funct. Anal.}, 262(10):4205--4269, 2012.

\bibitem{KuMiARMA2013}
T.~Kuusi and G.~Mingione.
\newblock Linear potentials in nonlinear potential theory.
\newblock {\em Arch. Ration. Mech. Anal.}, 207(1):215--246, 2013.

\bibitem{KuMiguide}
T.~Kuusi and G.~Mingione.
\newblock Guide to nonlinear potential estimates.
\newblock {\em Bull. Math. Sci.}, 4(1):1--82, 2014.

\bibitem{KuMiCV14}
T.~Kuusi and G.~Mingione.
\newblock A nonlinear stein theorem.
\newblock {\em Calc. Var. Partial Differential Equations}, 51:45--86, 2014.

\bibitem{KuMiJEP2016}
T.~Kuusi and G.~Mingione.
\newblock Partial regularity and potentials.
\newblock {\em J. \'{E}c. polytech. Math.}, 3:309--363, 2016.

\bibitem{KuMi2018jems}
T.~Kuusi and G.~Mingione.
\newblock Vectorial nonlinear potential theory.
\newblock {\em J. Eur. Math. Soc. (JEMS)}, 20(4):929--1004, 2018.

\bibitem{KuMiSi}
T.~Kuusi, G.~Mingione, and Y.~Sire.
\newblock Nonlocal equations with measure data.
\newblock {\em Comm. Math. Phys.}, 337(3):1317--1368, 2015.

\bibitem{KuMiSi18}
T.~Kuusi, G.~Mingione, and Y.~Sire.
\newblock Regularity issues involving the fractional {$p$}-{L}aplacian.
\newblock In {\em Recent developments in nonlocal theory}, pages 303--334. De
  Gruyter, Berlin, 2018.

\bibitem{KuNoSi}
T.~Kuusi, S.~Nowak, and Y.~Sire.
\newblock Gradient regularity and first-order potential estimates for a class
  of nonlocal equations.
\newblock {\em arXiv:2212.01950}, 2022.

\bibitem{Min11}
G.~Mingione.
\newblock Gradient potential estimates.
\newblock {\em J. Eur. Math. Soc. (JEMS)}, 13(2):459--486, 2011.

\bibitem{Nak}
K.~Nakamura.
\newblock Harnack's estimate for a mixed local-nonlocal doubly nonlinear
  parabolic equation.
\newblock {\em Calc. Var. Partial Differential Equations}, 62(2):Paper No. 40,
  45, 2023.

\bibitem{NP23ARMA}
Q.-H. Nguyen and N.~C. Phuc.
\newblock A comparison estimate for singular {$p$}-{L}aplace equations and its
  consequences.
\newblock {\em Arch. Ration. Mech. Anal.}, 247(3):Paper No. 49, 24, 2023.

\bibitem{Trudinger-Wang}
N.~S. Trudinger and X.-J. Wang.
\newblock On the weak continuity of elliptic operators and applications to
  potential theory.
\newblock {\em Amer. J. Math.}, 124(2):369--410, 2002.

\end{thebibliography}
\end{document}